\documentclass[review,hidelinks,onefignum,onetabnum]{siamart220329}

\usepackage{amsmath}
\usepackage{bbm}
\usepackage{amsfonts}
\usepackage{graphicx}
\usepackage{epstopdf}
\usepackage{algorithmic}
\ifpdf
  \DeclareGraphicsExtensions{.eps,.pdf,.png,.jpg}
\else
  \DeclareGraphicsExtensions{.eps}
\fi

\newsiamremark{remark}{Remark}
\newsiamremark{hypothesis}{Hypothesis}
\crefname{hypothesis}{Hypothesis}{Hypotheses}
\newsiamthm{claim}{Claim}

\usepackage{amsopn}

\usepackage{enumitem}                

\usepackage{xcolor}

\usepackage[xcolor=dvipdf, authormarkup=none]{changes}

\usepackage[cmintegrals]{newpxmath}  

\usepackage{bm}                      

\usepackage{dsfont}                                              
  \newcommand{\IR}{\ensuremath\mathds{R}}                        
  \newcommand{\IP}{\ensuremath\mathds{P}}                        

\newcommand*{\setE}{\ensuremath{\mathcal{T}}}                    
\newcommand*{\Gammah}{\Gamma_h}                                  
\newcommand*{\Nst}{N_\mathrm{st}}                                
\renewcommand*{\vec}[1]{{\boldsymbol{#1}}}                       
\DeclareMathAlphabet{\mathbfsf}{\encodingdefault}{\sfdefault}{bx}{n}
\newcommand*{\vecc}[1]{\mathbfsf{#1}}                            
\newcommand*{\transpose}[1]{{#1}^\mathrm{T}}                     
\newcommand*{\normal}{\vec{n}}                                   
\newcommand*{\grad}{\vec{\nabla}}                                
\renewcommand*{\div}{\vec{\nabla}\cdot}                          
\newcommand*{\laplace}{\upDelta}                                 
\newcommand*{\strain}{{\boldsymbol{\varepsilon}}}                
\newcommand*{\llbrace}{\lbrace\hspace*{-0.18em}\vert}
\newcommand*{\rrbrace}{\vert\hspace*{-0.18em}\rbrace}
\newcommand*{\avg}[1]{\llbrace{#1}\rrbrace}                      
\newcommand*{\jump}[1]{\left\llbracket{#1}\right\rrbracket}      
\newcommand*{\abs}[1]{\ensuremath{|#1|}}                         
\newcommand*{\norm}[2]{\|#1\|_{#2}}                              
\newcommand*{\on}[2]{\left.#1\right\vert_{#2}}                   
\newcommand{\upwind}[1]{#1^{\uparrow}}                           

\newcommand*{\Rey}{\mathrm{Re}}                                  
\newcommand*{\Ca}{\mathrm{Ca}}                                   
\newcommand*{\Pe}{\mathrm{Pe}}                                   
\newcommand*{\Cn}{\mathrm{Cn}}                                   

\usepackage{longtable}
\usepackage{multirow}
\usepackage{tabularx}                                            
  \newcolumntype{R}{>{\raggedleft\arraybackslash}X}
  \newcolumntype{L}{>{\raggedright\arraybackslash}X}
  \newcolumntype{C}{>{\centering\arraybackslash}X}
\usepackage{booktabs}                                            
\usepackage{rotating}                                            

\headers{An efficient convex optimization bound-preserving limiter}{C. Liu, B. Riviere, J. Shen, and X. Zhang}

\title{ A simple and efficient convex optimization based bound-preserving high order accurate limiter for Cahn--Hilliard--Navier--Stokes system
\thanks{Submitted to the editors DATE.
\funding{X.Z. is supported by NSF DMS-2208518. J. S. is supported by NSFC 12371409. }}}

\author{Chen Liu\thanks{Department of Mathematics, Purdue University, 150 North University Street, West Lafayette,
Indiana 47907 (\email{liu3373@purdue.edu},   \email{zhan1966@purdue.edu}).}
\and Beatrice Riviere\thanks{Department of Computational Applied Mathematics and Operations Research, Rice University, 6100 Main Street, Houston, Texas 77005 (\email{riviere@rice.edu}).}
\and Jie Shen \thanks{
Eastern Institute of Technology, Ningbo, Zhejiang 315200, P.R. China (\email{jshen@eitech.edu.cn}). }
\and Xiangxiong Zhang \footnotemark[2]}
\ifpdf
\hypersetup{
  pdftitle={An efficient convex optimization bound-preserving limiter},
  pdfauthor={C. Liu, B. Riviere, J. Shen, and X. Zhang}
}
\fi

\begin{document}

\maketitle

\begin{abstract}
For time-dependent PDEs, 
the numerical schemes can be rendered bound-preserving without losing conservation and accuracy, by a post processing procedure of  
solving a constrained minimization in each time step. Such a constrained optimization can be formulated as a nonsmooth convex minimization, which can be efficiently solved by first order optimization methods, if using the optimal algorithm parameters. By analyzing the asymptotic linear convergence rate of the generalized Douglas--Rachford splitting method, optimal algorithm parameters can be approximately expressed as a simple function of  the number of out-of-bounds cells. 
We demonstrate the efficiency of this simple choice of algorithm parameters by applying such a limiter to cell averages of a discontinuous Galerkin scheme solving phase field equations for 3D demanding problems. 
Numerical tests on a sophisticated 3D Cahn--Hilliard--Navier--Stokes system indicate that the limiter is high order accurate, very efficient, and well-suited for large-scale simulations. For each time step, it takes at most $20$ iterations for the Douglas--Rachford splitting to enforce bounds and conservation up to the round-off error, for which the computational cost is at most $80N$ with $N$ being the total number of cells. 
\end{abstract}



\section{Introduction}

\subsection{Objective and motivation}
We are interested in a simple approach to enforce bound-preserving property of a high order accurate scheme for  phase field models, without destroying conservation and accuracy. 
Many numerical methods, especially high order accurate schemes, do not preserve bounds. 
For the sake of both physical meaningfulness
and robustness of numerical computation, it is critical to enforce both conservation and bounds.

\par
Bound-preserving schemes have been well studied in the literature for equations like hyperbolic and parabolic PDEs.
One popular approach of constructing a bound-preserving high order scheme was introduced   in \cite{zhang2010maximum,zhang2010positivity} for conservation laws, which can be extended to  parabolic equations \cite{sun2018discontinuous, srinivasan2018positivity} and  Navier--Stokes equations \cite{fan2022positivity, zhang2017positivity}, as well as implicit or semi-implicit time discretizations \cite{qin2018implicit,liu2023positivity}. 
However, this method, and most of other popular bound-preserving schemes for conservation laws and parabolic equations such as exponential time differencing \cite{du2021maximum}, 
are  based on the fact that   the simplest low order scheme is bound-preserving, which is  no longer true for a fourth order PDE like the Cahn--Hilliard equation, unless a very special implementation is used such as implicit treatment of a logarithmic potential \cite{CHEN2019100031}. 

A simple cut-off without enforcing conservation does not destroy accuracy but it is of little interest, because convergence might be lost due to loss of conservation.
A meaningful objective is to enforce bounds without destroying conservation. For the Cahn--Hilliard equation, an exponential function transform approach was used in \cite{huang2022bound}, with conservation achieved up to some small time error. 
If the logarithmic  energy potential  is used and treated implicitly, bounds can also be ensured \cite{CHEN2019100031}. 
A  Lagrange multiplier approach  in \cite{cheng2022new, cheng2022new-2}  provides a new interpretation for the cut-off method, and can preserve mass by solving a nonlinear algebraic  equation for the additional space independent Lagrange multiplier. Even though  the flux limiting \cite{kuzmin2009explicit, xu2014parametrized, guermond2019invariant, ern2022invariant} can be formally extended to Cahn--Hilliard equation \cite{frank2020bound, liu2022pressure}, it is not clear whether flux limiters can preserve high order accuracy for a fourth order PDE.
\textcolor{black}{Recently a bound-preserving finite volume scheme, which is first order accurate in time and second order accurate in space, has been constructed for the Cahn--Hilliard equation \cite{bailo2023unconditional}.}

In practice, the logarithmic potential  causes additional  difficulty in nonlinear system solvers in many schemes, thus the   double well polynomial potential with a degenerate mobility is often used as an easier surrogate. 
With the double well potential,  numerical schemes might violate the bounds much more since it does not enforce bounds  $\phi\in[-1, 1]$ like the log potential. 
In this paper, we will explore a simple and efficient high order accurate post processing procedure for preserving bounds and conservation up to round-off errors, such that it can be easily applied to any  numerical method solving the Cahn--Hillard equation,  especially for the polynomial potential.

\subsection{A  bound-preserving limiter via convex minimization}

Consider a scalar PDE as an example. Assume its  solution $u$ satisfies $m \leq u\leq M$ for all time and location, where $m$ and $M$ are constant bounds.  For simplicity, we only consider enforcing cell averages in a high order accurate discontinuous Galerkin (DG) scheme by the convex minimization, then using
the simple   Zhang--Shu limiter in \cite{zhang2010maximum,zhang2010positivity} to enforce bounds of point values of the DG solution.  
But this convex minimization approach can be easily extended to enforcing bounds of point values for any other numerical scheme such as finite difference and continuous finite element methods. 
\par 
Let $\bar u_i$ ($i=1,\cdots,N$) be all the DG solution cell averages at time step $n$ on a uniform mesh. Given $\vec{u}=\transpose{\begin{bmatrix} \bar{u}_1 & \bar{u}_2 &\cdots &\bar{u}_N\end{bmatrix}}\in \IR^N$, 
we would like to post process it to $\vec x=\transpose{\begin{bmatrix} x_1 &  x_2 &\cdots &  x_N\end{bmatrix}}\in \IR^N$ such that it is bound preserving $x_i\in[m, M]$, conservative $\sum_i x_i=\sum_i \bar u_i$, and accurate in the sense that $\|\vec x-\vec u\|$ should be small. Namely, we   consider minimize
$\|\vec x-\vec u\|$ under constraints $x_i\in[m, M]$
and $\sum_{i=1}^N x_i=\sum_{i=1}^N \bar u_i.$
 To change as few cell averages as possible, the convex  $\ell^1$-norm is often used to approximate the NP-hard $\ell^0$-norm. The $\ell^1$-norm is nonsmooth   without any strong convexity, thus the minimization might still be too expensive to solve. For the sake of efficiency, we propose the $\ell^2$-norm instead: \begin{equation}
\label{formulation-opt2}
    \min_{\vec x} \|\vec x-\vec u\|_2^2\quad\mathrm{s.t.}\quad x_i\in[m, M]
    ~~\text{and}~~ \sum_{i=1}^N x_i=\sum_{i=1}^N \bar u_i.
\end{equation}

Obviously, the minimizer to \eqref{formulation-opt2} is conservative and bound-preserving. The justification of accuracy is also straightforward, as long as $\vec u$ is an accurate numerical solution, which is a reasonable assumption and has been proved to hold for many DG schemes of a variety of PDEs, e.g., see \cite{liu2022convergence} for Cahn--Hilliard--Navier--Stokes (CHNS) equations.
Let $\bar{u}^\ast_i$ and $\bar{u}^0_i$ be the cell averages of the exact solution at time $t^n$ and initial condition, respectively. Then $\sum_i \bar{u}^\ast_i=\sum_i \bar{u}^0_i=\sum_i \bar{u}_i$ and $\bar{u}^\ast_i\in[m, M]$ imply that $\vec u^\ast$ is a feasible point satisfying the constraints of \eqref{formulation-opt2}. The minimizer $\vec x^\ast$ to \eqref{formulation-opt2}
then satisfies $\|\vec x^\ast-\vec u\|_2\leq \|\vec u^\ast-\vec u\|_2$, thus $\|\vec x^\ast-\vec u^\ast\|_2\leq \|\vec x^\ast-\vec u\|_2+\|\vec u-\vec u^\ast\|_2\leq 2\|\vec u^\ast-\vec u\|_2$. Therefore,  the limiter \eqref{formulation-opt2} does  not lose the order of accuracy.

\subsection{Efficient convex optimization algorithms}
The main catch of using \eqref{formulation-opt2} in a large scale computation, is  the possible huge  cost of solving \eqref{formulation-opt2}  to machine accuracy,   unless proven or shown otherwise, which is our main focus.   
It is a convention use the indicator function $\iota_\Omega(\vec x) =
\begin{cases}
0, & \vec{x}\in \Omega\\
+\infty, & \vec{x}\notin\Omega 
\end{cases}$ for any set $\Omega$,  
 to rewrite  \eqref{formulation-opt2} as:
\begin{equation}
\label{formulation-opt3}
    \min_{\vec x} \frac{\alpha}{2}\|\vec x-\vec u\|_2^2+\iota_{\Lambda_1}(\vec{x})
    +\iota_{\Lambda_2}(\vec{x}), 
\end{equation}
where $\alpha>0$ is a parameter and the sets $\Lambda_1$ and $\Lambda_2$ are
$ \Lambda_1 = \{\vec{x}:\, \sum_i x_i=\sum_i \bar u_i\}, \quad
\Lambda_2 = \{ \vec{x}:\, x_i\in[m, M]\}.$
The two indicator functions in \eqref{formulation-opt3} are convex but nonsmooth, and the $\ell^2$ term is strongly convex, thus
\eqref{formulation-opt3} has a unique minimizer $\vec x^\ast$.
Many   optimization algorithms, e.g., fast proximal gradient (FISTA) \cite{nesterov2013gradient, beck2017first} applied to \eqref{formulation-opt3}, can be proven to converge linearly. But a provable global linear rate is usually quite pessimistic, much slower than the actual convergence rate. It is possible to obtain sharp asymptotic   rate   for  methods like the generalized  Douglas--Rachford splitting solving $\ell^1$ minimization \cite{demanet2016eventual}, which can be used for designing best parameters. 
So we consider the generalized Douglas--Rachford splitting    \cite{lions1979splitting}, which  is equivalent to some other popular   methods such as PDHG \cite{chambolle2016introduction}, ADMM \cite{fortin2000augmented}, dual split Bregman \cite{goldstein2009split}, see also \cite{demanet2016eventual} and references therein for the equivalence.

\subsection{The generalized Douglas--Rachford splitting method}\label{sec:DR_splitting}
Splitting algorithms naturally arise for composite optimization of the form 
\begin{subequations}
\label{formulation-opt4}
\begin{equation}\min_{\vec{x}}f(\vec{x}) + g(\vec{x}),
\label{formulation-opt4-1}
\end{equation}
where functions $f$ and $g$ are convex and have simple subdifferentials and resolvents.
Let $ \partial{f}$ and $\partial{g}$ denote the subdifferentials of $f$ and $g$. Their resolvents are defined as
$$\mathrm{J}_{\gamma \partial{f}} = (\mathrm{I}+\gamma \partial{f})^{-1}= \mathrm{argmin}_{\vec{z}} \gamma f(\vec{z}) + \frac{1}{2}\norm{\vec{z}-\vec{x}}{2}^2,\quad \gamma>0,$$
$$\mathrm{J}_{\gamma \partial{g}} = (\mathrm{I}+\gamma \partial{g})^{-1}= \mathrm{argmin}_{\vec{z}} \gamma g(\vec{z}) + \frac{1}{2}\norm{\vec{z}-\vec{x}}{2}^2,\quad \gamma>0.$$
We rewrite \eqref{formulation-opt3} into $\min_{\vec{x}}f(\vec{x}) + g(\vec{x})$
by defining\begin{equation}\label{formulation-opt4-2}
f(\vec{x}) = \frac{\alpha}{2}\norm{\vec{x}-\vec{u}}{2}^2 + \iota_{\Lambda_1}(\vec{x})
\quad\text{and}\quad
g(\vec{x}) = \iota_{\Lambda_2}(\vec{x}),
\end{equation} 
\end{subequations}
where two sets are $\Lambda_1=\{\vec{x}:\vecc{A}\vec{x} = b\}$ and $\Lambda_2=\{\vec{x}\!:m\leq\vec{x}\leq M\}$, with $\vecc{A}=\begin{bmatrix}
    1 & \cdots & 1
\end{bmatrix}$, $b=\sum_i \bar u_i$, and $m\leq\vec{x}\leq M$ denoting entrywise inequality. 
The subdifferentials and resolvents can be explicitly given as
\begin{equation}\label{eq:subdiff_f}
\partial{f}(\vec{x}) = \alpha(\vec{x} - \vec{u}) + \mathcal{R}(\transpose{\vecc{A}}), \quad  \mathrm{J}_{\gamma \partial{f}}(\vec{x}) 
= \frac{1}{\gamma\alpha+1}\big(\vecc{A}^+(b - \vecc{A}\vec{x}) + \vec{x}\big) + \frac{\gamma\alpha}{\gamma\alpha+1}\vec{u},
\end{equation}
\begin{align}\label{eq:subdiff_g}
[\partial{g}(\vec{x})]_i =
\begin{cases}
[0,+\infty], & \text{if}~ x_i = M,\\
0,           & \text{if}~ x_i\in(m,M),\\
[-\infty,0], & \text{if}~ x_i = m.
\end{cases}
\quad  [\mathrm{J}_{\gamma \partial{g}}(\vec{x})]_i= \min{(\max{(x_i, m)},M)},
\end{align}
where $\mathcal{R}(\transpose{\vecc{A}})$ denotes the range of the matrix $\transpose{\vecc{A}}$
and $\vecc{A}^{+} = \transpose{\vecc{A}}(\vecc{A}\transpose{\vecc{A}})^{-1}$.

Define reflection operators as $\mathrm{R}_{\gamma \partial f} = 2\mathrm{J}_{\gamma \partial f} - \mathrm{I}$ and $\mathrm{R}_{\gamma \partial g} = 2\mathrm{J}_{\gamma \partial g} - \mathrm{I}$, where $\mathrm{I}$ denotes the identity operator.
The generalized Douglas--Rachford splitting for \eqref{formulation-opt4-1} can be written as: 
\begin{equation}\label{gDR}
\begin{cases}
\vec{y}^{k+1} 
= \lambda{\displaystyle\frac{\mathrm{R}_{\gamma \partial f}\mathrm{R}_{\gamma \partial g} + \mathrm{I}}{2}} \vec{y}^k + (1-\lambda) \vec{y}^k
= \lambda\mathrm{J}_{\gamma \partial f}\circ(2\mathrm{J}_{\gamma \partial g} - \mathrm{I}) \vec{y}^k + (\mathrm{I} - \lambda\mathrm{J}_{\gamma \partial g}) \vec{y}^k \\
\vec{x}^{k+1} = \mathrm{J}_{\gamma \partial g}(\vec{y}^{k+1})
\end{cases}\hspace{-0.5cm}.
\end{equation}
where $\vec{y}$ is an auxiliary variable, $\gamma>0$ is step size, and $\lambda\in(0,2)$ is a parameter.  For two convex functions $f(\vec x)$ and $g(\vec x)$, \eqref{gDR} converges for any $\gamma>0$ and any fixed  $\lambda\in(0,2)$, see \cite{lions1979splitting}. If one function is   strongly convex, then  $\lambda=2$ also converges.

\subsection{The bound-preserving post processing procedure for DG schemes}
\label{sec-intro-limiter}
At time step $n$,  let $u_i(x,y,z)$ be the DG polynomial on a uniform mesh in the $i$-th cell with cell average $\bar u_i$. We define the following bound-preserving limiter:

{\bf Step I}:  Solve \eqref{formulation-opt3} to post process the cell averages. 
Let $c=\frac{1}{\alpha\gamma+1}$, then the iteration \eqref{gDR} on \eqref{formulation-opt4} can be explicitly written as: 
\begin{subequations}
    \label{gDR-average}
\begin{equation}
\label{gDR2}
\begin{cases}
\vec x^{k} &=  \min{(\max{(\vec y^k, m)},M)}\\
\vec z^k&=2\vec x^k-\vec y^k\\ 
\vec{y}^{k+1} 
&= \lambda c (\vec z^k-\vec 1[\frac{1}{N}(\sum_iz_i^k-b)])+ \lambda (1-c)\vec u+\vec{y}^k -\lambda \vec{x}^k
\end{cases},
\end{equation}
where $\vec 1$ is the constant one vector of size $N$ and $b=\sum_i\bar u_i$ is a constant, $\lambda\in(0,2]$ is the fixed relaxation parameter.  
Each iterate $\vec x^k$ is bound-preserving but is not conservative until converging to the minimizer $\vec x^\ast$. We iterate \eqref{gDR2} until 
relative change is small enough $\norm{\vec{y}^{k+1} - \vec{y}^{k}}{2} \leq \epsilon$, to get an approximated
 minimizer $\vec x^\ast$ to \eqref{formulation-opt3}, for which the conservation would be satisfied up to round-off errors. 
We then modify DG polynomials by modifying the cell averages, i.e., shift them by a constant:
\begin{equation}
    \widetilde{u}_i(x,y,z)={u}_i(x,y,z)-\bar u_i+x^*_i,\quad i=1,\cdots, N.
\end{equation}
\end{subequations}

{\bf Step II}: Cell averages of modified DG polynomials $\widetilde{u}_i(x,y,z)$ are in the range $[m, M]$, so we can apply the simple scaling limiter by Zhang and Shu in \cite{zhang2010maximum,zhang2010positivity} to further enforce bounds at quadrature points, without losing conservation and accuracy. Let $S_i$ be the set of interested points in each cell, then the Zhang--Shu limiter for the polynomial $\widetilde{u}_i(x,y,z)$ with cell average $x^*_i\in [m,M]$ is given as
\begin{equation}
\label{zhang-shu}
\widehat{u}_i(x,y,z)=\theta(\widetilde{u}_i(x,y,z)-x^\ast_i)+x^\ast_i,\quad \theta=\min\left\{1, \frac{|m-x^\ast_i|}{|m_i-x^\ast_i|}, \frac{|M-x^\ast_i|}{|M_i-x^\ast_i|} \right\},
\end{equation}
where
$m_i=\min\limits_{(x,y,z)\in S_i} \widetilde{u}_i(x,y,z)$ and $M_i=\max\limits_{(x,y,z)\in S_i} \widetilde{u}_i(x,y,z)$.
See the appendix in \cite{zhang2017positivity} for a rigorous proof of the high order accuracy of \eqref{zhang-shu}. 

 \textcolor{black}{We emphasize that the Zhang-Shu limiter \eqref{zhang-shu} can preserve bounds or positivity provided that the cell averages are within bounds or are positive, which can be proven for DG methods coupled with the limiter \eqref{zhang-shu} for hyperbolic problems including scalar conservation laws, compressible Euler and compressible Navier-Stokes equations \cite{zhang2010maximum,zhang2010positivity, zhang2017positivity}, because DG methods with suitable numerical fluxes satisfy a weak monotonicity property for these problems \cite{zhang2017positivity}. However, such a weak monotonicity property is simply not true for high order DG schemes solving fourth order PDEs. Thus, if using only the limiter \eqref{zhang-shu}, the high order DG methods will not be bound-preserving for Cahn-Hilliard equations.  For all the numerical tests shown in this paper, DG methods with only the Zhang-Shu limiter will produce cell averages outside of the range $[-1,1]$.}

\subsection{The main results}
We will analyze asymptotic  convergence rate of iteration \eqref{gDR2} and give a sharp convergence rate formula, by which it is possible to pick up nearly optimal combination of parameters $c=\frac{1}{\alpha\gamma+1}$ and $\lambda$ to achieve fast convergence for the iteration \eqref{gDR2}. 
The   asymptotic linear convergence rate we derive for \eqref{formulation-opt3} is similar to the one for $\ell^1$ minimization in \cite{demanet2016eventual}. 
These rate formulae depend on the unknown $\vec x^\ast$, so usually it is impossible to use the formulae for tuning algorithm parameters, unless $\vec x^\ast$ can be easily estimated. 
For \eqref{formulation-opt3},  it is possible to pick up a nearly optimal combination of optimization algorithm parameters by only calculating number of bad cells $\bar u_i\notin[m, M]$, which is the first main result of this paper.

Let $\hat r$ be the
number of bad cells $\bar u_i\notin[m, M]$, and let $\hat \theta=\cos^{-1}\sqrt{\frac{\hat r}{N}}$, then our  analysis suggests the following simple choice of nearly optimal parameters:
\begin{equation}
    \begin{cases}
    c=\frac{1}{2}, \lambda = \frac{4}{2-\cos{(2\hat \theta)}}, &\quad \mbox{if } \hat \theta \in(\frac{3}{8}\pi,\frac{1}{2}\pi],\\
    c=\frac{1}{(\cos\hat\theta + \sin\hat\theta)^2}, \lambda =\frac{2}{1+\frac{1}{1+\cot\hat\theta}-\frac{1}{(\cos\hat\theta + \sin\hat\theta)^2}}, &\quad \mbox{if } \hat \theta \in(\frac{1}{4}\pi,\frac{3}{8}\pi],\\
  c=\frac{1}{(\cos\hat\theta + \sin\hat\theta)^2}, \lambda =2,  &\quad \mbox{if } \hat \theta \in(0,\frac{1}{4}\pi].
\end{cases}\label{opt-parameter}
\end{equation}
 We emphasize that both $c$ and $\lambda$ should be the constants w.r.t. 
 iteration index $k$ in \eqref{gDR2}, once they are chosen by \eqref{opt-parameter}.
Notice that $\lambda(1-c)\vec u$ is a constant for the iteration \eqref{gDR2} and each entry of $\vec z^k-\vec 1[\frac{1}{N}(\sum_iz_i^k-b)]$ can be computed by $z^k_i-[\frac{1}{N}(\sum_iz_i^k-b)]$, thus if only counting number of computing multiplications, $\min$, and $\max$, the computational complexity of each iteration in  \eqref{gDR2} is $4N$.  
By using the  parameters \eqref{opt-parameter}, it takes at most 20 iterations of \eqref{gDR2} to converge in all our numerical tests, thus the cost of iterating \eqref{gDR2} until convergence would be at most $80N$, which is highly efficient and well-suited for large-scale simulations.

The numerical observation of at most 20 iterations can also be explained by the asymptotic convergence rate analysis, which is another main result. Assuming the number of bad cells $\bar u_i\notin[m, M]$ is much smaller than the number of total cells $N$, we will show that the asymptotic convergence rate  of \eqref{gDR2} using \eqref{opt-parameter} is given by 
\begin{equation}
    -\frac{\cos{(2\theta )}}{2-\cos{(2\theta )}}\approx-\frac{\cos{(2\hat \theta)}}{2-\cos{(2\hat\theta)}}=\frac{1-2\cos{\hat\theta}^2}{3-2\cos{\hat\theta}^2}=\frac{1-2\frac{\hat r}{N}}{3-2\frac{\hat r}{N}}\approx \frac13,\quad \mbox{if } \hat r\ll N,
    \label{opt-rate}
\end{equation}
with $\theta(\vec x^*)$ being an unknown angle,  which can be approximated by $\hat \theta$.
If the ratio of bad cells is very small, \eqref{gDR2} will have a local convergence rate almost like $\|\vec y^k-\vec y^*\|\leq C \left(\frac13\right)^k$, which would take around 30 iterations to reach around 1E-15 if $C=1$. 

\subsection{Organization of the paper}
The rest of the paper is organized as follows.
In Section~\ref{sec:convergence_analysis}, we analyze the asymptotic linear convergence rate of the Douglas--Rachford splitting \eqref{gDR} and \eqref{gDR2}, and derive the parameter guideline \eqref{opt-parameter}. 
In Section~\ref{sec:application_chns}, we discuss an application of our bound-preserving limiting strategy to an important phase-field model, the CHNS system.
The numerical tests are given in Section~\ref{sec:experiments}. 
Section \ref{sec:remark} are concluding remarks.

\section{Asymptotic linear convergence rate analysis}\label{sec:convergence_analysis}
In this section, we derive the asymptotic linear convergence rate of generalized Douglas--Rachford splitting \eqref{gDR} for solving the minimization problem \eqref{formulation-opt4}.  
The discussion in this section follows closely the analysis for $\ell^1$ minimization in \cite{demanet2016eventual}. Even though $\ell^1$ minimization  is harder than $\ell^2$ minimization, the analysis for \eqref{formulation-opt4} is not necessarily a straightforward extension of those in \cite{demanet2016eventual} because 
\eqref{eq:subdiff_f} and \eqref{eq:subdiff_g} are different from operators in \cite{demanet2016eventual}. 

For convenience, let $F=\partial f$
and $G=\partial g$ denote the subdifferential operators.
Let $\mathrm{S}(\vec x)$ be the cut-off operator, i.e., $[\mathrm{J}_{\gamma G}(\vec x)]_i=[\mathrm{S}(\vec x)]_i=\min{(\max{(x_i, m)},M)}$.  

We keep the discussion a bit more general by considering a general linear constraint $\vecc A\vec x=b=\vecc A\vec u$ in the function $f(\vec x)$ in  \eqref{formulation-opt4-2}, and assume $\vecc A$ has less number of rows than the number of columns, with full row rank such that $\vecc A^+=\transpose{\vecc A}(\vecc A \transpose{\vecc A})^{-1}$ is well defined. When needed, we will plug in the special case $\vecc A= \begin{bmatrix}
    1 & 1 & \cdots & 1
\end{bmatrix}.$

\subsection{The fixed point set}\label{sec:structure_fix_p}
Let $\mathrm{P}(\vec{x}) = \vecc{A}^+(b - \vecc{A}\vec{x}) + \vec{x}$ denote the projection operator. 
Then, the resolvents can be written as $\mathrm{J}_{\gamma F}(\vec{x}) = \frac{1}{\gamma\alpha+1}\mathrm{P}(\vec{x}) + \frac{\gamma\alpha}{\gamma\alpha+1}\vec{u}$ and $\mathrm{J}_{\gamma G}(\vec{x}) = \mathrm{S}(\vec{x})$.
Let $\mathrm{T}_\gamma$ denote the iteration operator for $\vec y$ in \eqref{gDR}, then it  becomes:
\begin{equation}\label{eq:operator_T_gamma}
\mathrm{T}_\gamma 
= \frac{\lambda}{\gamma\alpha+1}\mathrm{P}\circ(2\mathrm{S}-\mathrm{I}) 
+ (\mathrm{I} - \lambda\mathrm{S})
+ \frac{\lambda\gamma\alpha}{\gamma\alpha+1}\vec{u}.
\end{equation}
The fixed point $\vec y^\ast$ of $\mathrm{T}_\gamma$ is not the minimizer of \eqref{formulation-opt4}, while $\vec x^\ast=\mathrm{J}_{\gamma G}(\vec y^\ast)=\mathrm{S}(\vec y^\ast)$ is the minimizer. 
The fixed point set of the operator $\mathrm{T}_\gamma$ has the following structure.
\begin{theorem}\label{thm:fixed_point_set}
The set of fixed point of operator $\mathrm{T}_\gamma$ is
$$\Pi = \{\vec{y}^\ast\!: \vec{y}^\ast = \vec{x}^\ast + \gamma\vec{\eta},~ \vec{\eta}\in -\partial f(\vec{x}^\ast) \cap \partial g(\vec{x}^\ast)\}.$$ 
\end{theorem}
\begin{proof}
We first show any $\vec{y}^\ast \in \Pi$ is a fixed point of the operator $\mathrm{T}_\gamma$.
$\forall\vec{\eta}\in\partial g(\vec{x}^\ast)$ in \eqref{eq:subdiff_g}, we have $\mathrm{S}(\vec{y}^\ast) = \vec{x}^\ast$, since the $i$-th entry of the vector $\vec{y}^\ast = \vec{x}^\ast + \gamma\vec{\eta}$ satisfies
\begin{equation*}
[\vec{y}^\ast]_i
\begin{cases}
\in[M,+\infty],  & \text{if}~ x^\ast_i = M,\\
=x^\ast_i,     & \text{if}~ x^\ast_i\in(m,M),\\
\in [-\infty,m], & \text{if}~ x^\ast_i = m.
\end{cases}
\end{equation*}
Thus, we have $\mathrm{P}\circ(2\mathrm{S}-\mathrm{I})\vec{y}^\ast = \mathrm{P}(2\vec x^*-\vec y^*)= \mathrm{P}(\vec x^*-\gamma \vec{\eta}) = \vec{x}^\ast - \gamma\vec{\eta} + \gamma\vecc{A}^+\vecc{A}\vec{\eta}$, where $\vecc{A}\vec x^\ast=b$ is used. 
And $\vec{\eta}\in -\partial f(\vec{x}^\ast)$ in \eqref{eq:subdiff_f} implies that there exists $\vec{\xi}$ such that $\vec{\eta} = -\alpha(\vec{x}^\ast-\vec{u}) + \transpose{\vecc{A}}\vec{\xi}$. Multiplying both sides by $\vecc{A}$, with $\vecc{A}\vec{x}^\ast =b= \vecc{A}\vec{u}$  we get $\vecc A\vec \eta=\vecc A\transpose{\vecc{A}}\vec{\xi}$, thus $\vec{\xi}=(\vecc A\transpose{\vecc{A}})^{-1}\vecc A\vec \eta$
and 
$ \gamma\vec{\eta} =- \gamma\alpha(\vec{x}^\ast-\vec{u})+\gamma\vecc{A}^+\vecc{A}\vec{\eta}$. Then, we have $\mathrm{P}\circ(2\mathrm{S}-\mathrm{I})\vec{y}^\ast = (\gamma\alpha+1)\vec{x}^\ast - \gamma\alpha\vec{u}$. 
Therefore
\begin{equation*}
\mathrm{T}_\gamma(\vec{y}^\ast) 
= \frac{\lambda}{\gamma\alpha+1} \Big((\gamma\alpha+1)\vec{x}^\ast - \gamma\alpha\vec{u}\Big)
+ \vec{y}^\ast - \lambda\vec{x}^\ast
+ \frac{\lambda\gamma\alpha}{\gamma\alpha+1}\vec{u}
= \vec{y}^\ast.
\end{equation*}
Next, we show any fixed point $\vec{y}^\ast$   belongs to set $\Pi$. Let $\vec{\eta} = (\vec{y}^\ast - \vec{x}^\ast)/\gamma$. Then, $\vec{y}^\ast$ being a fixed point   implies $\mathrm{J}_{\gamma G}(\vec{y}^\ast) = \vec{x}^\ast$.
Recall that $\mathrm{J}_{\gamma G}=\mathrm{S}$, we have
\begin{itemize}[topsep=2pt,noitemsep]
\item[\emph{i}.] if $x_i^\ast + \gamma\eta_i \geq M$, then $x_i^\ast=\mathrm{S}(x_i^\ast + \gamma\eta_i) = M$, thus $\eta_i \in [0,+\infty]$;
\item[\emph{ii}.] if $x_i^\ast + \gamma\eta_i \in (m,M)$, then $x_i^\ast=\mathrm{S}(x_i^\ast + \gamma\eta_i) = x_i^\ast + \gamma\eta_i$, thus $\eta_i = 0$;
\item[\emph{iii}.] if $x_i^\ast + \gamma\eta_i \leq m$, then $x_i^\ast=\mathrm{S}(x_i^\ast + \gamma\eta_i) = m$, thus $\eta_i \in [-\infty,0]$.
\end{itemize}
So $\vec{\eta}\in\partial{g}(\vec{x}^\ast)$.
And $\vec{y}^\ast = \mathrm{T}_\gamma(\vec{y}^\ast)$ is equivalent to $\vec{y}^\ast = \frac{\lambda}{2}(\mathrm{R}_{\gamma F}\mathrm{R}_{\gamma G} + \mathrm{I})\vec{y}^\ast + (1-\lambda)\vec{y}^\ast$, which implies $\vec{y}^\ast = \mathrm{R}_{\gamma F}\mathrm{R}_{\gamma G}(\vec{y}^\ast)$. Recall $\mathrm{J}_{\gamma G}(\vec{y}^\ast) = \vec{x}^\ast$ and $\vec{y}^\ast = \vec{x}^\ast + \gamma\vec{\eta}$, we have 
\begin{equation*}
\vec{y}^\ast
= \mathrm{R}_{\gamma F}(2\mathrm{J}_{\gamma G}(\vec{y}^\ast) - \vec{y}^\ast) 
= \mathrm{R}_{\gamma F}(\vec{x}^\ast - \gamma\vec{\eta})
= 2\mathrm{J}_{\gamma F}(\vec{x}^\ast - \gamma\vec{\eta}) - (\vec{x}^\ast - \gamma\vec{\eta}).
\end{equation*}
So $\vec{x}^\ast = \mathrm{J}_{\gamma F}(\vec{x}^\ast - \gamma\vec{\eta})$, 
which implies $\vec{x}^\ast = \mathrm{argmin}_{\vec{z}} \gamma f(\vec{z}) + \frac{1}{2}\norm{\vec{z}-(\vec{x}^\ast - \gamma\vec{\eta})}{2}^2$. By the critical point equation, we have $\vec{0}\in \gamma \partial f(\vec{x}^\ast) +  \gamma\vec{\eta}$ thus $\vec{\eta}\in-\partial{f}(\vec{x}^\ast)$.
\end{proof}
Let $\mathcal{B}_r(\vec{z})$ denote a closed ball in $\ell^2$-norm centered at $\vec{z}$ with radius $r$. Define set $\mathcal{Q}$:
\begin{equation*}
\mathcal{Q} = Q_1 \otimes Q_2 \otimes \cdots \otimes Q_n, \quad\text{where}~
Q_i =
\begin{cases}
[M,+\infty], & \text{if}~ x^\ast_i = M,\\
(m,M),       & \text{if}~ x^\ast_i \in (m,M),\\
[-\infty,m], & \text{if}~ x^\ast_i = m.
\end{cases}
\end{equation*}
For any fixed point $\vec{y}^\ast$, the Theorem~\ref{thm:fixed_point_set} implies there exists an $\vec{\eta} = \frac{1}{\gamma}(\vec{y}^\ast-\vec{x}^\ast) \in\partial g(\vec{x}^\ast)$ and by \eqref{eq:subdiff_g} we have $\vec{x}^\ast+\gamma\vec{\eta} \in \mathcal{Q}$ for any $\gamma>0$, which gives $\vec{y}^\ast\in\mathcal{Q}$. Let $\epsilon\geq0$ be the least upper bound such that $\mathcal{B}_\epsilon(\vec{y}^\ast) \subset \mathcal{Q}$.
If $\epsilon > 0$, then $\vec{y}^\ast$ is an interior fixed point and we call this the standard case; otherwise, $\vec{y}^\ast$ is a boundary fixed point and we call this the non-standard case.
In the standard case that the sequence $\vec{y}^k$ converges to an interior fixed point $\vec{y}^\ast$. There exists a large enough integer $K>0$ such that $\norm{\vec{y}^K - \vec{y}^\ast}{2}<\epsilon$ holds. For any $k\geq K$,   the operator $\mathrm{T}_\gamma$ is nonexpansive \cite{lions1979splitting}, so 
\begin{equation*}
\norm{\vec{y}^k - \vec{y}^\ast}{2} = \norm{\mathrm{T}_\gamma(\vec{y}^{k-1}) - \mathrm{T}_\gamma(\vec{y}^\ast)}{2} \leq \norm{\vec{y}^{k-1} - \vec{y}^\ast}{2} \leq \cdots \leq \norm{\vec{y}^K - \vec{y}^\ast}{2}<\epsilon.
\end{equation*}
Thus, after taking the generalized Douglas--Rachford iteration \eqref{gDR} sufficiently many times, the iterates will always belong to the ball $\mathcal{B}_\epsilon(\vec{y}^\ast)\subset \mathcal{Q}$, namely the iteration enters the asymptotic convergence regime and the cut-off location does not change.
\par
In the rest of this paper, we only focus on  the standard case. The non-standard case can be analyzed by utilizing the same technique as in \cite{demanet2016eventual}.
The non-standard case has not been observed in our numerical experiments.

\subsection{The characterization of the operator \texorpdfstring{$\mathrm{T}_{\gamma}$}{Tg2}}
Assume the unique solution $\vec{x}^\ast$ of the minimization problem \eqref{formulation-opt4} has $r$ components equal to $m$ or $M$.  
We further assume $r<N$, e.g., not all the cell averages will touch the boundary   $m$ or $M$, which is a quite reasonable assumption. 
We emphasize that $r$ is unknown, unless $\vec{x}^\ast$ is given.

Let $\vec{e}_i$ ($i=1,\cdots,N$) be the standard basis of $\IR^N$. Let $\vec{e}_j$ ($j=i_1, \cdots, i_r$) denote the basis vectors corresponding to entries $\vec{x}^\ast$  of being $m$ or $M$. 
Let $\vecc{B}$ be the corresponding $r\times N$ selector matrix, i.\,e., $\vecc{B} = \transpose{[\vec{e}_{i_1}, \cdots, \vec{e}_{i_r}]}$.

Recall that we only discuss the standard case, i.e., $\vec{y}^\ast$ is in the interior of $\mathcal{Q}$. 
Then, in the asymptotic convergence regime, i.e., after sufficiently many iterations, the iterate $\vec{y}_k$ will stay in the interior of $\mathcal{Q}$, thus
the operator $\mathrm{S}$ has an expression
\begin{equation}\label{eq:matrix_expression_S}
\mathrm{S}(\vec{y}) = \vec{y} - \vecc{B}^+\vecc{B}\vec{y} + \sum_{j\in{\{i_1,\cdots,i_r\}}} x^\ast_j \vec{e}_j.
\end{equation}
Note, the $j$-th component of $\vec{x}^\ast$, namely the $x^\ast_j$ in \eqref{eq:matrix_expression_S}, takes value $m$ or $M$ for any $j\in\{i_1, \cdots, i_r\}$.
Let $\vecc{I}_N$ denote an $N\times N$ identity matrix.
\begin{lemma}
For any $\vec{y}$ in the interior of $\mathcal{Q}$,  and a standard fixed point $\vec{y}^\ast$ in the interior of $\mathcal{Q}$, we have $\mathrm{T}_\gamma(\vec{y}) - \mathrm{T}_\gamma(\vec{y}^\ast) = \vecc{T}_{c,\lambda}(\vec{y} - \vec{y}^\ast)$, where the matrix $\vecc{T}_{c,\lambda}$ is given by
\begin{equation*}
\vecc{T}_{c,\lambda} = \lambda \Big(c(\vecc{I}_N - \vecc{A}^+\vecc{A})(\vecc{I}_N - \vecc{B}^+\vecc{B}) + c\vecc{A}^+\vecc{A} \vecc{B}^+\vecc{B} + (1-c)\vecc{B}^+\vecc{B}\Big) + (1-\lambda)\vecc{I}_N.
\end{equation*}
Here, $c = \frac{1}{\gamma\alpha+1}$ is a constant in $(0,1)$.
\end{lemma}
\begin{proof}
By \eqref{eq:matrix_expression_S},   $\mathrm{S}(\vec{y}) - \mathrm{S}(\vec{y}^\ast) = (\vecc{I}_N - \vecc{B}^+\vecc{B})(\vec{y} - \vec{y}^\ast)$.
So by \eqref{eq:operator_T_gamma}, 
\begin{align*}
\mathrm{T}_\gamma(\vec{y}) \,-&\, \mathrm{T}_\gamma(\vec{y}^\ast) 
\!=\! \frac{\lambda}{\gamma\alpha+1}\Big(\mathrm{P}(2\mathrm{S}(\vec{y}) - \vec{y}) - \mathrm{P}(2\mathrm{S}(\vec{y}^\ast) - \vec{y}^\ast)\Big)
+ (\vec{y} - \vec{y}^\ast) - \lambda(\mathrm{S}(\vec{y}) - \mathrm{S}(\vec{y}^\ast))\\
=&\, \frac{\lambda}{\gamma\alpha+1}(\vecc{I}_N - \vecc{A}^+\vecc{A})(\vecc{I}_N - 2\vecc{B}^+\vecc{B})(\vec{y} - \vec{y}^\ast)
+ (\vec{y} - \vec{y}^\ast)
- \lambda(\vecc{I}_N - \vecc{B}^+\vecc{B})(\vec{y} - \vec{y}^\ast)\\
=&\, \frac{\lambda}{\gamma\alpha+1}(\vecc{I}_N - \vecc{A}^+\vecc{A})(\vecc{I}_N - \vecc{B}^+\vecc{B})(\vec{y} - \vec{y}^\ast)
+ \frac{\lambda}{\gamma\alpha+1} \vecc{A}^+\vecc{A} \vecc{B}^+\vecc{B}(\vec{y} - \vec{y}^\ast)\\
+&\, \frac{\lambda\gamma\alpha}{\gamma\alpha+1} \vecc{B}^+\vecc{B}(\vec{y} - \vec{y}^\ast)
+ (1-\lambda)(\vec{y} - \vec{y}^\ast).
\end{align*}
Therefore, the matrix $\vecc{T}_{c,\lambda}$ can be expressed as follows:
\begin{equation*}
\vecc{T}_{c,\lambda} = \frac{\lambda}{\gamma\alpha+1}\Big((\vecc{I}_N - \vecc{A}^+\vecc{A})(\vecc{I}_N - \vecc{B}^+\vecc{B}) + \vecc{A}^+\vecc{A} \vecc{B}^+\vecc{B}\Big) + \frac{\lambda\gamma\alpha}{\gamma\alpha+1} \vecc{B}^+\vecc{B} + (1-\lambda)\vecc{I}_N.
\end{equation*} 
\end{proof}
\begin{definition}
Let $\mathcal{U}$ and $\mathcal{V}$ be two subspaces of $\IR^N$ with $\mathrm{dim}(\mathcal{U}) = p \leq \mathrm{dim}(\mathcal{V})$. The principal angles $\theta_k\in[0,\frac{\pi}{2}]$ ($k=1,\cdots,p$) between $\mathcal{U}$ and $\mathcal{V}$ are recursively defined by 
\begin{align*}
&\cos{\theta_k} = \vec{u}_k^\mathrm{T}\vec{v}_k = \max_{\vec{u}\in\mathcal{U}}\max_{\vec{v}\in\mathcal{V}}\vec{u}^\mathrm{T}\vec{v},\\
\text{such~that}~~&\norm{\vec{u}}{2} = \norm{\vec{v}}{2} = 1,~ \vec{u}_j^\mathrm{T}\vec{u} = 0,~ \vec{v}_j^\mathrm{T}\vec{v} = 0,~ j = 1,2,\cdots,k-1. 
\end{align*}
The vectors $(\vec{u}_1,\cdots,\vec{u}_p)$ and $(\vec{v}_1,\cdots,\vec{v}_p)$ are principal vectors.  
\end{definition}
Our next goal is to decompose the matrix $\vecc{T}_{c,\lambda}$ with principal angles between subspaces $\mathcal{N}(\vecc{A})$ and $\mathcal{N}(\vecc{B})$. 
To simplify the writeup, we define matrix $\vecc{T} = (\vecc{I}_N - \vecc{A}^+\vecc{A})(\vecc{I}_N - \vecc{B}^+\vecc{B}) + \vecc{A}^+\vecc{A} \vecc{B}^+\vecc{B}$. Thus, we rewrite $\vecc{T}_{c,\lambda} = \lambda(c\vecc{T} + (1-c)\vecc{B}^+\vecc{B}) + (1-\lambda)\vecc{I}_N$.
Let $\vecc{A}_0$ be an $N\times(N-1)$ matrix whose columns are orthogonal basis of $\mathcal{N}(\vecc{A})$ and $\vecc{A}_1$ be an $N\times 1$ matrix whose columns are orthogonal basis of $\mathcal{R}(\transpose{\vecc{A}})$.
Similarly, let $\vecc{B}_0$ be an $N\times(N-r)$ matrix whose columns are orthogonal basis of $\mathcal{N}(\vecc{B})$ and $\vecc{B}_1$ be an $N\times r$ matrix whose columns are orthogonal basis of $\mathcal{R}(\transpose{\vecc{B}})$.

Since both $\vecc{A}^+\vecc{A}$ and $\vecc{A}_1\transpose{\vecc{A}}_1$ represent the projection
to $\mathcal{R}(\transpose{\vecc{A}})$, we have $\vecc{A}^+\vecc{A} = \vecc{A}_1\transpose{\vecc{A}}_1$. Similarly, $\vecc{I}_N - \vecc{A}^+\vecc{A} = \vecc{A}_0\transpose{\vecc{A}}_0$. Thus we have $\vecc{T} = \vecc{A}_0\transpose{\vecc{A}}_0\vecc{B}_0\transpose{\vecc{B}}_0 + \vecc{A}_1\transpose{\vecc{A}}_1\vecc{B}_1\transpose{\vecc{B}}_1$.

Define matrix $\vecc{E}_0=\transpose{\vecc{A}}_0\vecc{B}_0$ and matrix $\vecc{E}_1=\transpose{\vecc{A}}_1\vecc{B}_0$. Since $\vecc{A}_0\transpose{\vecc{A}}_0 + \vecc{A}_1\transpose{\vecc{A}}_1 = \vecc{I}_N$, we have $\vecc{B}_0 = (\vecc{A}_0\transpose{\vecc{A}}_0 + \vecc{A}_1\transpose{\vecc{A}}_1)\vecc{B}_0 = \vecc{A}_0\vecc{E}_0 + \vecc{A}_1\vecc{E}_1$. Therefore, we rewrite
\begin{align}\label{eq:decomp_B0B0T_1}
\vecc{B}_0\transpose{\vecc{B}}_0 = (\vecc{A}_0\vecc{E}_0 + \vecc{A}_1\vecc{E}_1)(\transpose{\vecc{E}}_0\transpose{\vecc{A}}_0 + \transpose{\vecc{E}}_1\transpose{\vecc{A}}_1) 
= 
\begin{bmatrix}
\vecc{A}_0 & \vecc{A}_1
\end{bmatrix}
\begin{bmatrix}
\vecc{E}_0\transpose{\vecc{E}}_0 & \vecc{E}_0\transpose{\vecc{E}}_1 \\ 
\vecc{E}_1\transpose{\vecc{E}}_0 & \vecc{E}_1\transpose{\vecc{E}}_1
\end{bmatrix}
\begin{bmatrix}
\transpose{\vecc{A}}_0\\ 
\transpose{\vecc{A}}_1
\end{bmatrix}.
\end{align}
The singular value decomposition (SVD) of the $(N-1)\times(N-r)$ matrix $\vecc{E}_0$ is $\vecc{E}_0 = \vecc{U}_0\cos{\vecc{\Theta}}\transpose{\vecc{V}}$ with singular values $\cos{\theta_1}$, $\cdots$, $\cos{\theta_{N-r}}$ in nonincreasing order.
We know that   $\theta_i$ ($i=1,\cdots,N-r$) are the principal angles between   $\mathcal{N}(\vecc{A})$ and $\mathcal{N}(\vecc{B})$. 
\par
Notice $\transpose{\vecc{E}}_1\vecc{E}_1 = \transpose{\vecc{B}}_0\vecc{A}_1\transpose{\vecc{A}}_1\vecc{B}_0$ and $\vecc{A}_1\transpose{\vecc{A}}_1 = \vecc{I}_N - \vecc{A}_0\transpose{\vecc{A}}_0$, we have $\transpose{\vecc{E}}_1\vecc{E}_1 = \transpose{\vecc{B}}_0\vecc{B}_0 - \transpose{\vecc{B}}_0\vecc{A}_0\transpose{\vecc{A}}_0\vecc{B}_0 = \vecc{I}_{N-r}-\transpose{\vecc{E}}_0\vecc{E}_0$. Recall the SVD of $\vecc{E}_0$, we have $\transpose{\vecc{E}}_1\vecc{E}_1 = \vecc{V}\sin^2{\vecc{\Theta}}\transpose{\vecc{V}}$.  
Thus $\vecc{E}_1$ can be expressed as $\vecc{U}_1\sin{\vecc{\Theta}}\transpose{\vecc{V}}$, which is however not the SVD of $\vecc{E}_1$.
To this end, let matrix $\widetilde{\vecc{A}}=[\vecc{A}_0\vecc{U}_0 ~ \vecc{A}_1\vecc{U}_1]$, then \eqref{eq:decomp_B0B0T_1} becomes
\begin{equation}\label{eq:decomp_B0B0T}
\vecc{B}_0\transpose{\vecc{B}}_0 = 
\widetilde{\vecc{A}}
\begin{bmatrix}
\cos^2{\vecc{\Theta}} & \sin{\vecc{\Theta}}\cos{\vecc{\Theta}} \\ 
\sin{\vecc{\Theta}}\cos{\vecc{\Theta}} & \sin^2{\vecc{\Theta}}
\end{bmatrix}
\transpose{\widetilde{\vecc{A}}}.
\end{equation}
Because of $\vecc{B}_1\transpose{\vecc{B}}_1 = \vecc{I}_N-\vecc{B}_0\transpose{\vecc{B}}_0$ and $\widetilde{\vecc{A}}\transpose{\widetilde{\vecc{A}}}=\vecc{I}_N$, we have the decomposition 
\begin{equation}\label{eq:decomp_B1B1T}
\vecc{B}_1\transpose{\vecc{B}}_1 = 
\widetilde{\vecc{A}}
\begin{bmatrix}
\sin^2{\vecc{\Theta}} & -\sin{\vecc{\Theta}}\cos{\vecc{\Theta}} \\ 
-\sin{\vecc{\Theta}}\cos{\vecc{\Theta}} & \cos^2{\vecc{\Theta}}
\end{bmatrix}
\transpose{\widetilde{\vecc{A}}}.
\end{equation}
Notice $\vecc{A}_0\transpose{\vecc{A}}_0\widetilde{\vecc{A}}=[\vecc{A}_0\vecc{U}_0~\vecc{O}_{N\times(N-r)}]$ and $\vecc{A}_1\transpose{\vecc{A}}_1\widetilde{\vecc{A}}=[\vecc{O}_{N\times(N-r)}~\vecc{A}_1\vecc{U}_1]$, by \eqref{eq:decomp_B0B0T} and \eqref{eq:decomp_B1B1T}, we obtain
\begin{equation}\label{eq:matrixT}
\vecc{T} = 
\widetilde{\vecc{A}}
\begin{bmatrix}
\cos^2{\vecc{\Theta}} & \sin{\vecc{\Theta}}\cos{\vecc{\Theta}} \\ 
-\sin{\vecc{\Theta}}\cos{\vecc{\Theta}} & \cos^2{\vecc{\Theta}}
\end{bmatrix}
\transpose{\widetilde{\vecc{A}}}.
\end{equation}
Therefore, use \eqref{eq:matrixT}   and consider $\vecc{B}^+\vecc{B} = \vecc{B}_1\transpose{\vecc{B}}_1$, the matrix $\vecc{T}_{c,\lambda}$ becomes
\begin{equation}\label{eq:mat_Tclamb}
\vecc{T}_{c,\lambda} = 
\widetilde{\vecc{A}}
\setlength{\arraycolsep}{2.25pt}\begin{bmatrix}
\lambda c\cos^2{\vecc{\Theta}} + \lambda (1-c)\sin^2{\vecc{\Theta}} + (1-\lambda)\vecc{I}_{N-r} & \lambda(2c-1)\sin{\vecc{\Theta}}\cos{\vecc{\Theta}} \\ 
-\lambda\sin{\vecc{\Theta}}\cos{\vecc{\Theta}} & \lambda \cos^2{\vecc{\Theta} + (1-\lambda)\vecc{I}_{N-r}}
\end{bmatrix}
\transpose{\widetilde{\vecc{A}}}.
\end{equation}

\subsection{Asymptotic convergence rate}
With the assumption $r<N$, there exists a nonzero principal angle between subspaces $\mathcal{N}(\vecc{A})$ and $\mathcal{N}(\vecc{B})$. The following lemma gives values of all the principal angles. 
\begin{lemma}\label{lem:principle_angle_value}
The principal angles $\theta_i$, $i = 1,\cdots,N-r$, between subspaces $\mathcal{N}(\vecc{A})$ and $\mathcal{N}(\vecc{B})$ satisfy
\begin{equation}\label{eq:principle_angle_value}
\cos{\theta_1} = \cdots = \cos{\theta_{N-r-1}} = 1 
\quad\text{and}\quad
\cos{\theta_{N-r}} = \sqrt{\frac{r}{N}}. 
\end{equation}
\end{lemma}
\begin{proof}
Let $\mathcal{N}(\vecc{A})^\perp$ denote the orthogonal complement of space $\mathcal{N}(\vecc{A})$.
Since $\vecc{A} = \begin{bmatrix}
    1 & 1 & \cdots & 1
\end{bmatrix}\in\IR^{1\times N}$, we have $\mathcal{N}(\vecc{A})^\perp = \mathrm{span}\{\vec{1}\}$.
Recall the columns of $\vecc{B}_0$ are the orthogonal basis of $\mathcal{N}(\vecc{B})$. 
The principal angles between $\mathcal{N}(\vecc{A})^\perp$ and $\mathcal{N}(\vecc{B})$ can be computed via the SVD of $\frac{1}{\sqrt{N}}\transpose{\vec{1}}\vecc{B}_0$.
Each column of $\vecc{B}_0$ is a standard basis $\vec{e}_j$, where $j\neq i_1,\cdots,i_r$. Thus 
\begin{equation*}
\transpose{\Big(\frac{1}{\sqrt{N}}\transpose{\vec{1}}\vecc{B}_0\Big)}\Big(\frac{1}{\sqrt{N}}\transpose{\vec{1}}\vecc{B}_0\Big) 
= \frac{1}{N}
\begin{bmatrix}
1 & 1 & \cdots & 1\\ 
1 & 1 & \cdots & 1\\ 
\vdots & \vdots &  & \vdots\\ 
1 & 1 & \cdots & 1
\end{bmatrix}_{(N-r)\times(N-r)}.
\end{equation*}
The eigenvalues of the $(N-r)\times (N-r)$ matrix consisting of all ones, are $N-r$ and $0, \cdots, 0$. So the singular values of $\frac{1}{\sqrt{N}}\transpose{\vec{1}}\vecc{B}_0$ are $\sqrt{\frac{N-r}{N}}$ and $0, \cdots, 0$. 
We conclude $\cos{\theta_{N-r}} = \sqrt{\frac{r}{N}}$, since the non-trivial principal angles between $\mathcal{N}(\vecc{A})$ and $\mathcal{N}(\vecc{B})$ and the corresponding non-trivial principal angles between $\mathcal{N}(\vecc{A})^\perp$ and $\mathcal{N}(\vecc{B})$ sum up to $\frac{\pi}{2}$, see the Theorem~2.7 in \cite{knyazev2007majorization}.
In addition, since the dimension of $\mathcal{N}(\vecc{A})$ is $N-1$ and the dimension of $\mathcal{N}(\vecc{B})$ is $N-r$, then as long as $N-r>1$, from the definition of principal angles, it is straightforward to see $\cos{\theta_1} = \cdots = \cos{\theta_{N-r-1}} = 1$.
\end{proof}
By Lemma~\ref{lem:principle_angle_value},  there exists only one nonzero principal angle $\theta_{N-r}$. By eliminating zero columns in \eqref{eq:mat_Tclamb}, \eqref{eq:mat_Tclamb} can be simplified as 
{\footnotesize
\begin{align*}
 \resizebox{.99\textwidth}{!}{$
\vecc{T}_{c,\lambda} =
[\vecc{A}_0\vecc{U}_0 ~ \vecc{A}_1]
\setlength{\arraycolsep}{3.0pt}\begin{bmatrix}
\vecc{0}_{r-1} & & & \\
& (1-\lambda+\lambda c)\vecc{I}_{N-r-1} & & \\
& & \lambda c\cos^2{\theta_{N-r}} + \lambda (1-c)\sin^2{\theta_{N-r}} + (1-\lambda) & \lambda(2c-1)\sin{\theta_{N-r}}\cos{\theta_{N-r}} \\ 
& & -\lambda\sin{\theta_{N-r}}\cos{\theta_{N-r}} & \lambda \cos^2{\theta_{N-r} + (1-\lambda)}
\end{bmatrix}
\begin{bmatrix}
\transpose{\vecc{U}}_0\transpose{\vecc{A}}_0\\
\transpose{\vecc{A}}_1
\end{bmatrix}.$}
\end{align*}
}
From \eqref{eq:mat_Tclamb} we know the matrix $\vecc{T}_{c,\lambda}$ is a nonnormal matrix, thus $\norm{\vecc{T}_{c,\lambda}^k}{2}$ is significantly smaller than $\norm{\vecc{T}_{c,\lambda}}{2}^k$ for sufficiently large $k$.
Therefore, the asymptotic convergence rate is governed by $\lim_{k\rightarrow\infty} \norm{\vecc{T}_{c,\lambda}^k}{2}^{\frac{1}{k}}$, which is equal to the norm of the eigenvalue of $\vecc{T}_{c,\lambda}$ with the largest magnitude.
We have
\begin{multline*}
\det(\vecc{T}_{c,\lambda}-\rho\vecc{I}) = (\rho - 1+\lambda-\lambda c)^{N-r-1} (\rho - 1+\lambda c)^{r-1}\\
\times \left[\rho^2 - (\lambda (c \cos{2\theta_{N-r}} - 1) +2)\rho + \lambda^2 c \sin^2{\theta_{N-r}} + \lambda(c\cos{2\theta_{N-r}} - 1) + 1\right].
\end{multline*}
By Lemma~\ref{lem:principle_angle_value}, the matrix $\vecc{T}_{c,\lambda}$ has eigenvalues $\rho_0 = 1-\lambda c$ and $\rho_1 = 1 - \lambda(1-c)$ corresponding to the principle angles $\theta_1, \cdots, \theta_{N-r-1}$, Corresponding to the principle angle $\theta_{N-r}$, the matrix $\vecc{T}_{c,\lambda}$ has another two eigenvalues, $\rho_2$ and $\rho_3$, satisfying the following quadratic equation:
\begin{equation}\label{eq:quadratic_eq}
\rho^2 - (\lambda (c \cos{2\theta_{N-r}} - 1) +2)\rho + \lambda^2 c \sin^2{\theta_{N-r}} + \lambda(c\cos{2\theta_{N-r}} - 1) + 1 = 0.
\end{equation}
The discriminant of above equation is $\Delta = \lambda^2(c^2\cos^2{2\theta_{N-r}} - 2c + 1)$.
The two solutions of $\Delta = 0$ are $[1 \pm \sin(2\theta_{N-r})]/\cos^{2}(2\theta_{N-r})$. Notice that $[1 + \sin(2\theta)]/\cos^{2}(2 \theta) \geq 1$ for any $\theta \in [0, \frac{\pi}{2}]$ and $c \in(0,1)$. Let $c^{\ast} = [1-\sin(2 \theta_{N-r})]/\cos^{2}(2\theta_{N-r})$, then the magnitudes of $\rho_2$ and $\rho_3$ are:
\begin{align*}
\text{if}~ c \leq c^{\ast}, \quad\text{then}~
&\abs{\rho_2} = \frac{1}{2}\abs{\lambda c \cos(2\theta_{N-r})-\lambda+2 + \lambda\sqrt{\cos^{2}(2 \theta_{N-r}) c^{2}-2 c+1}\,},\\
&\abs{\rho_3} = \frac{1}{2}\abs{\lambda c \cos(2\theta_{N-r})-\lambda+2 - \lambda\sqrt{\cos^{2}(2 \theta_{N-r}) c^{2}-2 c+1}\,},\\
\text{if}~ c > c^{\ast}, \quad\text{then}~
&\abs{\rho_2} = \abs{\rho_3} = \sqrt{c\lambda^{2}\sin^{2} \theta_{N-r} - (1-c\cos(2 \theta_{N-r})) \lambda+1}\,.
\end{align*}
Recall  the generalized Douglas--Rachford splitting \eqref{gDR} and \eqref{gDR2} converges due to convexity \cite{lions1979splitting}. When the iterations enter  the asymptotic regime (after the cut-off location of the operator $\mathrm{S}$ does not change), the convergence  rate is governed by the largest magnitude of eigenvalues $\rho_0$, $\rho_1$, $\rho_2$, and $\rho_3$:
\begin{theorem}
\label{thm-rate}
For a standard fixed point of generalized Douglas--Rachford splitting iteration as defined in Section~\ref{sec:structure_fix_p}, the asymptotic convergence rate of  \eqref{gDR} solving \eqref{formulation-opt4} is linear. There exists a sufficiently large $K>0$, such that for any integer $k\geq K$, we have 
\begin{equation*}
\norm{\vec{y}^k - \vec{y}^\ast}{2} \leq \widetilde{C} \Big(\min_{c,\lambda}\max\{\abs{\rho_0}, \abs{\rho_1},\abs{\rho_2},\abs{\rho_3}\}\Big)^k,
\end{equation*}
where $K$ and $\widetilde{C}$ may depend on $\vecc{A}$, $b$, and $\vec{y}^0$.
\end{theorem}

\subsection{A simple strategy of choosing nearly optimal parameters}\label{sec:param_selection}
For solving problem \eqref{formulation-opt4}, after the iteration of algorithm \eqref{gDR} enters   the asymptotic linear convergence regime, the rate of convergence is governed by the largest magnitude of $\rho_0$, $\rho_1$, $\rho_2$, and $\rho_3$. For seeking optimal parameters, we can safely ignore $\rho_0$ because it is straightforward to verify that $\rho_0\leq \rho_1$ with the optimal parameters derived below.
It is highly preferred to construct a guideline for selecting parameters $c$ and $\lambda$ such that for   $\max\{\abs{\rho_1},\abs{\rho_2},\abs{\rho_3}\}$ is reasonably small.
\par
We first consider the case $\theta_{N-r}\in (\frac{\pi}{4},\frac{\pi}{2}]$. It is easy to check $c^{\ast} \!=\! \frac{1}{(\cos\theta_{N-r} + \sin\theta_{N-r})^2} \in (\frac{1}{2},1]$. Define surfaces 
$\Gamma_i = \{(c,\lambda,z): 0<c<c^\ast,~ 0<\lambda\leq2,\, z=\abs{\rho_i}\}$, where $i \in \{1, 2, 3\}$.
For any point $(c,\lambda,z)\in \Gamma_2\cap\Gamma_3$, due to the fact that $\abs{a+b} = \abs{a-b}$ implies $ab=0$ for any $a,b\in\IR$, we have $(\lambda c\cos(2\theta_{N-r}) - \lambda + 2)\sqrt{\Delta} = 0$.
When $c < c^{\ast}$ the discriminant $\Delta > 0$, we get $\lambda c\cos(2\theta_{N-r}) - \lambda + 2 = 0$.
Thus, if there exists a point belongs to $\Gamma_1\cap \Gamma_2\cap \Gamma_3$, then it satisfies
\begin{align*}
\begin{cases}
\abs{1-\lambda(1-c)} = \frac{\lambda}{2}\sqrt{\cos^{2}(2\theta_{N-r}) c^{2} - 2c + 1}\,,\\
\lambda c\cos(2\theta_{N-r}) - \lambda + 2 = 0.
\end{cases}
\end{align*}
On surfaces $\Gamma_i$, $i\in\{1, 2, 3\}$, the parameters $c\in(0,c^\ast)$ and $\lambda \in (0,2]$ implies above equations only have one solution $c = \frac{1}{2}$ and $\lambda = \frac{4}{2-\cos{(2\theta_{N-r})}}$. Thus, we have 
\begin{equation}\label{eq:three_surfs_cap}
\Gamma_1\cap \Gamma_2\cap \Gamma_3 = \Big\{\Big(\frac{1}{2}, \frac{4}{2-\cos{(2\theta_{N-r})}}, -\frac{\cos{(2\theta_{N-r})}}{2-\cos{(2\theta_{N-r})}}\Big)\Big\}.
\end{equation}
Therefore, we know when $\theta_{N-r}\in(\frac{\pi}{4},\frac{\pi}{2}]$, the minimum of $\max\{\abs{\rho_1},\abs{\rho_2},\abs{\rho_3}\}$ for $c\in(0,c^\ast)$ and $\lambda \in (0,2]$ is not greater than $-\frac{\cos{(2\theta_{N-r})}}{2-\cos{(2\theta_{N-r})}}$.
To deal with $c\in[c^\ast,1)$, we need the following lemma.
\begin{lemma}\label{lem:param_select_ineq}
Assume $\rho_1$ and $\rho_2$ are functions of $c$ and $\lambda$, for which the minimum can be attained. Then, the following inequality holds.
\begin{equation*}
\min_{c,\lambda}\, \max\{\abs{\rho_1}, \abs{\rho_2}\} \geq \max\{\min_{c,\lambda}\abs{\rho_1}, \min_{c,\lambda}\abs{\rho_2}\}.
\end{equation*}
\end{lemma}
\begin{proof}
Assume the minimum of $\max\{\abs{\rho_1}, \abs{\rho_2}\}$ is achieved at $(c_0, \lambda_0)$. 
We have
\begin{itemize}[topsep=2pt,noitemsep]
\item[\emph{i}.] If $\abs{\rho_1(c_0, \lambda_0)} \geq \abs{\rho_2(c_0, \lambda_0)}$, then 
$ 
\min_{c,\lambda}\, \max\{\abs{\rho_1}, \abs{\rho_2}\} = \abs{\rho_1(c_0, \lambda_0)} \geq \min_{c,\lambda}\abs{\rho_1}.$ 
\item[\emph{ii}.] If $\abs{\rho_1(c_0, \lambda_0)} < \abs{\rho_2(c_0, \lambda_0)}$, then 
$ \min_{c,\lambda}\, \max\{\abs{\rho_1}, \abs{\rho_2}\} = \abs{\rho_2(c_0, \lambda_0)} > \abs{\rho_1(c_0, \lambda_0)}.$
Proof by contradiction: assume $\underset{c,\lambda}{\min} \max\{\abs{\rho_1}, \abs{\rho_2}\} < \underset{c,\lambda}{\min} \abs{\rho_1}$, then it implies $\abs{\rho_1(c_0,\lambda_0)}<\underset{c,\lambda}{\min}\abs{\rho_1}$, which is impossible.
\end{itemize}
Thus, $\underset{c,\lambda}{\min} \max\{\abs{\rho_1}, \abs{\rho_2}\} \geq \underset{c,\lambda}{\min}\abs{\rho_1}$. Similarly,   $\underset{c,\lambda}{\min} \max\{\abs{\rho_1}, \abs{\rho_2}\} \geq \underset{c,\lambda}{\min}\abs{\rho_2}$.
\end{proof}
When $c\in[c^\ast,1)$, the magnitude of $\rho_2$ and $\rho_3$ are equal, namely we only need to find suitable parameters $c$ and $\lambda$ such that the $\max\{\abs{\rho_1},\abs{\rho_2}\}$ is reasonably small. 
It is easy to verify that, when $c\in[c^\ast,1)$ and $\lambda\in(0,2]$, the function $\rho_1$ is monotonically increasing with respect to $c$ and monotonically decreasing with respect to $\lambda$. Thus, $\rho_1(c^\ast,2) = 2c^\ast-1>0$ gives $\abs{\rho_1} = \rho_1$.
Associated with $\lambda$ greater or less than $-\frac{\cos{(2\theta_{N-r})}}{\sin^2\theta_{N-r}}$, we have two cases.
\begin{enumerate}
\item[1.] 
When $\lambda\in(0,-\frac{\cos{(2\theta_{N-r})}}{\sin^2\theta_{N-r}}]$, recall the monotonicity of $\rho_1$, we have
\begin{multline*}
\min_{c\in[c^\ast,1),~ \lambda\in(0,-\frac{\cos{(2\theta_{N-r})}}{\sin^2\theta_{N-r}}]}\abs{\rho_1} 
= \rho_1\Big(c^\ast, -\frac{\cos{(2\theta_{N-r})}}{\sin^2\theta_{N-r}}\Big)\\
= 1 + \frac{\cos{(2\theta_{N-r})}}{\sin^2\theta_{N-r}}\Big(1-\frac{1}{(\cos\theta_{N-r} + \sin\theta_{N-r})^2}\Big)
> \frac{1}{2} > -\frac{\cos{(2\theta_{N-r})}}{2-\cos{(2\theta_{N-r})}}.
\end{multline*}
By Lemma~\ref{lem:param_select_ineq}, when the principal angle $\theta_{N-r}\in(\frac{\pi}{4}, \frac{\pi}{2}]$, we know
\begin{equation*}
\min_{c\in[c^\ast,1),~ \lambda\in(0,-\frac{\cos{(2\theta_{N-r})}}{\sin^2\theta_{N-r}}]}\, \max\{\abs{\rho_1}, \abs{\rho_2}\} > -\frac{\cos{(2\theta_{N-r})}}{2-\cos{(2\theta_{N-r})}}.
\end{equation*}
Therefore, the common point of the three surfaces $\Gamma_1$, $\Gamma_2$, and $\Gamma_3$ in \eqref{eq:three_surfs_cap} is still a good choice.
\item[2.] When $\lambda\in(-\frac{\cos{(2\theta_{N-r})}}{\sin^2\theta_{N-r}},2]$, define 
$\kappa = c\lambda^{2}\sin^{2} \theta_{N-r} - (1-c\cos(2 \theta_{N-r})) \lambda+1.$
We have
$
\frac{\partial \kappa}{\partial c} = \lambda(\lambda\sin^2\theta_{N-r} + \cos{(2\theta_{N-r})})>0,
$
which implies $\kappa$ is monotonically increasing with respect to $c$ in the interval $[c^\ast,1)$. Thus, for any $c\geq c^\ast$, the $\abs{\rho_2(c,\lambda)}\geq\abs{\rho_2(c^\ast,\lambda)}$ holds. Again, recall the monotonicity of $\rho_1$, we obtain
\begin{equation*}
\min_{c\in[c^\ast,1),~ \lambda\in(-\frac{\cos{(2\theta_{N-r})}}{\sin^2\theta_{N-r}},2]} \!\!\!\!\!\!\max\{\abs{\rho_1}, \abs{\rho_2}\} =\!\!\! \underset{\lambda\in(-\frac{\cos{(2\theta_{N-r})}}{\sin^2\theta_{N-r}},2]}{\min}\!\!\!\!\max\{\abs{\rho_1(c^\ast,\lambda)}, \abs{\rho_2(c^\ast,\lambda)}\}.
\end{equation*}
Since $\abs{\rho_1(c^\ast,\lambda)} = 1-\lambda(1-c^\ast)$ and $\abs{\rho_2(c^\ast,\lambda)} = \abs{1 - \frac{\lambda}{1+\cot\theta_{N-r}}}$, when $\theta_{N-r}\in(\frac{\pi}{4}, \frac{\pi}{2}]$, $\frac{1}{1+\cot\theta_{N-r}} > 1-c^\ast$, then the equation $\abs{\rho_1(c^\ast,\lambda)} = \abs{\rho_2(c^\ast,\lambda)}$ has one and only one root 
\begin{equation*}
\lambda^\ast = \frac{2}{1+\frac{1}{1+\cot\theta_{N-r}}-\frac{1}{(\cos\theta_{N-r} + \sin\theta_{N-r})^2}}.
\end{equation*}
Therefore, we know when $\theta_{N-r}\in(\frac{\pi}{4},\frac{\pi}{2}]$, the minimum of $\max\{\abs{\rho_1},\abs{\rho_2},\abs{\rho_3}\}$ for $c\in[c^\ast,1)$ and $\lambda \in (-\frac{\cos{(2\theta_{N-r})}}{\sin^2\theta_{N-r}},2]$ is not larger than $1-\lambda^\ast(1-c^\ast)$.
\end{enumerate}
\par
Next, let us consider the case $\theta_{N-r}\in (0,\frac{\pi}{4}]$. When $c\in(0,c^\ast)$ and $\lambda\in(0,2]$, the discriminant $\Delta>0$, namely the quadratic equation \eqref{eq:quadratic_eq} has two real roots. Moreover, $\abs{\rho_2} > \abs{\rho_3}$ obviously. Thus, we only need to minimize the $\max\{\abs{\rho_1}, \abs{\rho_2}\}$.
Define
\begin{equation*}
\tilde{\kappa} = \lambda c \cos(2\theta_{N-r})-\lambda+2 + \lambda\sqrt{\cos^{2}(2 \theta_{N-r}) c^{2}-2 c+1}.
\end{equation*}
Since for any $\theta_{N-r}\in (0,\frac{\pi}{4}]$, $c\in(0,c^\ast)$, and $\lambda\in(0,2]$ the $\lambda c \cos(2\theta_{N-r})-\lambda+2>0$, we have $\abs{\rho_2} = \frac{1}{2}\tilde{\kappa}$. From
\begin{align*}
\frac{\partial\tilde{\kappa}}{\partial c} &= \lambda\Big(\cos(2\theta_{N-r}) + \frac{c\cos^2(2\theta_{N-r})-1}{\sqrt{\cos^{2}(2 \theta_{N-r}) c^{2}-2 c+1}}\Big) \leq 0,\\
\frac{\partial\tilde{\kappa}}{\partial\lambda} &= c \cos(2\theta_{N-r}) - 1 + \sqrt{\cos^{2}(2 \theta_{N-r}) c^{2}-2 c+1} \leq 0,
\end{align*}
we know the $\tilde{\kappa}$ is monotonically decreasing with respect to both $c$ and $\lambda$. 
Thus $\tilde{\kappa}$ take minimum at $c=c^\ast$ and $\lambda=2$. By Lemma~\ref{lem:param_select_ineq}, when the principal angle $\theta_{N-r}\in(0,\frac{\pi}{4}]$, we know
\begin{equation}\label{eq:scenario1_case1_1}
\min_{c\in(0,c^\ast),~\lambda\in(0,2]}\, \max\{\abs{\rho_1}, \abs{\rho_2}\} 
\geq \min_{c\in(0,c^\ast),~\lambda\in(0,2]}\abs{\rho_2}
=\frac{1}{2}\tilde{\kappa}(c^\ast,2) = c^\ast\cos{2\theta_{N-r}}.
\end{equation}
Notice, when $c=c^\ast$ and $\lambda=2$, the magnitude of $\rho_1$ and $\rho_2$ can be simplified as $\abs{\rho_1} = \abs{2c^\ast-1}$ and $\abs{\rho_2} = c^\ast\cos{2\theta_{N-r}}$, where $c^{\ast} = \frac{1}{(\cos\theta_{N-r} + \sin\theta_{N-r})^2}$. It is easy to check that $\abs{\rho_2}>\abs{\rho_1}$ holds for any $\theta_{N-r}\in(0,\frac{\pi}{4}]$. We have
\begin{align}\label{eq:scenario1_case1_2}
\min_{c\in(0,c^\ast),~\lambda\in(0,2]}\!\!\!\!\! \max\{\abs{\rho_1}, \abs{\rho_2}\} 
\leq \max\{\abs{\rho_1(c^\ast,2)}, \abs{\rho_2(c^\ast,2)}\} 
= \abs{\rho_2(c^\ast,2)} 
= c^\ast\cos{2\theta_{N-r}}.
\end{align}
From above \eqref{eq:scenario1_case1_1} and \eqref{eq:scenario1_case1_2}, we obtain the minimum of $\max\{\abs{\rho_1}, \abs{\rho_2}, \abs{\rho_3}\}$ equals $c^\ast \cos{2\theta_{N-r}}$, which is achieved at $c=c^\ast$ and $\lambda=2$.
When $c\in[c^\ast,2)$, following the similar argument as above, we can show $\abs{\rho_1} = 1-\lambda(1-c)$, which is monotonically increasing with respect to $c$ and monotonically decreasing with respect to $\lambda$. In addition, we also have $\abs{\rho_2} = \abs{\rho_3}$ which is monotonically increasing with respect to $c$. Thus, we have
\begin{align*}
\min_{c\in[c^\ast,1),~ \lambda\in(0,2]} \max\{\abs{\rho_1}, \abs{\rho_2}, \abs{\rho_3}\} 
&= \min_{\lambda\in(0,2]} \max\{\abs{\rho_1(c^\ast,\lambda)}, \abs{\rho_2(c^\ast,\lambda)}\} \\
&= \min_{\lambda\in(0,2]} \frac{1}{2}\lambda c^\ast \cos(2\theta_{N-r}) - \frac{1}{2}\lambda + 1. 
\end{align*}
The last equality above is due to the fact the $\abs{\rho_1(c^\ast,\lambda)} \leq \abs{\rho_2(c^\ast,\lambda)}$ holds for any $\theta_{N-r}\in(0,\frac{\pi}{4}]$.
From $\lambda c^\ast \cos(2\theta_{N-r}) - \lambda$ is monotonically decreasing with respect to $\lambda$, we know, in this case, the minimum equals $c^\ast\cos(2\theta_{N-r})$, which is taken at $c=c^\ast$ and $\lambda=2$. 
\par
To this end, let us make a summary of the parameter selection principle as follows.
\begin{itemize}[leftmargin=0.4cm]
\item[1.] When $\theta_{N-r}\in(\frac{3}{8}\pi,\frac{1}{2}\pi]$, a suitable choice of parameters are: $c=\frac{1}{2}$, $\lambda = \frac{4}{2-\cos{(2\theta_{N-r})}}$. The associated asymptotic linear convergence rate is governed by $-\frac{\cos{(2\theta_{N-r})}}{2-\cos{(2\theta_{N-r})}}$.
\item[2.] When $\theta_{N-r}\in(\frac{1}{4}\pi,\frac{3}{8}\pi]$, a suitable choice of parameters are: $c=c^\ast$, $\lambda =\lambda^\ast$. The associated asymptotic linear convergence rate is governed by $1-\lambda^\ast(1-c^\ast)$.
\item[3.] When $\theta_{N-r}\in(0,\frac{1}{4}\pi]$, a suitable choice of parameters are: $c=c^\ast$, $\lambda =2$. The associated asymptotic linear convergence rate is governed by $c^\ast\cos(2\theta_{N-r})$.
\end{itemize}

\begin{remark}\label{rmk:estimate_r}
The exact value of the principal angle $\theta_{N-r}$ in \eqref{eq:principle_angle_value} is unknown. But it is simple to estimate $\theta_{N-r}$ by counting the number of bad cells, e.g., let $\hat{r}$ be the number of $u_i\notin[m, M]$ and 
 use $\hat{r}$ instead of $r$ in \eqref{eq:principle_angle_value} to compute $\theta_{N-r}$. This gives a simple guideline \eqref{opt-parameter} for choosing nearly optimal parameters, which is efficient in all our numerical tests as shown in  Section~\ref{sec:experiments}.
\end{remark}

\begin{remark}
    In a large scale 3D problem, usually the ratio of bad cells with cell averages out of bound in the DG scheme is quite small. In such a case, we expect $r\ll N$, with which $\theta_{N-r}$ is very close to zero. In this case, by the discussions above, the convergence rate in Theorem \ref{thm-rate} becomes $-\frac{\cos{(2\theta_{N-r})}}{2-\cos{(2\theta_{N-r})}}$. If $\hat{r}$ is also a good approximation to $r$, which is usually true in this context, then we get the rate  \eqref{opt-rate}.
\end{remark}

With the  guideline \eqref{opt-parameter} for choosing nearly optimal parameters in \eqref{gDR2}, we can use the two-step limiter as explained in Section
\ref{sec-intro-limiter} to enforce bounds of DG solutions.
\section{Application to phase-field equations}\label{sec:application_chns}
One of the popular approaches for modeling multi-phase fluid flow in micro-to-millimeter pore structures is to use phase-field equations \cite{frank2018direct}. Efficient and accurate pore-scale fluid dynamics simulators have important applications in digital rock physics (DRP), which has been extensively used in the petroleum industry for optimizing enhanced oil recovery schemes.

\subsection{Mathematical model}
In an open bounded domain $\Omega\subset\IR^d$ over a time interval $(0,T]$, the dimensionless CHNS equations are given by:
\begin{subequations}\label{eq:CHNS:model}
\begin{align}
\partial_t{\phi} - \frac{1}{\Pe}\div{(\mathcal{M}(\phi)\grad{\mu})} + \div{(\phi\vec{v})} &= 0 && \text{in}~(0,T]\times\Omega,\label{eq:CHNS:model1}\\
\mu + \Cn^2 \laplace{\phi} - \Phi'(\phi) &= 0  && \text{in}~(0,T]\times\Omega,\label{eq:CHNS:model2}\\
\partial_t{\vec{v}} + \vec{v}\cdot\grad{\vec{v}} - \frac{2}{\Rey}\div{\strain{(\vec{v})}} + \frac{1}{\Rey\Ca}\grad{p} - \frac{3}{2\sqrt{2}\,\Rey\Ca\Cn} \mu\grad{\phi} &= 0 && \text{in} ~(0,T]\times\Omega,\label{eq:CHNS:model3}\\
\div{\vec{v}} &= 0 && \text{in} ~(0,T]\times\Omega,\label{eq:CHNS:model4}
\end{align}
\end{subequations}
where $\phi$, $\mu$, $\vec{v}$, and $p$ are order parameter, chemical potential, velocity, and pressure. 
The non-dimensional quantities $\Pe$, $\Cn$, $\Rey$, and $\Ca$ denote the P\'eclet number, Cahn number, Reynolds number, and capillary number, respectively. 
The strain tensor is given by $\strain{(\vec{v})} = \frac{1}{2}(\grad{\vec{v}}+\transpose{(\grad{\vec{v}})})$.
The function $\mathcal{M}$ denotes mobility. Typical choices of $\mathcal{M}$ include the constant mobility $\mathcal{M}(\phi) = \mathcal{M}_0>0$, where $\mathcal{M}_0$ can be set to $1$ after nondimensionalization, and the degenerate mobility $\mathcal{M}(\phi) = 1- \phi^2$. 
The function $\Phi$ is a scalar potential, which is also called chemical energy density. Classical and widely used forms are the polynomial Ginzburg--Landau (GL) double well potential:
$\Phi(\phi) = \frac{1}{4}(1-\phi)^2(1+\phi)^2$ 
and the Flory--Huggins (FH) logarithmic potential with parameters $\alpha$ and $\beta$:
$\Phi(\phi) = \frac{\alpha}{2}\big((1+\phi)\ln\big(\frac{1+\phi}{2}\big) + (1-\phi)\ln\big(\frac{1-\phi}{2}\big)\big) + \frac{\beta}{2}(1-\phi^2)$.
\par
We supplement \eqref{eq:CHNS:model} with initials $\phi=\phi^0$ and $\vec{v}=\vec{v}^0$ on $\{0\}\times\Omega$.
Let $\normal$ denote the unit outward normal to domain $\Omega$. We decompose the boundary $\partial{\Omega}$ into three disjoint subsets $\partial{\Omega} = \partial{\Omega}^{\mathrm{wall}} \cup \partial{\Omega}^{\mathrm{in}} \cup \partial{\Omega}^{\mathrm{out}}$, where $\partial{\Omega}^{\mathrm{wall}}$ denotes fluid--solid interface and $\partial{\Omega}^{\mathrm{in}}$ and $\partial{\Omega}^{\mathrm{out}}$ are inflow boundary and outflow boundary
\begin{equation*}
\partial{\Omega}^{\mathrm{in}} = \{\vec{x}\in\partial{\Omega}:~ \vec{v}\cdot\normal < 0\}
~~\text{and}~~
\partial{\Omega}^{\mathrm{out}} = \partial{\Omega} \setminus (\partial{\Omega}^{\mathrm{wall}} \cup \partial{\Omega}^{\mathrm{in}}).
\end{equation*}
We prescribe Dirichlet boundary conditions $\phi = \phi_\mathrm{D}$ and $\vec{v} = \vec{v}_\mathrm{D}$ on $(0,T]\times\partial{\Omega}^{\mathrm{in}}$.
For velocity, the no-slip boundary condition $\vec{v}=\vec{0}$ is used on $(0,T]\times\partial{\Omega}^{\mathrm{wall}}$ and ``do nothing'' boundary condition $(2\strain{(\vec{v})} - \frac{1}{\Ca}p\vecc{I})\normal = \vec{0}$ is applied on $(0,T]\times\partial{\Omega}^{\mathrm{out}}$. 
Wettability is modeled by a contact angle $\vartheta$ that is enforced by:
$\grad{\phi}\cdot\normal = -\frac{2\sqrt{2} \cos(\vartheta)}{3\Cn} \mathit{g}'(\phi)$ 
on~$(0,T]\times(\partial\Omega^{\mathrm{wall}}\cup\partial\Omega^{\mathrm{out}})$,
where the function $\mathit{g}$ is a blending function. The closed-form expression of $\mathit{g}$ depends on the choice of chemical energy density \cite{carlson2011dissipation}. For the Ginzburg--Landau potential, we have $\mathit{g}(\phi) = \frac{1}{4}(\phi^3-3\phi+2)$. 
In addition, we employ the homogeneous Neumann boundary condition $\mathcal{M}(\phi)\grad{\mu}\cdot\normal = 0$ on $(0,T]\times\partial{\Omega}$ to ensure the global mass conservation.
\par
The order parameter $\phi$ is the difference between the mass fraction $\phi_\mathrm{A}$ and $\phi_\mathrm{B}$ of the phase $\mathrm{A}$ and phase $\mathrm{B}$. 
With constraint $\phi_\mathrm{A} + \phi_\mathrm{B} = 1$ for a two-component mixture as well as mass fractions belonging to $[0,1]$, a physically meaningful range of the order parameter field is $[-1,1]$.
The Cahn--Hilliard equation with the degenerate mobility or with the logarithmic potential enjoys bound-preserving property \cite{tierra2015numerical}. However, for constant mobility with GL polynomial potential, the analytical solution of Cahn--Hilliard equation is not bound-preserving \cite{barrett1999finite}.
For a given initial data $\phi^0\in[-1,1]$, it is an open question whether the solution of a fully coupled CHNS system should have a bounded order parameter in $[-1,1]$.
On the other hand, empirically we would expect a reasonable solution, e.g., the discrete order parameter field, should be bounded by $-1$ and $1$ for any time $t>0$.

\subsection{Time discretization}
The CHNS equations form a highly nonlinear coupled system. One of the popular approaches of constructing efficient numerical algorithms for large-scale simulations in complex computational domains is to use splitting methods, e.g., to decouple the mass and momentum equations and to further split the convection from the incompressibility constraint \cite{shen2012modeling}. Also, see \cite{guermond2006overview,glowinski2003finite} for an overview of the splitting methods for time-dependent incompressible flows.
\par
We uniformly partition the interval $[0,T]$ into $\Nst$ subintervals. Let $\tau$ denote the time step size. 
For the chemical energy density, we adopt a convex--concave decomposition of the form $\Phi = \Phi_{+} + \Phi_{-}$, where the convex part $\Phi_{+}$ is treated time implicitly and the concave part $\Phi_{-}$ is treated time explicitly. 
For the nonlinear convection $\vec{v}\cdot\grad{\vec{v}}$, the form $\mathcal{C}(\cdot,\cdot)$ is a semi-discretization that satisfies a positivity property, see the equation (12) in \cite{liu2019interior}. 
For any $1 \leq n \leq \Nst$, our first-order time discretization algorithm  consists of the following steps:
\begin{itemize}[leftmargin=0.5cm]
\item[] Step~1. Given $(\phi^{n-1}, \vec{w}^{n-1})$, compute $(\phi^n, \mu^n)$ such that
\begin{align*}
\phi^n - \frac{\tau}{\Pe}\div{(\mathcal{M}(\phi^{n-1})\grad{\mu^n})} + \tau\div(\phi^{n}\vec{w}^{n-1}) = \phi^{n-1}& && \text{in}~\Omega,\\
-\mu^n - \Cn^2\laplace{\phi^n} + \Phi_+\,\!'(\phi^n) = - \Phi_-\,\!'(\phi^{n-1})& && \text{in}~\Omega.
\end{align*}
\item[] Step~2. Given $(\phi^{n}, \mu^n, \vec{v}^{n-1}, p^{n-1}, \psi^{n-1})$, compute $\vec{v}^n$ such that
\begin{align*}
\vec{v}^n + \tau \mathcal{C}(\vec{v}^{n-1},\vec{v}^n) - \frac{2\tau}{\Rey}\div{\strain{(\vec{v}^n})} &= \vec{v}^{n-1} \nonumber\\ - \frac{\tau}{\Rey\Ca}\grad{(p^{n-1}+\psi^{n-1})} &+ \frac{3\tau}{2\sqrt{2}\,\Rey\Ca\Cn} \mu^{n}\grad{\phi^n} && \text{in}~\Omega.
\end{align*}
\item[] Step~3. Given $\vec{v}^n$, compute $\psi^n$ such that
\begin{align*}
-\laplace{\psi^n} &= -\frac{\Rey\Ca}{\tau}\div{\vec{v}^n} && \text{in}~\Omega.
\end{align*}
\item[] Step~4. Given $(\vec{v}^{n}, p^{n-1}, \psi^n)$, compute $(\vec{w}^n, p^n)$ such that
\begin{align*}
\vec{w}^n &= \vec{v}^n - \frac{\tau}{\Rey\Ca} \grad{\psi^n},\\
p^n &= p^{n-1} + \psi^n - \sigma_\chi \Ca \div{\vec{v}^n}.
\end{align*}
\end{itemize}
The parameter $\sigma_\chi$ is equal to $\frac{2}{d}$, namely, we use $\sigma_\chi=\frac{2}{3}$ for our numerical simulations in three dimensions.
To start time marching, we set $p^0 = 0$ and $\psi^0=0$. The functions $\phi^0$ and $\vec{w}^0 = \vec{v}^0$ are given initial data.
\begin{remark}
The above scheme  is a combination of the convex splitting approach for the Cahn--Hilliard equation with the classical rotational  pressure-correction algorithm  (see Section~3.4 in \cite{guermond2006overview}) for the Navier--Stokes equations. More precisely, Step~2 to Step~4 can be rewritten as follows:
\begin{align*}
&\frac{1}{\tau}(\vec{v}^n-\vec{w}^{n-1}) + \mathcal{C}(\vec{v}^{n-1},\vec{v}^n) - \frac{2}{\Rey}\div{\strain{(\vec{v}^n})} = -\frac{1}{\Rey\Ca}\grad{p^{n-1}} + \frac{3}{2\sqrt{2}\,\Rey\Ca\Cn} \mu^{n}\grad{\phi^n},\\
&\begin{cases}
\displaystyle \frac{1}{\tau}(\vec{w}^{n}-\vec{v}^n) + \frac{1}{\Rey\Ca}\grad{\psi^n} = 0,\\
\div{\vec{w}^{n}} = 0,
\end{cases}\hspace{2.75cm}
\psi^n = p^n - p^{n-1} + \sigma_\chi \Ca \div{\vec{v}^n}.
\end{align*}
We use  $\vec{w}^{n-1}$,  instead of $\vec{v}^{n-1}$, in the advection term in Step~1, since $\div{\vec{w}^{n-1}}=0$. 

For the sake of simplicity, we only presented a first-order version of the scheme, although high-order version can be constructed accordingly. On the other hand, 
it is also possible to construct energy dissipating schemes as in \cite{shen2015decoupled}. Since our focus in this paper is   in preserving bounds for a DG spacial discretization, we employ a simple time-marching strategy. 
\end{remark}

\subsection{Space discretization}
Decoupled splitting algorithms combined with interior penalty DG spatial formations have been constructed to solve various CHNS models in large-scale complex-domain DRP simulations \cite{frank2018direct,liu2020efficient,liu2022pressure}. Also, see \cite{liu2022convergence,masri2022discontinuous,masri2023improved} for solvability, stability, and optimal error estimates on using DG with decoupled splitting schemes for CHNS equations and viscous incompressible flow. Here, we briefly review the fully discrete scheme.
\par
Let $\setE_h = \{E_i\}$ be a family of conforming nondegenerate (regular) meshes of the domain $\Omega$ with maximum element diameter $h$. Let $\Gammah$ be the set of interior faces. For each interior face $e \in \Gammah$ shared by elements $E_{i^-}$ and $E_{i^+}$, with $i^- < i^+$, we define a unit normal vector $\normal_e$ that points from $E_{i^-}$ into $E_{i^+}$. For a boundary face, $e \subset \partial\Omega$, the normal vector $\normal_e$ is taken to be the unit outward vector to $\partial\Omega$.
Let $\IP_k(E_i)$ denote the set of all polynomials of degree at most $k$ on an element $E_i$.
Define the broken polynomial spaces $X_h$ and $\mathbf{X}_h$, for any $k \geq 1$,
\begin{align*}
X_h &= \{ \chi_h \in L^2(\Omega):~\chi_h|_{E_i}\in\IP_k(E_i),~ \forall E_i\in\setE_h\}, \\
\mathbf{X}_h &= \{ \vec{\theta}_h \in L^2(\Omega)^d:~\vec{\theta}_h|_{E_i}\in\IP_k(E_i)^d,~ \forall E_i\in\setE_h\}.
\end{align*}
The average and jump for any scalar quantity $\chi$ on a boundary face coincide with its trace; and on interior faces they are defined by
\begin{equation*}
\avg{\chi}|_e = \frac{1}{2}\on{\chi}{E_{i^-}} + \frac{1}{2}\on{\chi}{E_{i^+}}, \quad
\jump{\chi}|_e = \on{\chi}{E_{i^-}} - \on{\chi}{E_{i^+}}, \quad 
\forall e = \partial E_{i^{-}}\cap\partial E_{i^{+}}.
\end{equation*}
The related definitions for any vector quantity are similar. For more details see \cite{riviere2008discontinuous}.
\par
Let $(\cdot,\cdot)_{\mathcal{O}}$ denote the $L^2$ inner product over ${\mathcal{O}}$. For instance, on any face $e$ the $L^2$ inner product is denoted by $(\cdot,\cdot)_e$. We make use of the following compact notation for the $L^2$ inner product on the interior and boundary faces
\begin{equation*}
(\cdot,\cdot)_{\mathcal{O}} = \sum_{e\in\mathcal{O}} (\cdot,\cdot)_e, ~~\text{where}~~ 
\mathcal{O} = \Gammah,~\partial\Omega,~\partial\Omega^\mathrm{in},~\partial \Omega^\mathrm{out},~ \cdots.
\end{equation*}
For convenience, we omit the subscript when $\mathcal{O} = \Omega$, namely denote $(\cdot,\cdot) = (\cdot,\cdot)_\Omega$. 
We still use $\grad{}$ and $\div{}$ to denote the broken gradient and broken divergence. 
\par
For completeness, let us recall the DG forms below and we skip their derivation. 
Associated with the advection term $\div{(\phi\vec{w})}$ and the convection term $\vec{v}\cdot\grad{\vec{z}}$, we define
\begin{align*}
a_\mathrm{adv}(\phi,\vec{w},\chi) &=  
-(\phi,\vec{w} \cdot \grad{\chi})
+ (\upwind{\phi} \avg{\vec{w}\cdot\normal_e}, \jump{\chi})_{\Gammah}, \\
a_\mathrm{conv}(\vec{v};\vec{z},\vec{\theta}) &= 
(\vec{v}\cdot\grad{\vec{z}},\vec{\theta}) 
+ \frac{1}{2} (\div{\vec{v}}, \vec{z}\cdot\vec{\theta}) \\
&- \frac{1}{2} (\jump{\vec{v}\cdot\normal_e}, \avg{\vec{z}\cdot\vec{\theta}})_{\Gammah\cup\partial\Omega^\mathrm{in}} 
+ \sum_{E\in\setE_h} (\abs{\avg{\vec{v}}\cdot\normal_{E}},(\vec{z}^\mathrm{int} - \vec{z}^\mathrm{ext})\cdot\vec{\theta}^\mathrm{int})_{\partial E_{-}^{\vec{v}}}.
\end{align*}
The superscript $\mathrm{int}$ (resp. $\mathrm{ext}$) refers to the trace of a function on a face of $E$ coming from the interior (resp. exterior).
The set $\partial E_{-}^{\vec{v}}$ is the upwind part of $\partial{E}$, defined by $\partial E_{-}^{\vec{v}} = \{\vec{x}\in\partial{E}: \avg{\vec{v}}\cdot\normal_E<0\}$, where $\normal_E$ is the unit outward normal vector to $E$ \cite{girault2005discontinuous}.
The upwind quantity $\upwind{\phi}$ on an interior face $e$ is evaluated by
\begin{equation*}
\on{\phi^\uparrow}{e\in\Gammah} =
\begin{cases}
\on{\phi}{E_{i^-}} & \text{if}~\avg{\vec{w}}\cdot\normal_e \geq 0, \\
\on{\phi}{E_{i^+}} & \text{if}~\avg{\vec{w}}\cdot\normal_e < 0.
\end{cases} 
\end{equation*}
Associated with the operator $-\div{(z\grad{\xi})}$, we define
\begin{align*}
a_\mathrm{diff}(z;\xi,\chi) &= 
(z\grad{\xi}, \grad{\chi}) 
- (\avg{z\grad{\xi}\cdot\normal_e}, \jump{\chi})_{\Gammah} \\
&- (\avg{z\grad{\chi}\cdot\normal_e}, \jump{\xi})_{\Gammah}
+ \frac{\sigma}{h} (\jump{\xi},\jump{\chi})_{\Gammah}.
\end{align*}
Associated with the Laplace operator $-\laplace{\xi}$ (for terms $-\laplace{\phi}$ and $-\laplace{\psi}$), we define
\begin{align*}
-\laplace{\xi}+\text{Dirichlet on}~\partial{\Omega}^\mathrm{in}~ \leadsto &&
a_\mathrm{diff,in}(\xi,\chi) &= a_\mathrm{diff}(1; \xi,\chi)
- (\grad{\xi}\cdot\normal_e, \chi)_{\partial\Omega^\mathrm{in}} \\  &&
&- (\grad{\chi}\cdot\normal_e, \xi)_{\partial\Omega^\mathrm{in}}
+\frac{\sigma}{h} (\xi, \chi)_{\partial\Omega^\mathrm{in}},
\\
-\laplace{\xi}+\text{Dirichlet on}~\partial{\Omega}^\mathrm{out}~ \leadsto &&
a_\mathrm{diff,out}(\xi,\chi) &= a_\mathrm{diff}(1; \xi,\chi)
- (\grad{\xi}\cdot\normal_e, \chi)_{\partial\Omega^\mathrm{out}} \\ &&
&- (\grad{\chi}\cdot\normal_e, \xi)_{\partial\Omega^\mathrm{out}}
+\frac{\sigma}{h} (\xi, \chi)_{\partial\Omega^\mathrm{out}}.
\end{align*}
Associated with the diffusion term $-2\div{\strain(\vec{v})}$, we define 
\begin{multline*}
a_\mathrm{\mathrm{ellip}}(\vec{v},\vec{\theta}) = 
2(\strain{(\vec{v})}, \strain{(\vec{\theta})})
- 2(\avg{\strain{(\vec{v}}) \normal_e}, \jump{\vec{\theta}})_{\Gammah} 
- 2(\avg{\strain{(\vec{\theta}}) \normal_e}, \jump{\vec{v}})_{\Gammah} \\
+ \frac{\sigma}{h} (\jump{\vec{v}}, \jump{\vec{\theta}})_{\Gammah}
- 2(\strain{(\vec{v}})\normal_e, \vec{\theta})_{\partial\Omega^\mathrm{in}}
- 2(\strain{(\vec{\theta}}) \normal_e, \vec{v})_{\partial\Omega^\mathrm{in}}
+ \frac{\sigma}{h} (\vec{v}, \vec{\theta})_{\partial\Omega^\mathrm{in}}.
\end{multline*}
The remaining forms in the right-hand sides of the discrete equations account for the boundary conditions (see $b_\mathrm{diff}$ and $b_\mathrm{vel}$) and the pressure and potential (see $b_\mathrm{pres}$): 
\begin{align*}
b_\mathrm{diff}(\xi,\chi) 
&= -(\phi_\mathrm{D}, \grad{\chi}\cdot\normal_e)_{\partial\Omega^\mathrm{in}}\!
+\! \frac{\sigma}{h} (\phi_\mathrm{D}, \chi)_{\partial\Omega^\mathrm{in}}\!\!
- \frac{2\sqrt{2}\delta\cos(\vartheta)}{3\,\mathrm{Cn}} (\mathit{g}'(\xi),\chi)_{\partial\Omega^{\mathrm{wall}}\cup\partial\Omega^{\mathrm{out}}},\\
b_\mathrm{pres}(p,\psi,\vec{\theta}) 
&= -(p, \div{\vec{\theta}})
+(\avg{p}, \jump{\vec{\theta}\cdot\vec{n}_e})_{\Gammah\cup\partial\Omega} + (\grad{\psi}, \vec{\theta}),\\
b_\mathrm{vel}(\vec{\theta})
&= -\frac{3}{2}(\vec{v}_\mathrm{D}\cdot\vec{n}, \vec{v}_\mathrm{D} \cdot\vec{\theta})_{\partial\Omega^\mathrm{in}}
- \frac{2}{\Rey} (\strain{(\vec{\theta})\vec{n}_e}, \vec{v}_\mathrm{D})_{\partial\Omega^\mathrm{in}}
+ \frac{\sigma}{h \Rey} (\vec{v}_\mathrm{D}, \vec{\theta})_{\partial\Omega^\mathrm{in}}.
\end{align*}
In $b_\mathrm{diff}$, the parameter $\delta$ is a scalar field that equals the constant one for smooth solid boundaries only and that otherwise corrects the numerical impact of the jaggedness of the solid boundaries obtained from micro-CT scanning. The derivation of this boundary condition and the wettability model can be found in \cite{frank2018energy}.
\par
For any $1 \leq n \leq \Nst$, our fully discrete scheme for solving the CHNS equations \eqref{eq:CHNS:model} is as follows. 
\begin{itemize}[leftmargin=0.5cm]
\item[] {\bf Algorithm CHNS.} At time $t^n$, given scalar functions $\phi^{n-1}_h,p^{n-1}_h,\psi^{n-1}_h$ in $X_h$ and vector functions $\vec{v}^{n-1}_h, \vec{w}^{n-1}_h$ in $\mathbf{X}_h$.
\item[] Step~1. Compute $\phi^n_h,\mu^n_h\in X_h$, such that for all $\chi_h\in X_h$,
\begin{align*}
(\phi_h^n,\chi_h) + \frac{\tau}{\Pe} a_\mathrm{diff}(\mathcal{M}(\phi_h^{n-1}); \mu_h^n,\chi_h) &+ \tau a_\mathrm{adv}(\phi_h^n,\vec{w}_h^{n-1},\chi_h) \\
&= (\phi_h^{n-1},\chi_h) + \tau (\phi_\mathrm{D}\vec{w}_h^{n-1}\cdot\normal_e,\chi_h)_{\partial{\Omega}^\mathrm{in}},\\
- (\mu_h^n,\chi_h) + \Cn^2 a_\mathrm{diff,in}(\phi_h^n,\chi_h) &+ (\Phi_{+}\,\!'(\phi_h^n),\chi_h) \\
&= \Cn^2 b_\mathrm{diff}(\phi_h^{n-1},\chi_h) - (\Phi_{-}\,\!'(\phi_h^{n-1}),\chi_h).
\end{align*}
\item[] Step~2. Compute $\vec{v}^n_h\in\mathbf{X}_h$, such that for all $\vec{\theta}_h\in\mathbf{X}_h$,
\begin{multline*}
(\vec{v}^n_h,\vec{\theta}_h) + \tau a_{\mathrm{conv}}(\vec{v}^{n-1}_h,\vec{v}^n_h,\vec{\theta}_h) + \frac{\tau}{\Rey} a_{\mathrm{ellip}}(\vec{v}^n_h,\vec{\theta}_h) = (\vec{v}^{n-1}_h,\vec{\theta}_h) \\
- \frac{\tau}{\Rey\Ca} b_{\mathrm{pres}}(p^{n-1}_h,\psi_h^{n-1},\vec{\theta}_h)
+ \frac{3\tau}{2\sqrt{2}\,\Rey\Ca\Cn} (\mu_h^n\grad{\phi_h^n},\vec{\theta}_h) + \tau b_{\mathrm{vel}}(\vec{\theta}_h).
\end{multline*}
\item[] Step~3. Compute $\psi^n_h\in X_h$, such that for all $\chi_h\in X_h$,
\begin{equation*}
a_{\mathrm{diff, out}}(\psi^n_h,\chi_h) = -\frac{\Rey\Ca}{\tau}(\div{\vec{v}_h^n},\chi_h).
\end{equation*}
\item[] Step~4. Compute $\vec{w}_h^n \in \mathbf{X}_h$ and $p_h^n \in X_h$, such that for all $\vec{\theta} \in \mathbf{X}_h$ and $\chi_h \in X_h$,
\begin{align*}
&(\vec{w}^n_h,\vec{\theta}_h) + \sigma_\mathrm{div} (\div{\vec{w}^n_h}, \div{\vec{\theta}_h}) = (\vec{v}^n_h,\vec{\theta}_h) - \frac{\tau}{\Rey\Ca}(\grad{\psi^n_h},\vec{\theta}_h),\\
&(p^n_h,\chi_h) = (p^{n-1}_h,\chi_h) + (\psi^n_h,\chi_h) - \sigma_\chi \Ca (\div{\vec{v}^n_h}, \chi_h).
\end{align*}
\end{itemize}
For the initial conditions, we set $p^0_h = \psi^0_h = 0$, $\vec{w}^0_h = \vec{v}^0_h$;  we compute $\phi^0_h$ from the $L^2$ projection of $\phi^0$ followed with Zhang--Shu limiter and we obtain $\vec{v}^0_h$ from the $L^2$ projection of $\vec{v}^0$.
\par
To obtain a bound-preserving discrete order parameter field, at each time step after finishing computing Step~1 in the Algorithm~CHNS, we apply the two-stage limiting strategy, see Section~\ref{sec-intro-limiter}, to postprocess discrete order parameter $\phi_h^n$.
For the simulations in Section~\ref{sec:experiments}, we choose $m=-1$ and $M=1$.

\section{Numerical experiments}\label{sec:experiments}
In this section, we first verify the high order accuracy of our cell average limiter \eqref{gDR-average} for a manufactured smooth solution. Then we verify the efficiency of the limiter \eqref{gDR-average} when using the parameters \eqref{opt-parameter} on some representative physical simulations including spinodal decomposition, flows in micro structure, and merging droplets.
\par
We use $\IP_2$ scheme, e.g., discontinuous piecewise quadratic polynomials for space approximation, on cubic partitions of 3D domains. More details can be found 
in  \cite{frank2018finite}.
\par
The penalty parameters for all tests are as follows. We use $\sigma = 8$ on $\Gammah$ for $a_\mathrm{diff}$; $\sigma = 16$ on $\partial{\Omega}$ for $a_\mathrm{diff,in}$ and $a_\mathrm{diff,out}$; $\sigma = 32$ on $\Gammah$ and $\sigma = 64$ on $\partial{\Omega}^\mathrm{in}$ for $a_\mathrm{ellip}$. In addition, we set tolerance $\epsilon = 10^{-13}$ to terminate Douglas--Rachford iterations.

\subsection{Accuracy test}
We use the manufactured solution method on domain $\Omega=(0,1)^3$ with end time $T=0.1$ to test the spatial order of convergence for our cell average limiter \eqref{gDR-average}. 
\par
To trigger the cell average limiter \eqref{gDR-average}, e.g., produce a fully discrete solution with cell average out of $[-1,1]$ at each time step, we use constant mobility with GL polynomial potential and choose the prescribed order parameter field as an expression of a cosine function to power eight, as follows: 
$
\phi = 1-2\cos^{8}{\big(t+{\textstyle\frac{2\pi}{3}}(x+y+z)\big)}.
$
The chemical potential $\mu$ is an intermediate variable, which value is derived by the order parameter $\phi$. 
The prescribed velocity and pressure fields are taken from the Beltrami flow \cite{masri2022discontinuous}, which enjoys the property that the nonlinear convection is balanced by the pressure gradient and the velocity is parallel to vorticity. We have
$
\vec{v} = 
\begin{bmatrix}
-e^{-t+x}\sin{(y+z)}-e^{-t+z}\cos{(x+y)}\\
-e^{-t+y}\sin{(x+z)}-e^{-t+x}\cos{(y+z)}\\
-e^{-t+z}\sin{(x+y)}-e^{-t+y}\cos{(x+z)}
\end{bmatrix}
$
and
$p = -e^{-2t}(e^{x+z}\sin{(y+z)}\cos{(x+y)}
+e^{x+y}\sin{(x+z)}\cos{(y+z)}
+e^{y+z}\sin{(x+y)}\cos{(x+z)}
+{\textstyle\frac{1}{2}}e^{2x} + {\textstyle\frac{1}{2}}e^{2y} + {\textstyle\frac{1}{2}}e^{2z}
- \overline{p^0})$,
where $\overline{p^0} = 7.63958172715414$ guarantees zero average pressure over $\Omega$ for any $t>0$ up to round-off error.
The initial conditions and Dirichlet boundary condition for velocity are imposed by above manufactured solutions. For order parameter and chemical potential, we apply Neumann boundary condition. In addition, the right-hand side terms is evaluated by the prescribed solution.
\par
Let us estimate the spatial rates of convergence by computing solutions on a sequence of uniformly refined meshes with fixed time step size $\tau = 10^{-4}$. In our experiments, the time step size is small enough such that the spatial error dominates. We choose $\Rey = 1$, $\Ca = 1$, $\Pe = 1$, $\Cn = 1$, and the contact angle $\vartheta=90^\circ$ on $\partial{\Omega}$.
If $\mathtt{err}_{h}$ denotes the error on a mesh with resolution $h$, then the rate is given by $\ln(\mathtt{err}_{h}/\mathtt{err}_{h/2})/\ln{2}$. 
\par
We compare the $L^2_h$ rate and the $L^\infty_h$ rate of order parameter in three scenarios: not applying any limiter, only applying the cell average limiter \eqref{gDR-average}, and applying both limiters \eqref{gDR-average} and \eqref{zhang-shu}.
In those applied cell average limiter \eqref{gDR-average} cases, the limiter is triggered at each time step, see Figure~\ref{fig:numerical_experiments:accuracy} for the ratio of the number of trouble cells to the number of total elements. 
The convergence of our original DG scheme without applying any limiter is optimal, see the top rows in Table~\ref{tab:space_rate_CH}. The middle and bottom rows in Table~\ref{tab:space_rate_CH} show optimal convergence of the cases that only apply cell average limiter \eqref{gDR-average} and apply both cell average limiter \eqref{gDR-average} and Zhang--Shu limiter \eqref{zhang-shu}. 
Our limiting strategy preserves high order accuracy.
\textcolor{black}{We emphasize that DG methods with only the Zhang-Shu limiter will produce cell averages outside of the range $[-1,1]$ for this particular test.}

\begin{figure}[htbp]
\centering
\begin{tabularx}{0.75\linewidth}{@{}c@{\hspace{1.5cm}}c@{}}
\includegraphics[width=0.3\textwidth]{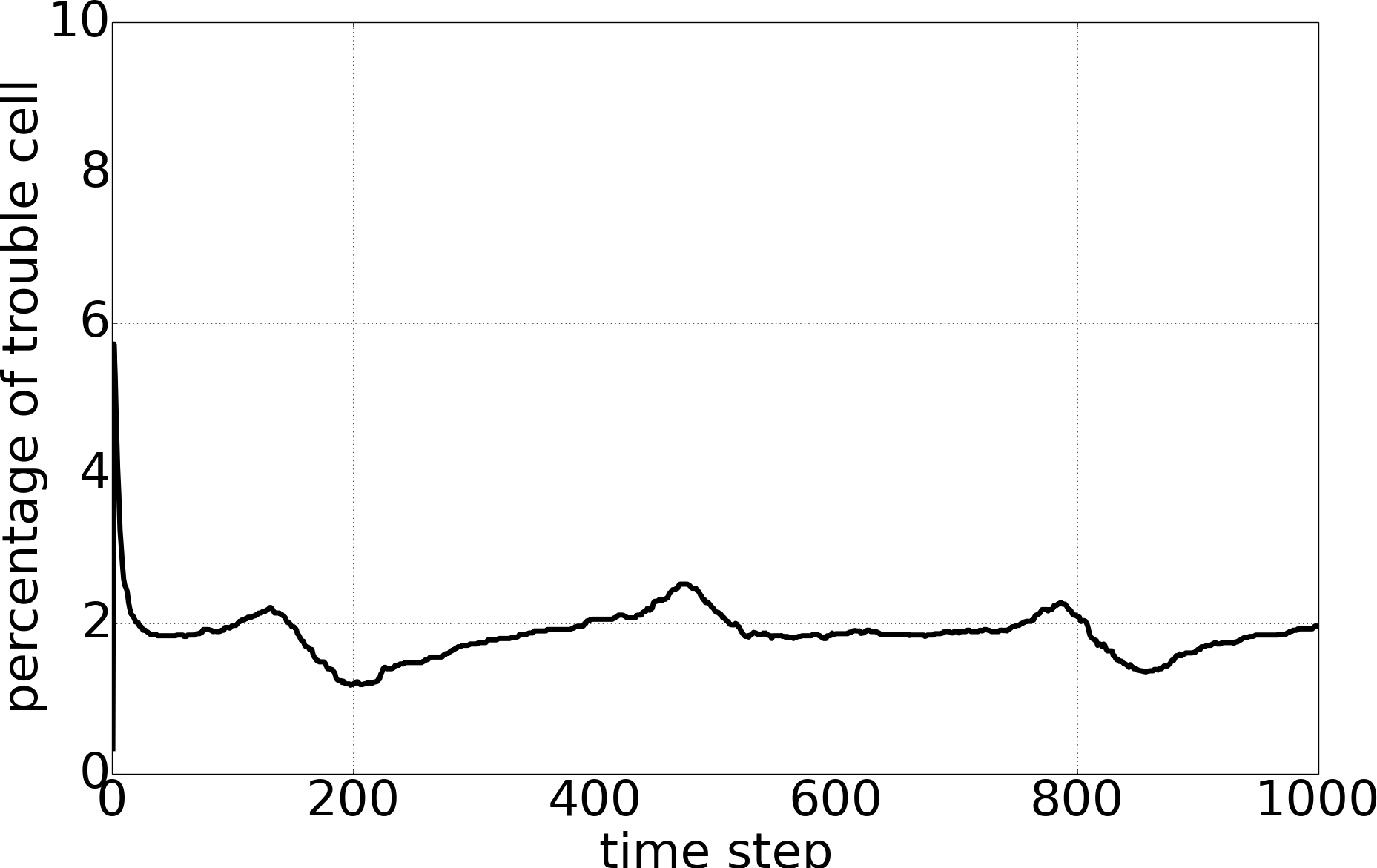} &
\includegraphics[width=0.3\textwidth]{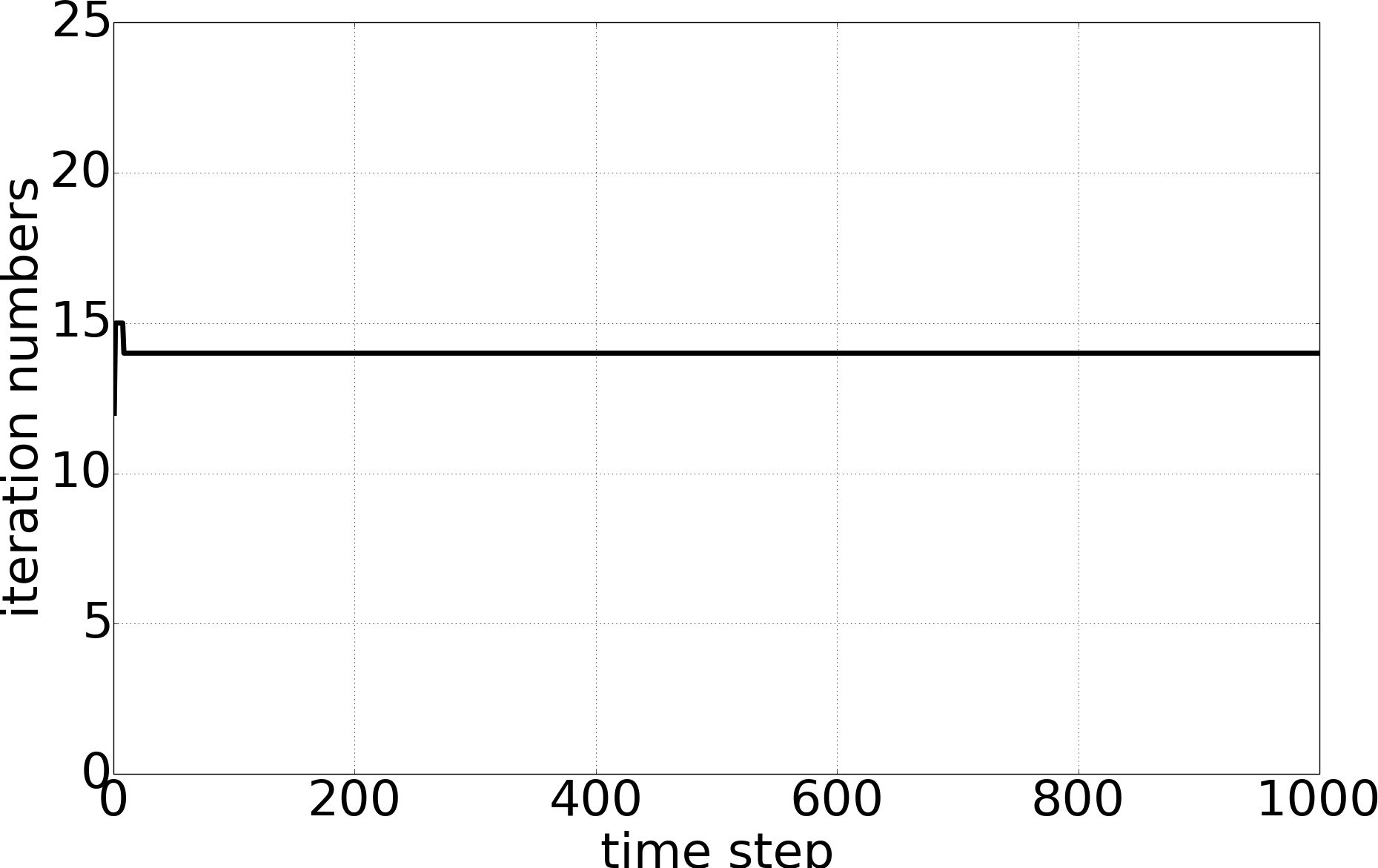} \\
\end{tabularx}
\caption{The performance of limiting strategy in the accuracy test of applying both limiters \eqref{gDR-average} and \eqref{zhang-shu} with mesh resolution $h = 1/2^5$. Left: the percentage of trouble cells at each time step for the cell average limiter \eqref{gDR-average}. Right: the number of Douglas--Rachford iterations at each time step. For each time step, at most 15 iterations are needed for \eqref{gDR2}}.
\label{fig:numerical_experiments:accuracy}
\end{figure}
\begin{table}[htbp]
\centering
\begin{tabularx}{0.95\linewidth}{@{~}c@{~~}|c@{~}|C@{\hspace{-1em}}C@{~}|C@{\hspace{-1em}}C@{~}}
\toprule
~ & $h$ & $\norm{\phi_h^{\Nst}\!-\phi(T)}{L^2_h}$ & rate & $\norm{\phi_h^{\Nst}\!-\phi(T)}{L^\infty_h}$ & rate \\
\midrule
\multirow{4}{*}{\begin{sideways}{no limiter$\hspace{0.25em}$}\end{sideways}} 
& $1/2^2$  & $2.034$\,E$-1$ & ---     & $5.636$\,E$-1$ & ---     \\
& $1/2^3$  & $4.903$\,E$-2$ & $2.053$ & $1.400$\,E$-1$ & $2.009$ \\
& $1/2^4$  & $5.714$\,E$-3$ & $3.101$ & $2.731$\,E$-2$ & $2.358$ \\
& $1/2^5$  & $4.833$\,E$-4$ & $3.564$ & $4.699$\,E$-3$ & $2.548$ \\
\midrule
\multirow{4}{*}{\begin{sideways}{DR$\hspace{0.25em}$}\end{sideways}} 
& $1/2^2$  & $2.053$\,E$-1$ & ---     & $5.826$\,E$-1$ & ---     \\
& $1/2^3$  & $4.954$\,E$-2$ & $2.051$ & $1.485$\,E$-1$ & $1.972$ \\
& $1/2^4$  & $5.720$\,E$-3$ & $3.115$ & $2.799$\,E$-2$ & $2.408$ \\
& $1/2^5$  & $4.834$\,E$-4$ & $3.565$ & $4.734$\,E$-3$ & $2.564$ \\
\midrule
\multirow{4}{*}{\begin{sideways}{DR$+$ZS$\hspace{0.25em}$}\end{sideways}} 
& $1/2^2$  & $2.872$\,E$-1$ & ---     & $7.631$\,E$-1$ & ---     \\
& $1/2^3$  & $5.970$\,E$-2$ & $2.266$ & $2.561$\,E$-1$ & $1.575$ \\
& $1/2^4$  & $7.181$\,E$-3$ & $3.057$ & $3.926$\,E$-2$ & $2.706$ \\
& $1/2^5$  & $4.833$\,E$-4$ & $3.893$ & $4.734$\,E$-3$ & $3.052$ \\
\bottomrule
\end{tabularx}
\caption{Errors and spatial convergence rates of order parameter. Top: the original DG scheme without applying any limiters. Middle: only apply the cell average limiter \eqref{gDR-average} (DR). Bottom: apply both of the cell average limiter \eqref{gDR-average} and Zhang--Shu limiter \eqref{zhang-shu}.}
\label{tab:space_rate_CH}
\end{table}

\vspace{-0.1cm}
\subsection{Spinodal decomposition}
Spinodal decomposition is a phase separation mechanism, by which an initially thermodynamically unstable homogeneous mixture spontaneously decomposes into two separated phases that are more thermodynamically favorable. The spinodal decomposition test is a widely used benchmark for validating CHNS simulators. 
In this part, we employ the degenerate mobility with GL polynomial potential. 
\par
We define a trefoil-shaped pipe, which is a set of points whose distance away from  the following parametric curve is less than $0.09$.
A trefoil knot: $x(t) = \frac{1}{8}(\cos{t}+2\cos{2t}) + \frac{1}{2}$, $y(t) = \frac{1}{8}(\sin{t}-2\sin{2t}) + \frac{1}{2}$, and $z(z) = \frac{1}{4} \sin{3t} + \frac{1}{2}$, where $t\in [0,2\pi]$.
Let us uniformly partition the unit cube $(0,1)^3$ into cubic cells with the mesh resolution $h = 1/100$. A cell is marked as fluid if its center is in the above pipe, otherwise is marked as solid. The computational domain $\Omega$ is defined as the union of all fluid cells.
We consider a closed system, i.e., $\partial{\Omega} = \partial{\Omega}^\mathrm{wall}$. The initial order parameter field is generated by sampling numbers from a discrete uniform distribution, $c^0|_{E_i} \sim \mathrm{U}\{-1,1\}$, and the initial velocity field is taken to be zero.
We take the time step size $\tau=1\times10^{-3}$. For physical parameters, we choose $\Rey = 1$, $\Ca = 0.1$, $\Pe = 1$, $\Cn = h$, and the contact angle $\vartheta=90^\circ$ on $\partial{\Omega}$. 
\par
Figure~\ref{fig:numerical_experiments:trefoil_c} shows snapshots of the order parameter field. We employ a rainbow color scale that maps the values in $[-1,1]$ from transparent blue to non-transparent red for plotting the order parameter field. The center of the diffusive interface is colored green. We observe that the homogeneous mixture decomposes into two separate phases. With a neutral wall, i.e., the contact angle $\vartheta=90^\circ$, in the final stage of the simulation, each of the two phases occupies several disjoint sections of the domain. The interfaces are perpendicular to the solid surface. Our limiters remove overshoots and undershoots. The global mass is conserved, see the left subfigure of Figure~\ref{fig:numerical_experiments:trefoil_iter}.
\par
The middle subfigure of Figure~\ref{fig:numerical_experiments:trefoil_iter} records the number of iterations of the Douglas--Rachford algorithm on each time step. 
To measure the convergence rate, we run the Douglas--Rachford algorithm for $10^3$ iterations with a very small tolerance to approximate $\vec{y}^\ast$ and $\vec{x}^\ast$ numerically. Then we plot $\norm{\vec{y}^k-\vec{y}^\ast}{2}$ and $\norm{\vec{x}^k-\vec{x}^\ast}{2}$. 
The right subfigure of Figure~\ref{fig:numerical_experiments:trefoil_iter} shows asymptotic linear convergence rates at the selected time step $128$. We see the convergence rates match our analysis in Theorem \ref{thm-rate}. 
In addition, we check the convergence rates on all of the rest steps that match with our analysis.
\begin{figure}[htbp]
\centering
\begin{tabularx}{\linewidth}{@{}c@{~}c@{~}c@{~}c@{~}c@{}}
\includegraphics[width=0.19\textwidth]{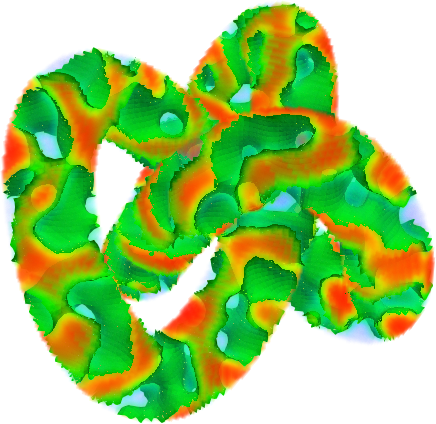} & 
\includegraphics[width=0.19\textwidth]{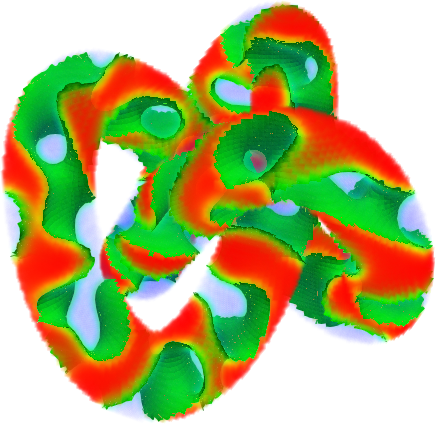} & 
\includegraphics[width=0.19\textwidth]{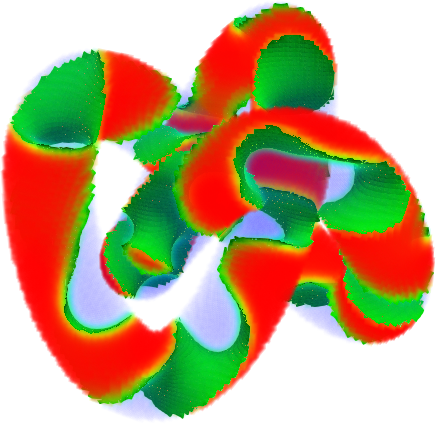} &
\includegraphics[width=0.19\textwidth]{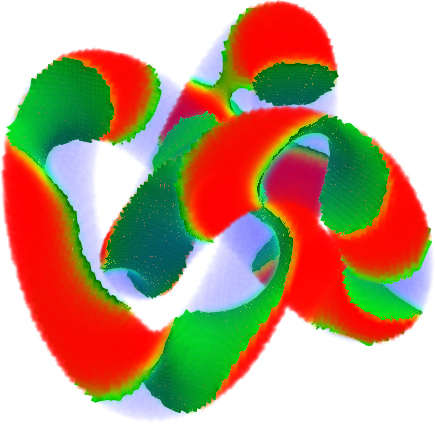} &
\includegraphics[width=0.19\textwidth]{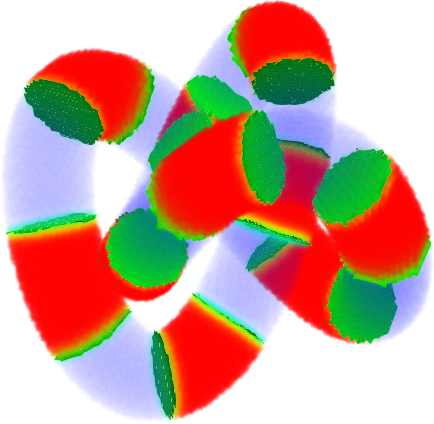} \\
\end{tabularx}
\caption{Selected snapshots at time steps $2^n$, where $n=3, 5, \cdots, 11$. 3D views of the evolution of order parameter field.}
\label{fig:numerical_experiments:trefoil_c}
\end{figure}
\begin{figure}[htbp]
\centering
\begin{tabularx}{0.9\linewidth}{@{}c@{~}c@{~}c@{}}
\includegraphics[width=0.3125\textwidth]{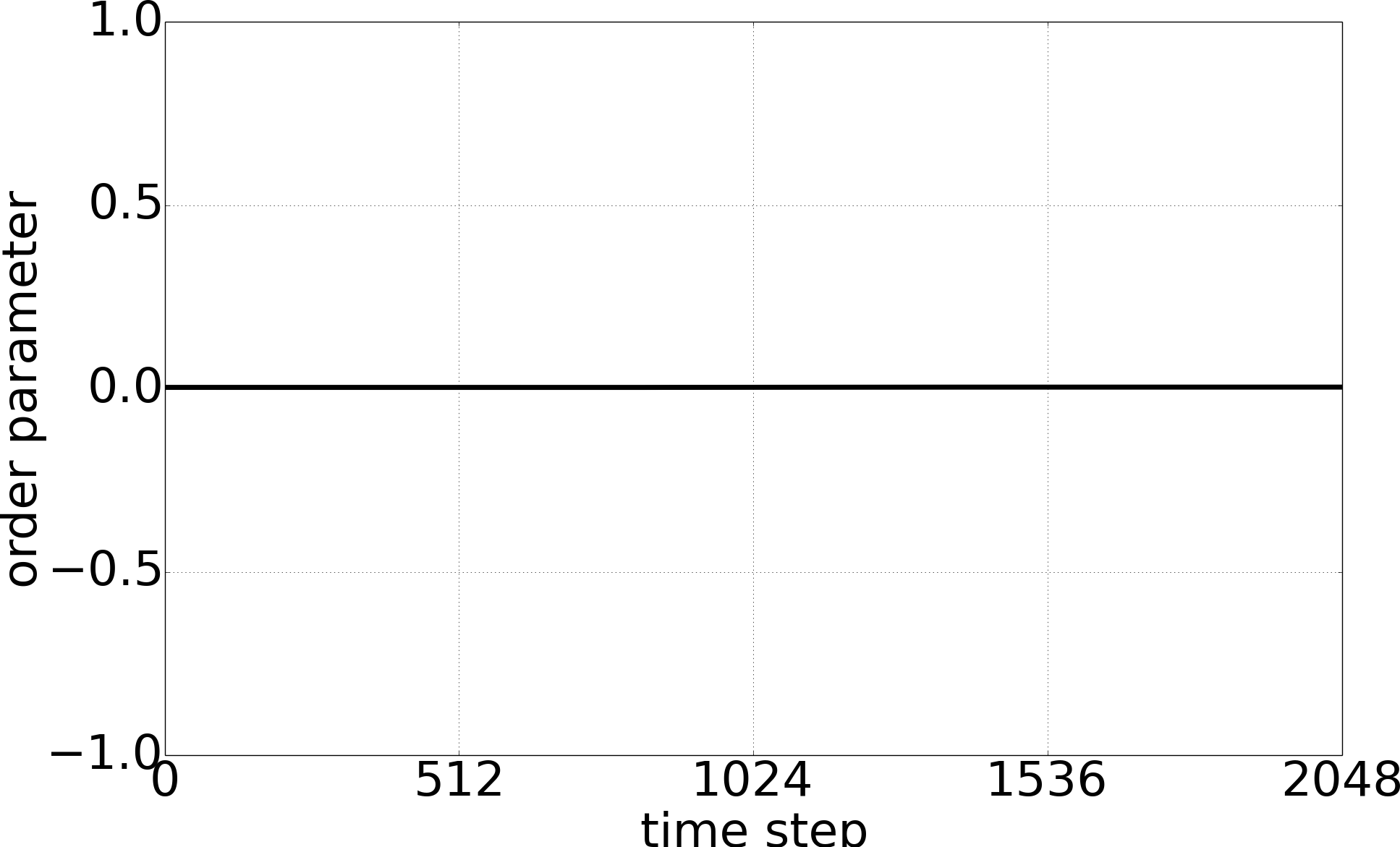} &
\includegraphics[width=0.3\textwidth]{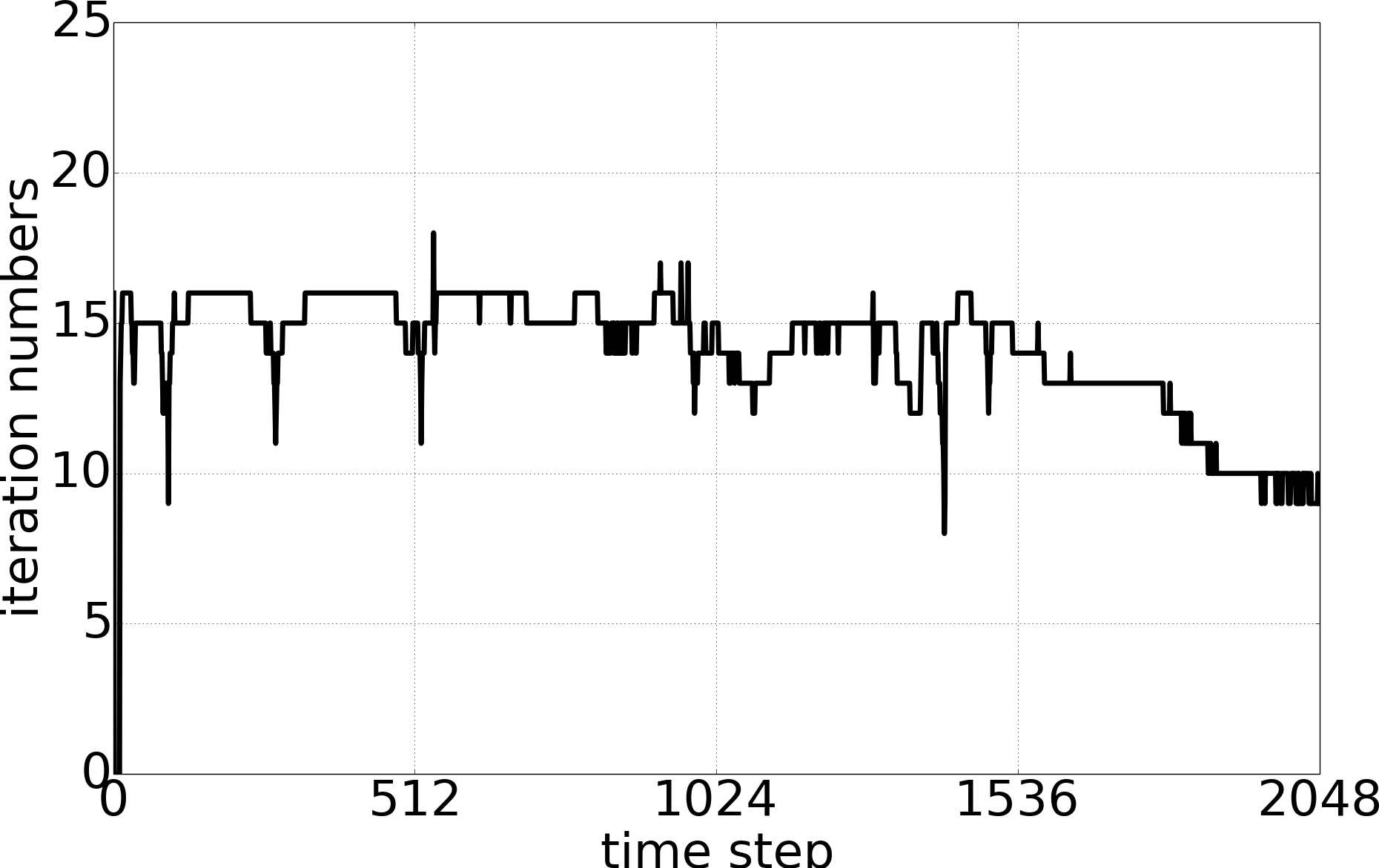} &
\includegraphics[width=0.3\textwidth]{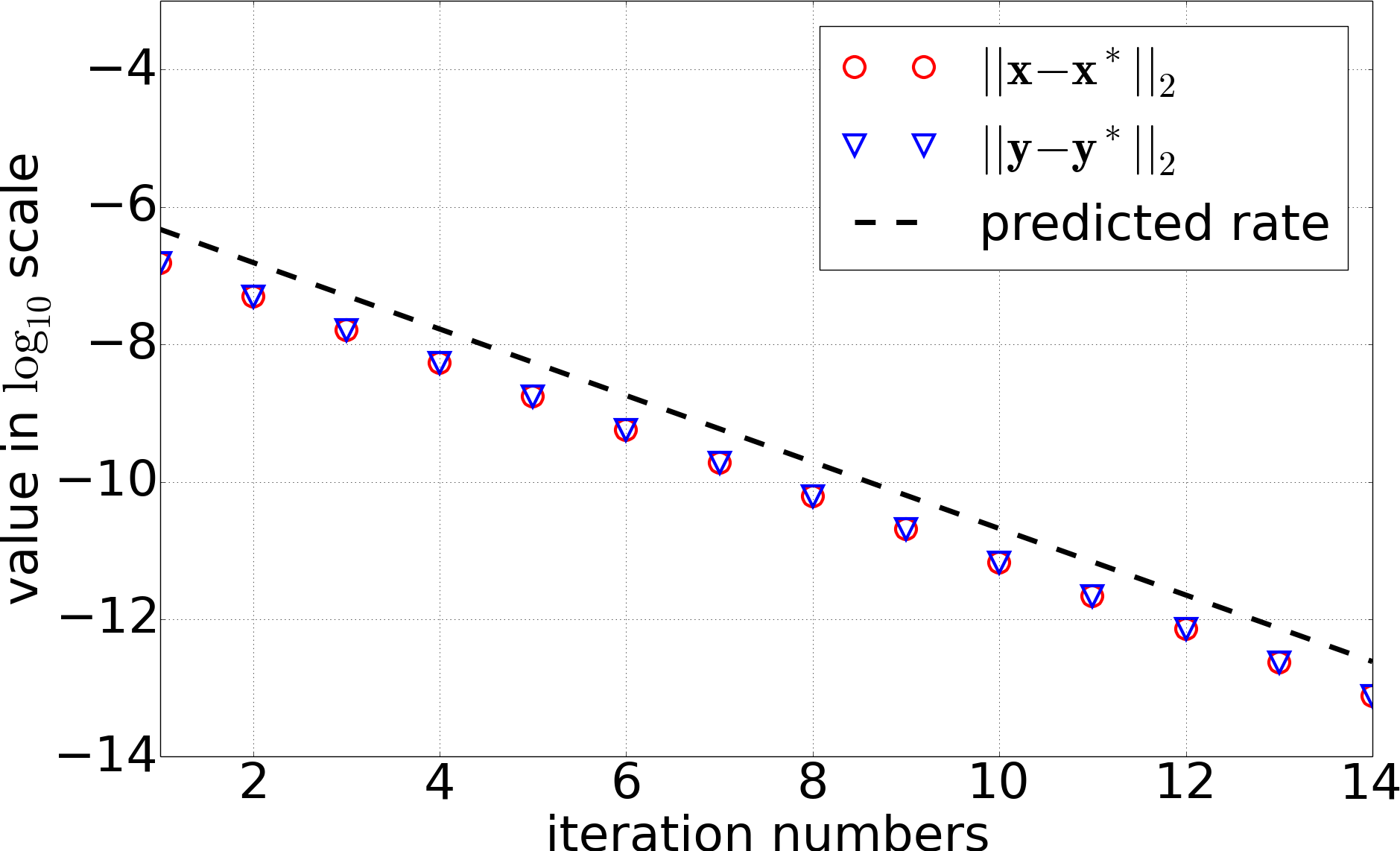} \\
\end{tabularx}
\caption{Left: the average of order parameter at each time step, which shows the conservation is preserved. Middle: the number of Douglas--Rachford iterations at each time step. Right: the asymptotic linear convergence at time step $128$. The predicted rate is the rate given in Theorem \ref{thm-rate}.}
\label{fig:numerical_experiments:trefoil_iter}
\end{figure}

\subsection{Micro structure simulations}
This example involves large P\'eclet flows in a microfluidic device, making it an interesting test for validating our bound-preserving scheme in simulating advection-dominated CHNS problems.
In this part, we use the constant mobility with GL polynomial potential.
\par
The microstructure image is a set of $334\times 210\times 10$ cubic cells of resolution $h = 1/350$. Analogous to the lab experiment setup, we add a buffer of $16\times 210\times 70$ cells to the left side. The pore space together with the buffer region form our computational domain $\Omega$, see Figure~\ref{fig:numerical_experiments:micro_geometry}. 
We refer to phase $\mathrm{A}$ the bulk phase with order parameter equals to $+1$ and phase $\mathrm{B}$ the bulk phase with order parameter equals to $-1$. The buffer zone is initially filled with phase $\mathrm{A}$ and the microstructure is initially filled with phase $\mathrm{B}$, respectively. The initial velocity field is taken to be zero.
The left boundary of $\Omega$ is inflow, the right boundary of $\Omega$ is outflow, and the rest boundaries of $\Omega$ are fluid--solid interfaces. On the inflow boundary, we prescribe $\phi_\mathrm{D} = 1$, e.g., the phase $\mathrm{A}$ is injected, and
$\vec{v}_\mathrm{D} = \frac{10000}{9}(y-0.2)(y-0.8)(z-0.4)(z-0.6)$. 
We the take time step size $\tau=5\times10^{-4}$. For physical parameters, we choose $\Rey = 1$, $\Ca = 1$, $\Pe = 100$, and $\Cn = h$. The microstructure surface is hydrophobic with respect to phase $\mathrm{A}$ with a contact angle $\vartheta=135^\circ$. The buffer surface and outflow boundary are neutral, namely $\vartheta=90^\circ$. 
\par
Figure~\ref{fig:numerical_experiments:micro_c} shows snapshots of the order parameter field as well as its values along the plane $\{(x,y,z)\in\Omega: z=0.5\}$ in mountain views. Similar to the previous example, we employ a rainbow color scale that maps the values in $[-1,1]$ from blue to red for plotting the order parameter field. The center of the diffusive interface is colored green. The values outside $[-1,1]$ are marked in black. We observe that phase $\mathrm{A}$ invades the microstructure while staying away from the solid surfaces due to the wettability constraint.
The top two rows correspond to the simulation without applying any limiter whereas the bottom two rows correspond to the simulation applying our two-stage limiting strategy.  
Our limiters remove overshoot and undershoot. The fluid dynamics are similar for both cases.
\par
Figure~\ref{fig:numerical_experiments:micro_DR_exact_r} shows the number of iterations of the Douglas--Rachford algorithm on each time step as well as the asymptotic linear convergence rates of selected time steps.
Here, the errors $\norm{\vec{y}^k-\vec{y}^\ast}{2}$ and $\norm{\vec{x}^k-\vec{x}^\ast}{2}$ are measured in a similar way as explained in the previous example.
A numerical way of getting an exact value of $r$ is to run the Douglas--Rachford iterations sufficiently many times with small enough tolerance and count the number of entries that stay out of the bounds in $\vec{y}^\ast$. Using the exact $r$ to compute the principal angle $\theta_{N-r}$, the numerical results match our analysis, see Figure~\ref{fig:numerical_experiments:micro_DR_exact_r}.
\begin{figure}[htbp]
\centering
\includegraphics[width=0.3\textwidth]{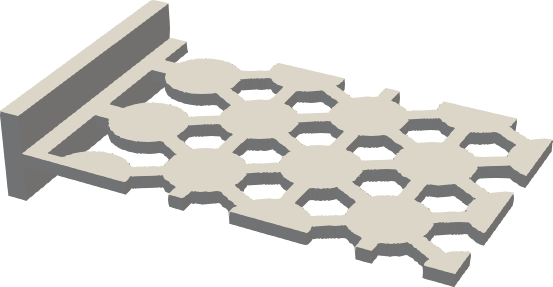}
\caption{The computational domain of the microstructure simulation.}
\label{fig:numerical_experiments:micro_geometry}
\end{figure}
\begin{figure}[htbp]
\centering
\begin{tabularx}{\linewidth}{@{}c@{~}c@{~}c@{~}c@{~}c@{}}
\includegraphics[width=0.19\textwidth]{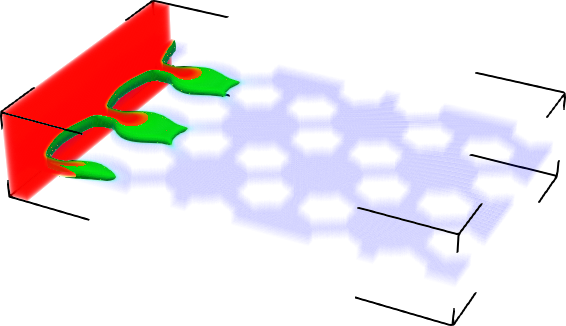} &
\includegraphics[width=0.19\textwidth]{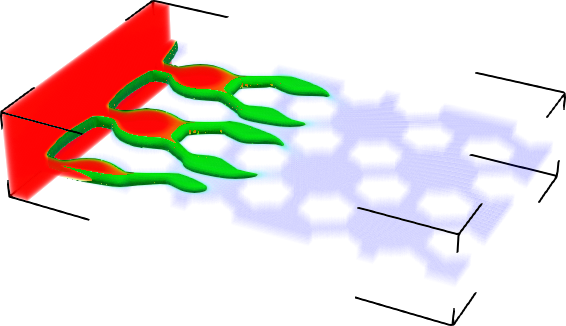} & 
\includegraphics[width=0.19\textwidth]{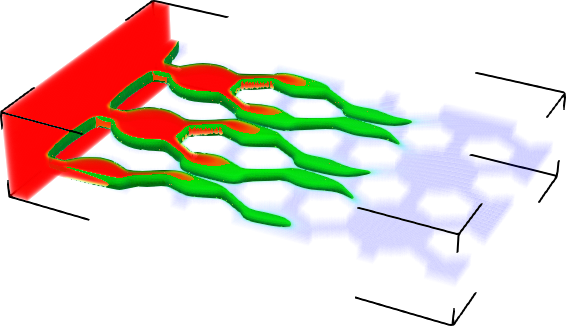} &
\includegraphics[width=0.19\textwidth]{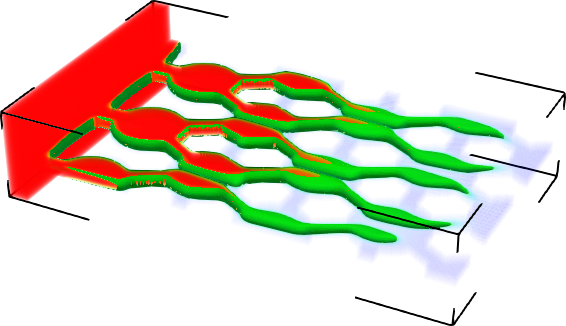} & 
\includegraphics[width=0.19\textwidth]{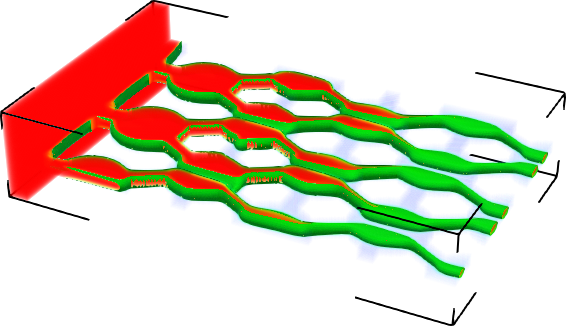} \\
\includegraphics[width=0.19\textwidth]{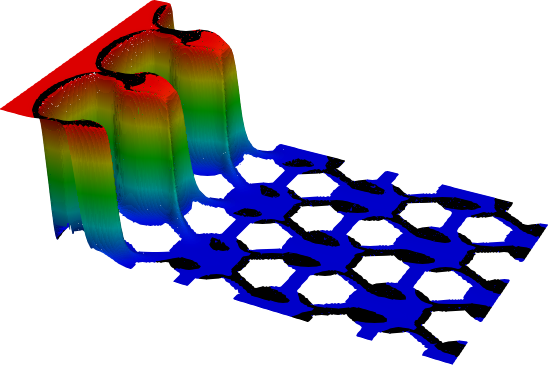} &
\includegraphics[width=0.19\textwidth]{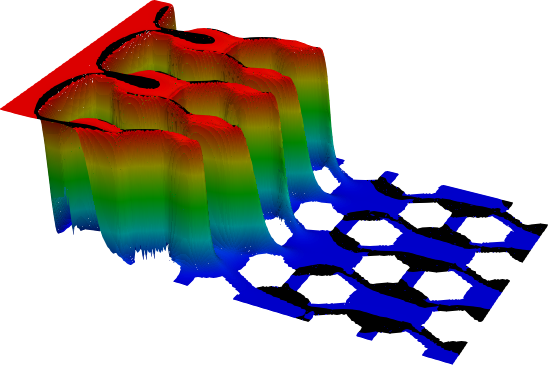} & 
\includegraphics[width=0.19\textwidth]{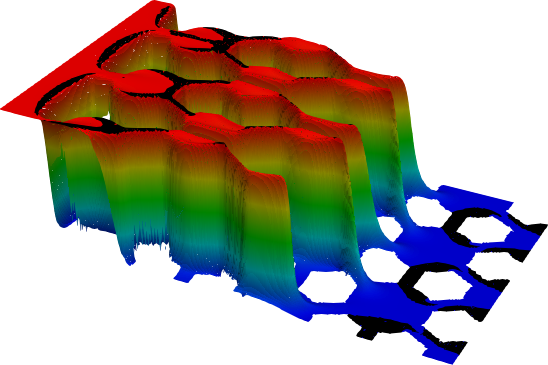} &
\includegraphics[width=0.19\textwidth]{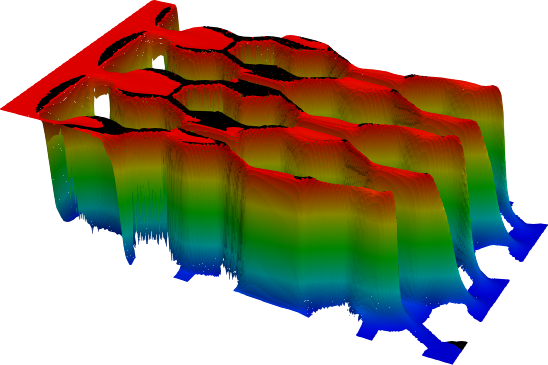} & 
\includegraphics[width=0.19\textwidth]{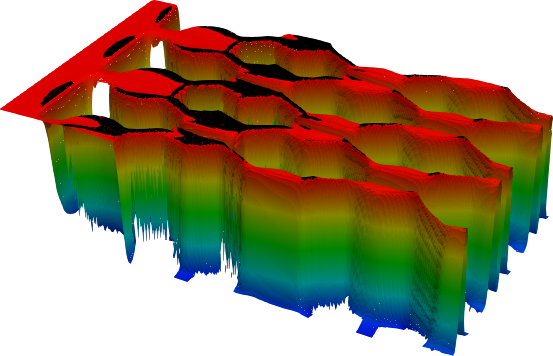} \\
\includegraphics[width=0.19\textwidth]{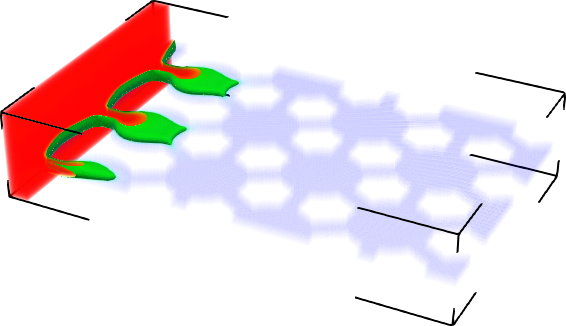} &
\includegraphics[width=0.19\textwidth]{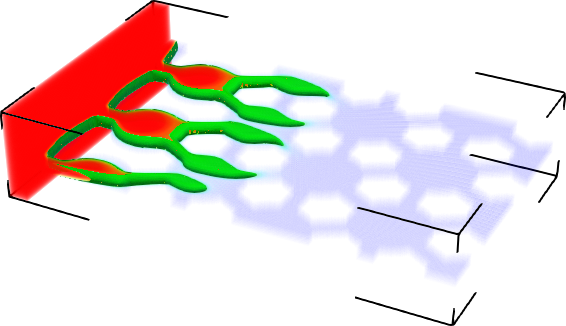} & 
\includegraphics[width=0.19\textwidth]{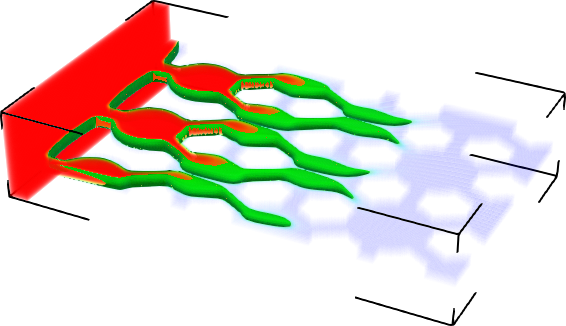} &
\includegraphics[width=0.19\textwidth]{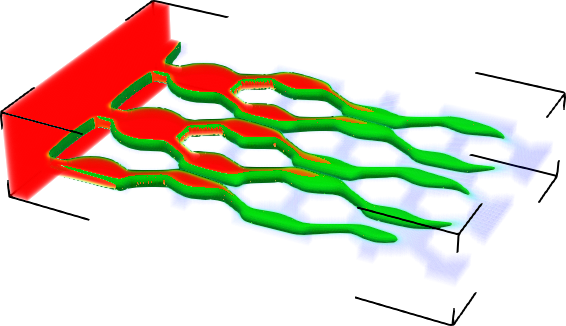} & 
\includegraphics[width=0.19\textwidth]{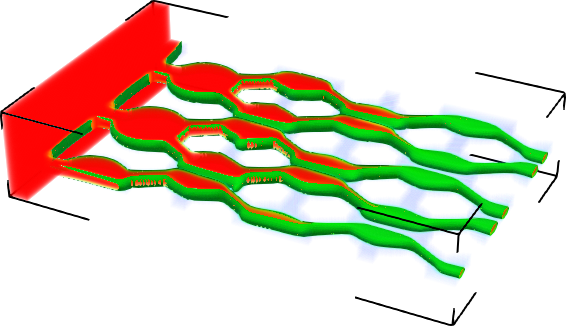} \\
\includegraphics[width=0.19\textwidth]{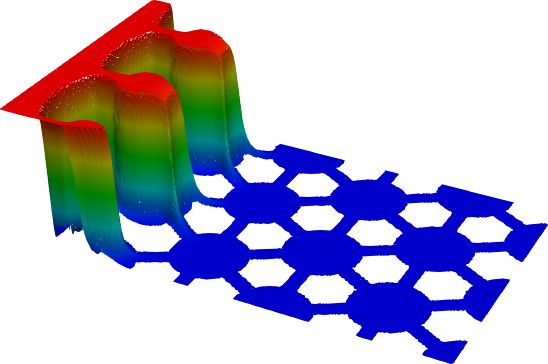} &
\includegraphics[width=0.19\textwidth]{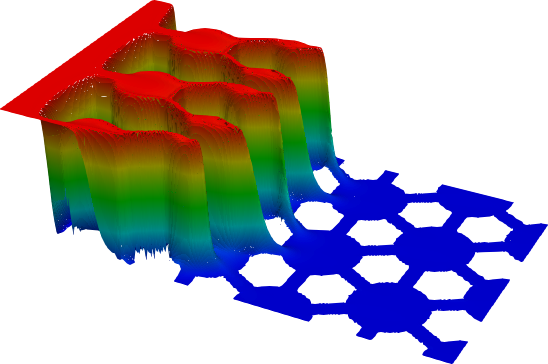} & 
\includegraphics[width=0.19\textwidth]{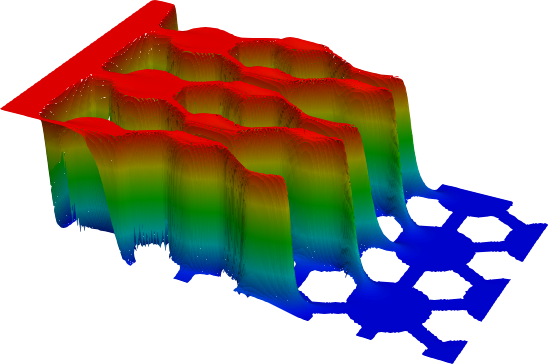} &
\includegraphics[width=0.19\textwidth]{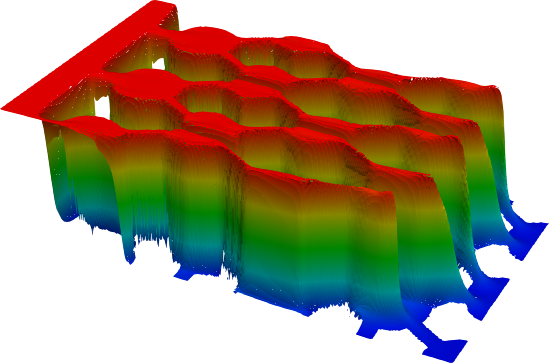} & 
\includegraphics[width=0.19\textwidth]{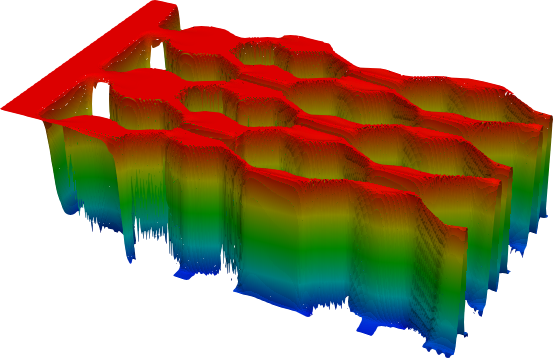} \\
\end{tabularx}
\caption{Selected snapshots at time steps $50$, $100$, $150$, $200$, and $250$. The first and third rows: 3D views of the evolution of the order parameter field. The second and fourth rows: plots of order parameter warped along the plane $\{z=0.5\}$. The top two rows are without limiters and the bottom two rows are with our limiters.}
\label{fig:numerical_experiments:micro_c}
\end{figure}
\begin{figure}[htbp]
\centering
\begin{tabularx}{0.9\linewidth}{@{}c@{~}c@{~}c@{}}
\includegraphics[width=0.3\textwidth]{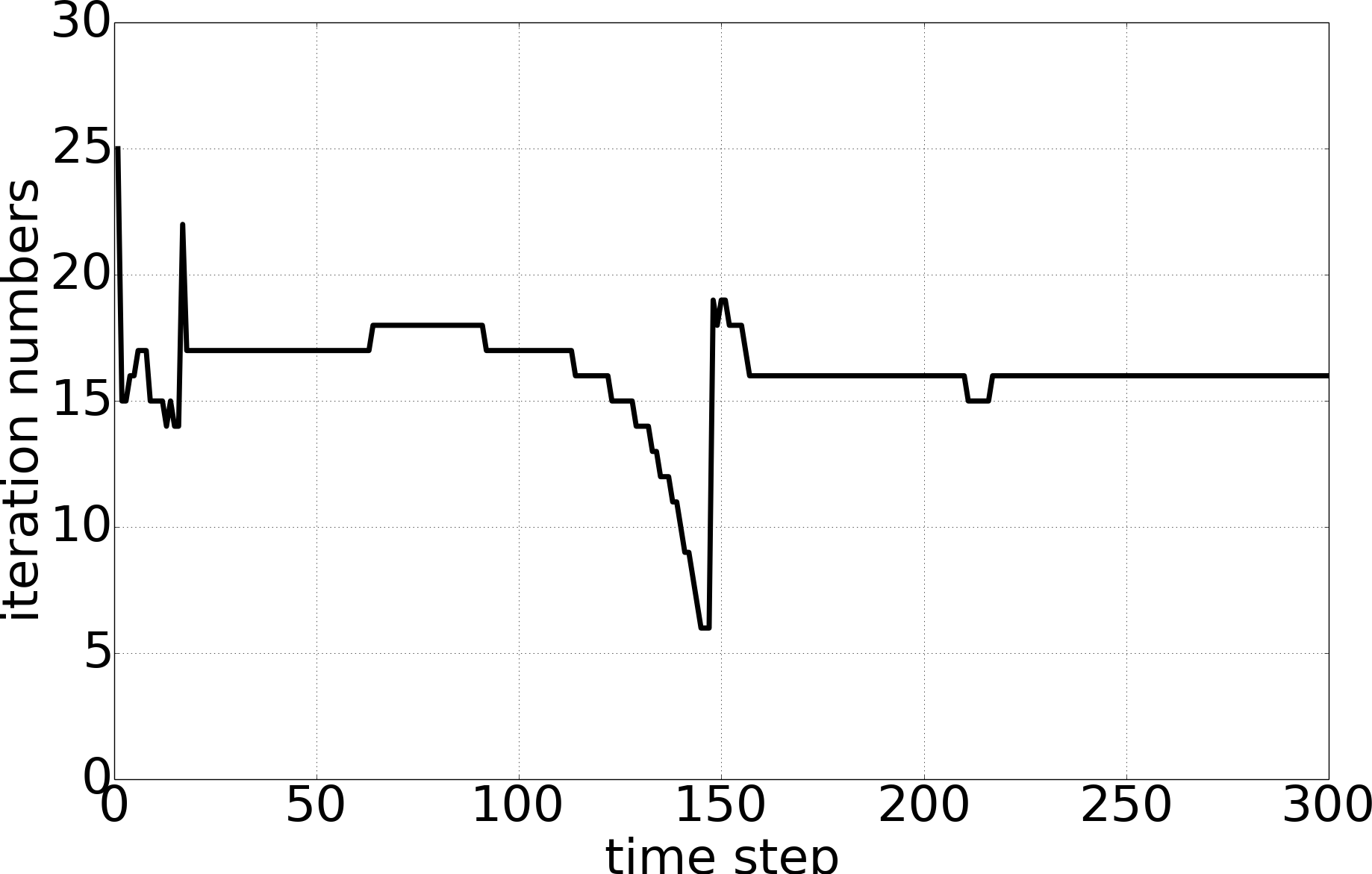} &
\includegraphics[width=0.3\textwidth]{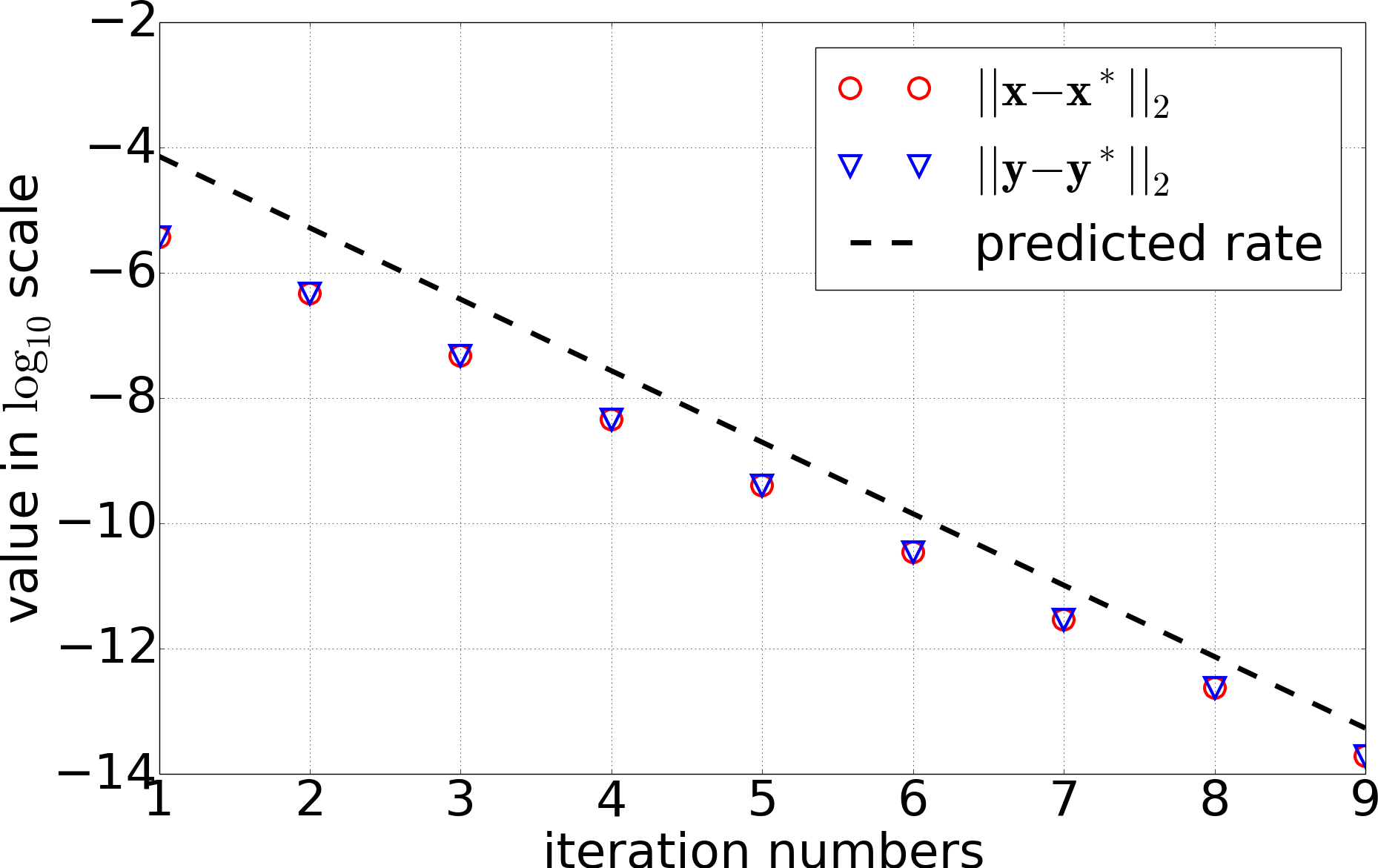} &
\includegraphics[width=0.3\textwidth]{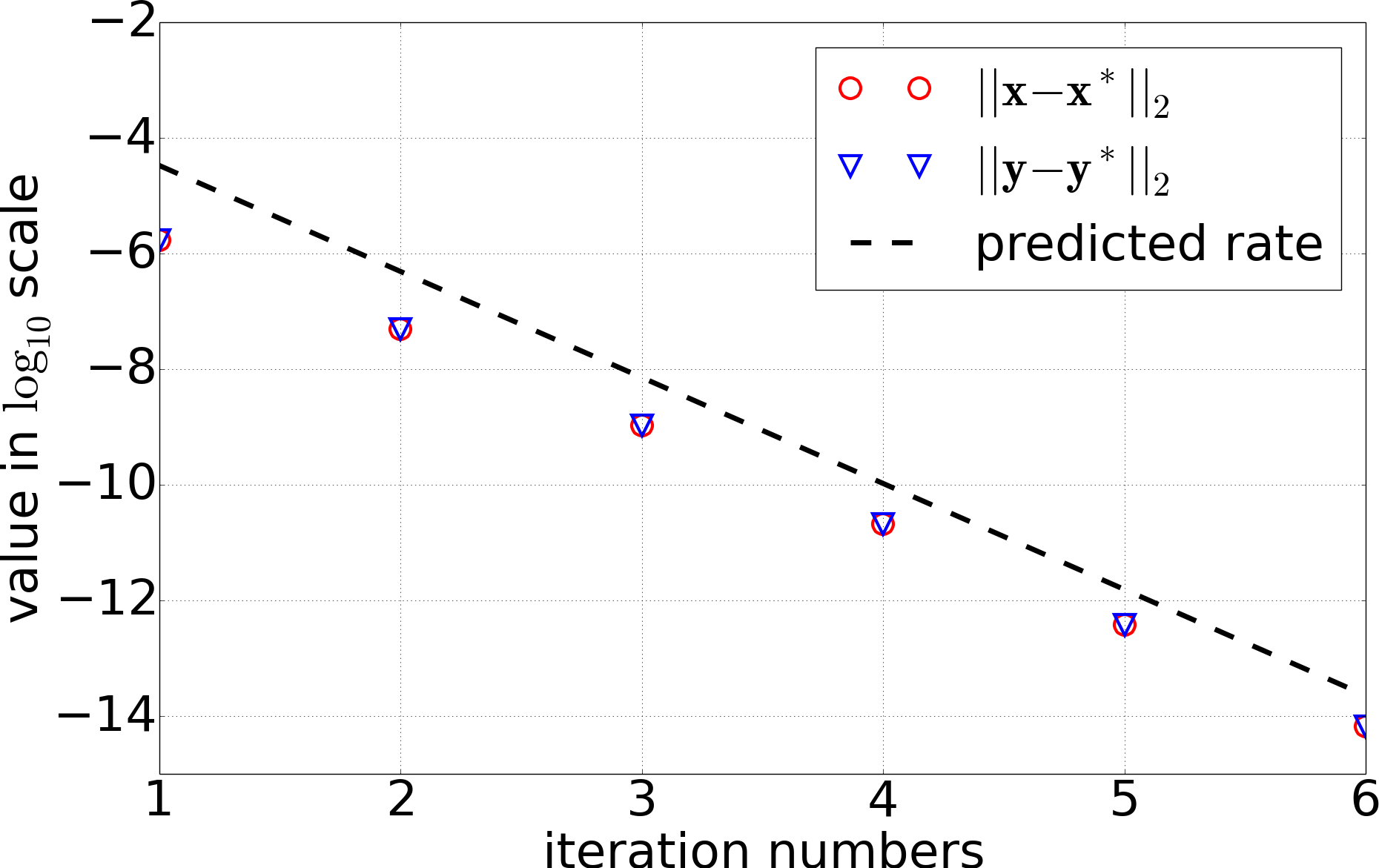} \\ 
\end{tabularx}
\caption{The left top figure shows the number of Douglas--Rachford iterations at each time step. The middle and right figures show the asymptotic linear convergence at time steps $150$ and $250$, where the principal angle $\theta_{N-r}$ is computed by using exact values of $r$.}
\label{fig:numerical_experiments:micro_DR_exact_r}
\end{figure}

\subsection{Merging droplets}
This example deals with droplets of fluid surrounded by another fluid. In a capillary-forces-dominated merging process, the large droplet wobbles several times and eventually evolves into the most thermodynamically favorable configuration, e.g., a single spherical droplet. 
\par
Let us consider four different scenarios. In the first scenario, we use constant mobility with GL polynomial potential and we do not apply any limiter. In the rest scenarios, we apply our two-stage limiting strategy. In the second scenario, we use constant mobility with GL polynomial potential. In the third scenario, we use constant mobility with FH logarithmic potential (parameters $\alpha=0.3$ and $\beta = 1$). And in the fourth scenario, we use degenerate mobility with GL polynomial potential.
\par
Let $\Omega = (0,1)^3$ to be a closed system, $\partial{\Omega} = \partial{\Omega}^\mathrm{wall}$ and
set the initial velocity field $\vec{v}^0 = \vec{0}$. Four droplets of phase $\mathrm{A}$ are initially in a non-equilibrium configuration, surrounded by phase $\mathrm{B}$, i.e., the initial order parameter field is prescribed by 
\[
\resizebox{.99\hsize}{!}{
$\phi^{0} = 
\max\Big\{-1,\,
\tanh{\Big(\frac{r_1 - \norm{\vec{x}-\vec{a}_0}{}}{\sqrt{2}\,\Cn}\Big)},\,
\tanh{\Big(\frac{r_1 - \norm{\vec{x}-\vec{a}_1}{}}{\sqrt{2}\,\Cn}\Big)},\,
\tanh{\Big(\frac{r_2 - \norm{\vec{x}-\vec{a}_2}{}}{\sqrt{2}\,\Cn}\Big)},\,
\tanh{\Big(\frac{r_2 - \norm{\vec{x}-\vec{a}_3}{}}{\sqrt{2}\,\Cn}\Big)}
\Big\},$
}
\]
where $\vec{a}_0 = \transpose{[0.35, 0.35, 0.35]}$ and $\vec{a}_1 = \transpose{[0.65, 0.65, 0.65]}$ are the centers of the two initial larger droplets with radius $r_1 = 0.25$; and $\vec{a}_2 = \transpose{[0.75, 0.25, 0.25]}$ and $\vec{a}_3 = \transpose{[0.25, 0.75, 0.75]}$
are the centers of the two initial smaller droplets with radius $r_2 = 0.16$.
For the FH logarithmic potential, we use $0.997\phi^0$ as the initial order parameter field to make its value away from the singularity.
We uniformly partition domain $\Omega$ by cubic elements with the mesh resolution $h = 1/50$ and take the time step size $\tau = 10^{-4}$. For physical parameters, we choose $\Rey = 1$, $\Ca = 10^{-4}$, $\Pe = 1$, $\Cn = h$, and the contact angle $\vartheta = 90^\circ$ on $\partial{\Omega}$.
\par
Figure~\ref{fig:numerical_experiments:drop_c1} shows snapshots of the order parameter field. The center of the diffusive interface is colored green and the bulk phases are colored transparent. We see the merging of the four droplets, the intermediate wobbling stages, and the final equilibrium configuration of a spherical droplet. 
We observe from Figure~\ref{fig:numerical_experiments:drop_c1} that the  fluid dynamics are visually similar in these scenarios. However, there are visible differences in certain one dimensional profiles, see Figure~\ref{fig:numerical_experiments:drop_c2}  for the order parameters at the line  $\{(x,y,z)\in \Omega: x = y = z\}$.
\par
Figure~\ref{fig:numerical_experiments:drop_c2} shows values of order parameter along the diagonal $\{(x,y,z)\in\Omega:x=y=z\}$ of the computational domain. In scenario 1, we observe bulk shift at near steady state, which is as expected since no limiters are applied. In secnarios 2 and 4, our limiters remove overshoots and undershoots. In scenario 3, the FH logarithmic potential ensures bounds without bulk shift. The cell average limiter \eqref{gDR-average} is not triggered but the Zhang--Shu limiter is triggered. 
The global mass is conserved, see the left subfigure in Figure~\ref{fig:numerical_experiments:drop_iter_and_mass}. 
\par
We plot the number of iterations of the Douglas--Rachford algorithm on each time step, see the right two subfigures in Figure~\ref{fig:numerical_experiments:drop_iter_and_mass}. 
We check the asymptotic linear convergence rates and they match with our analysis.
The errors $\norm{\vec{y}^k-\vec{y}^\ast}{2}$ and $\norm{\vec{x}^k-\vec{x}^\ast}{2}$ are measured in a similar way as in the previous example. 
\begin{figure}[htbp]
\centering
\begin{tabularx}{0.95\linewidth}{@{}c@{~~}c@{~}c@{~}c@{~}c@{~}c@{}}
\multirow{2}{*}{\begin{sideways}{\footnotesize constant mobility $+$ GL potential $\hspace{-1.85cm}$}\end{sideways}} &
\includegraphics[width=0.17\textwidth]{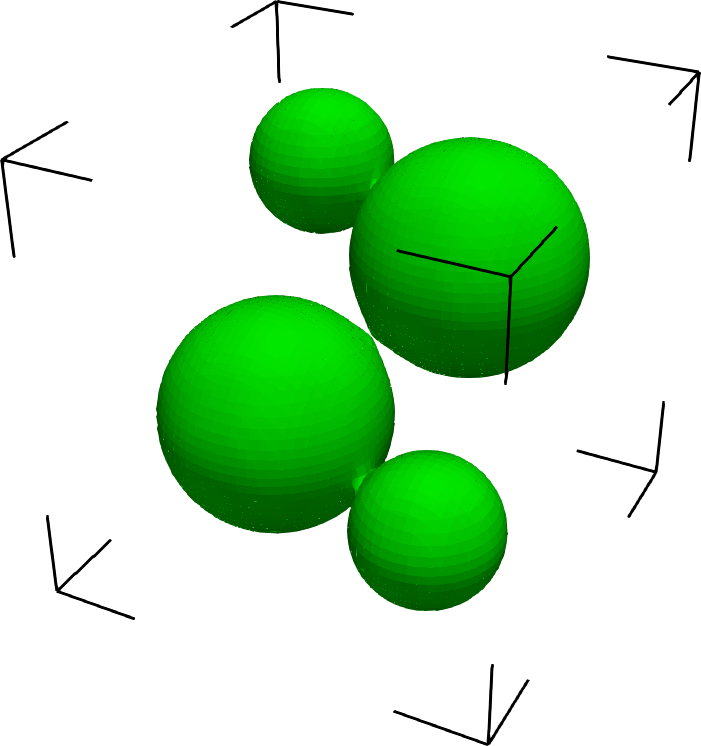} &
\includegraphics[width=0.17\textwidth]{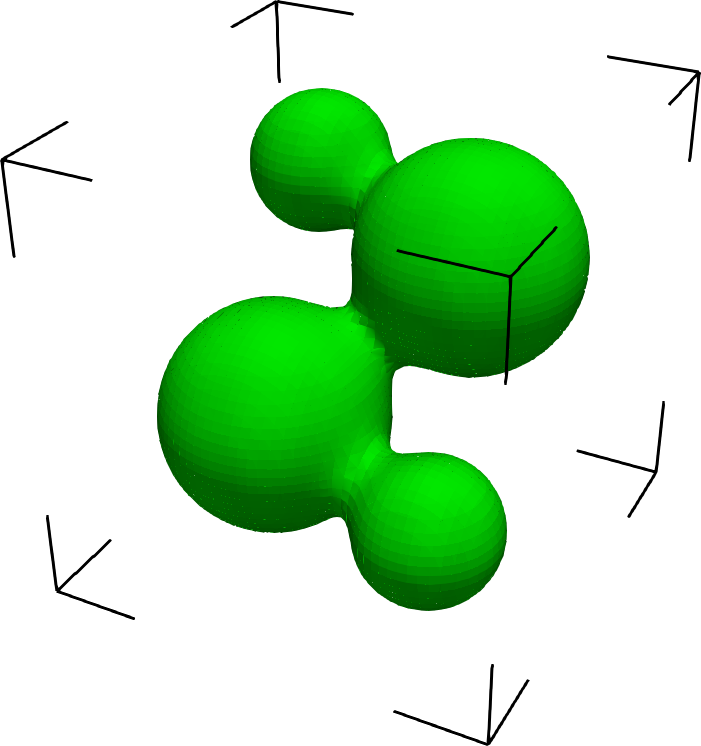} & 
\includegraphics[width=0.17\textwidth]{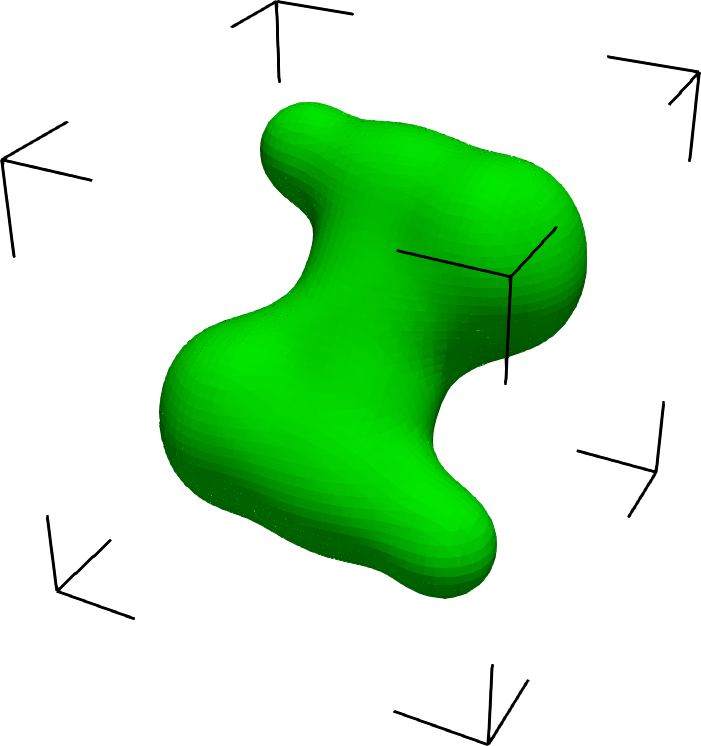} &
\includegraphics[width=0.17\textwidth]{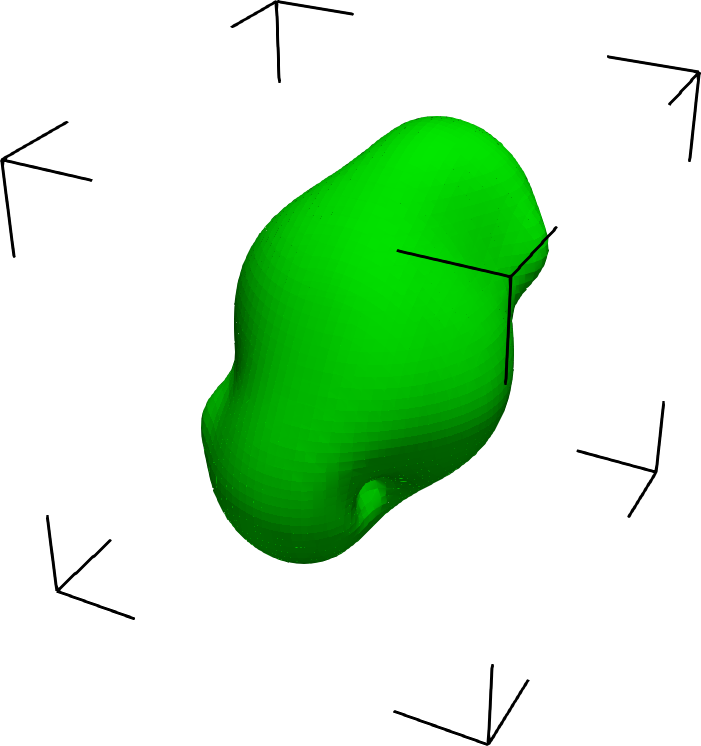} & 
\includegraphics[width=0.17\textwidth]{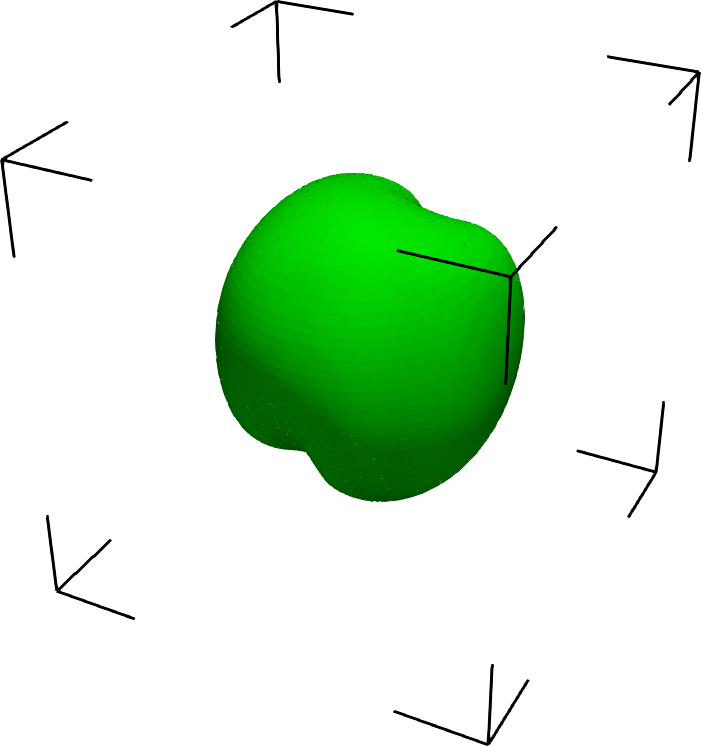} \\ &
\includegraphics[width=0.17\textwidth]{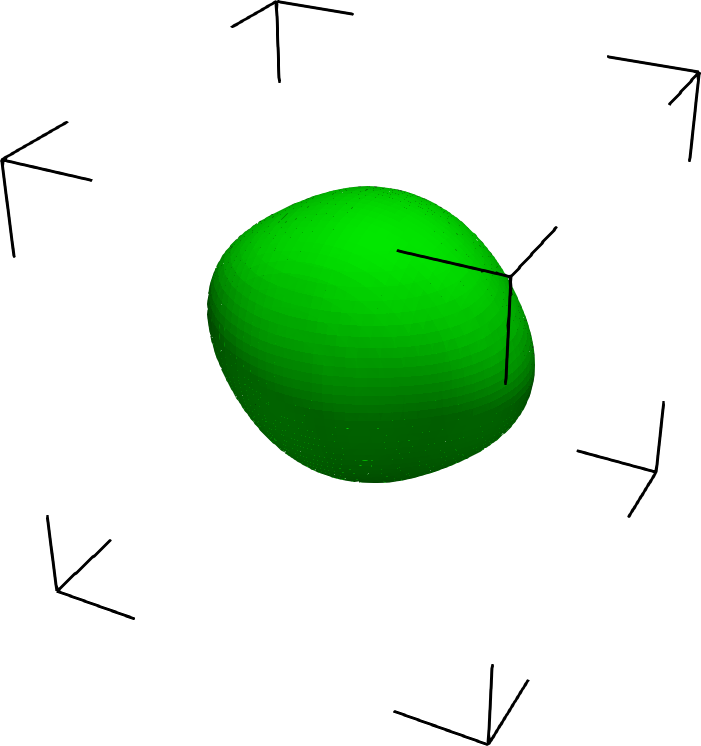} &
\includegraphics[width=0.17\textwidth]{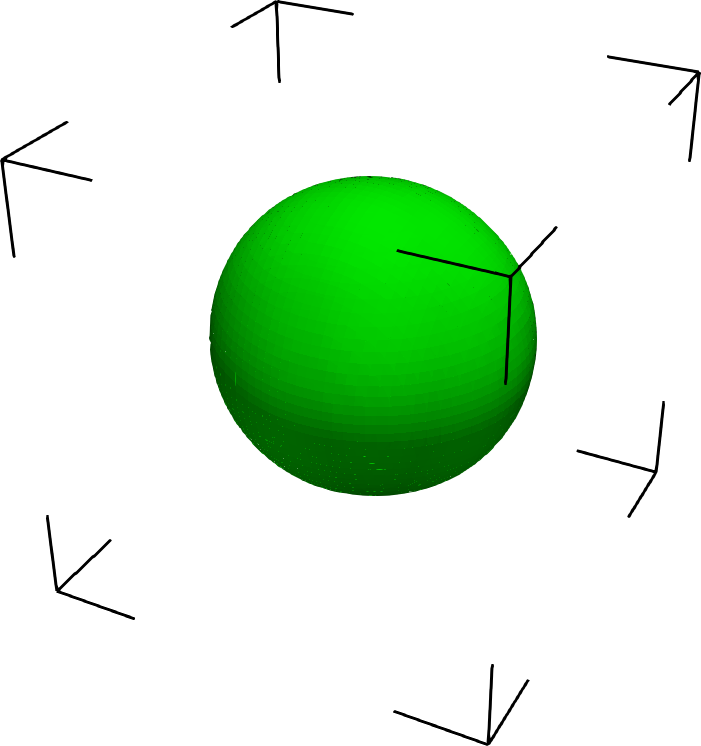} &
\includegraphics[width=0.17\textwidth]{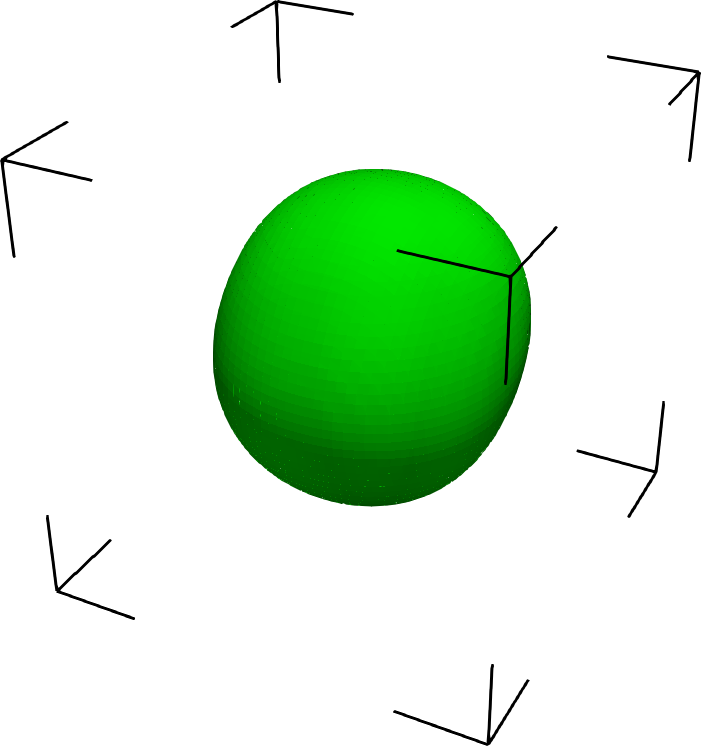} &
\includegraphics[width=0.17\textwidth]{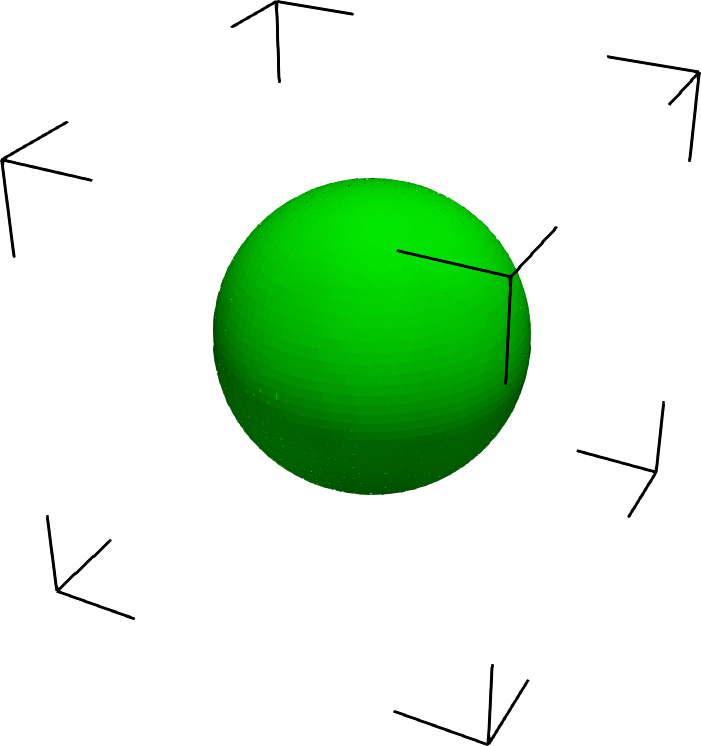} &
\includegraphics[width=0.17\textwidth]{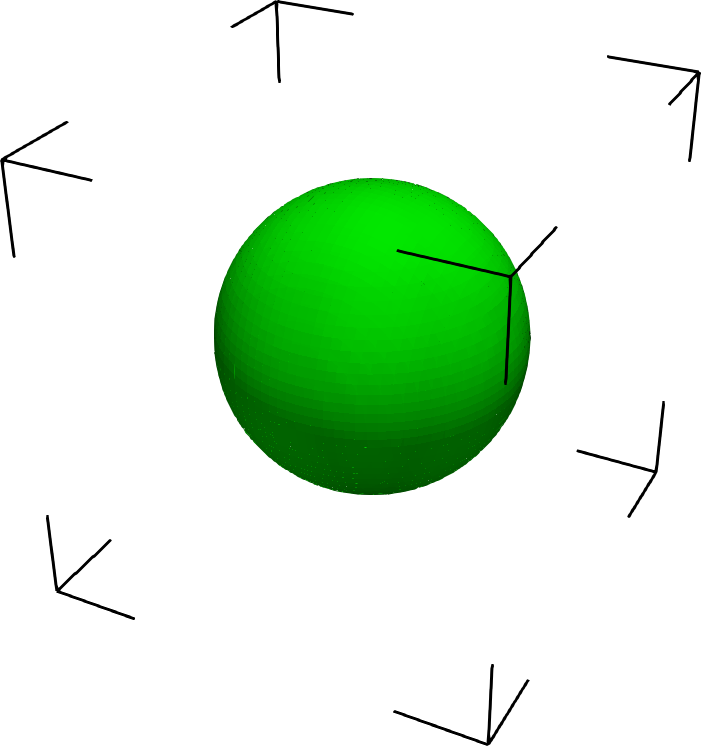} \\
\midrule
\multirow{2}{*}{\begin{sideways}{\footnotesize constant mobility $+$ FH potential $\hspace{-1.85cm}$}\end{sideways}} &
\includegraphics[width=0.17\textwidth]{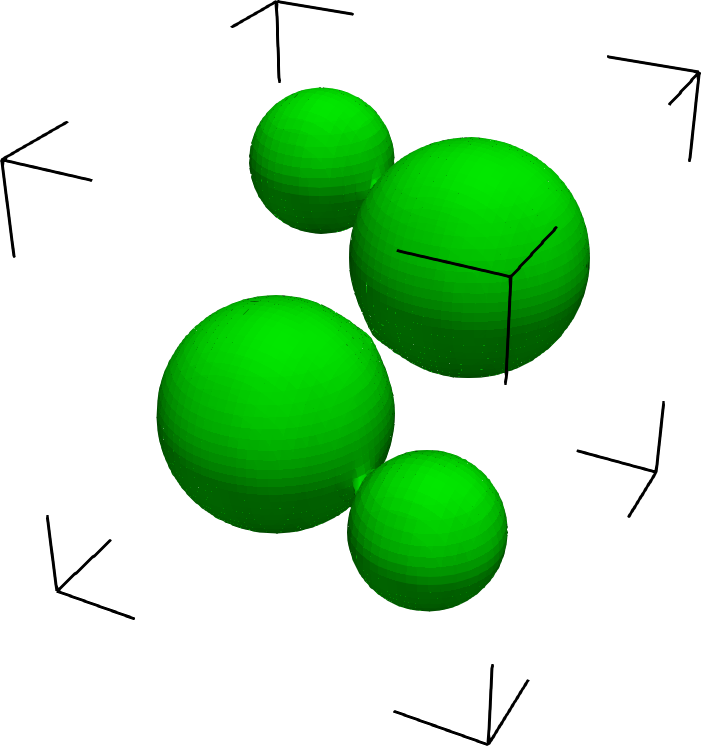} &
\includegraphics[width=0.17\textwidth]{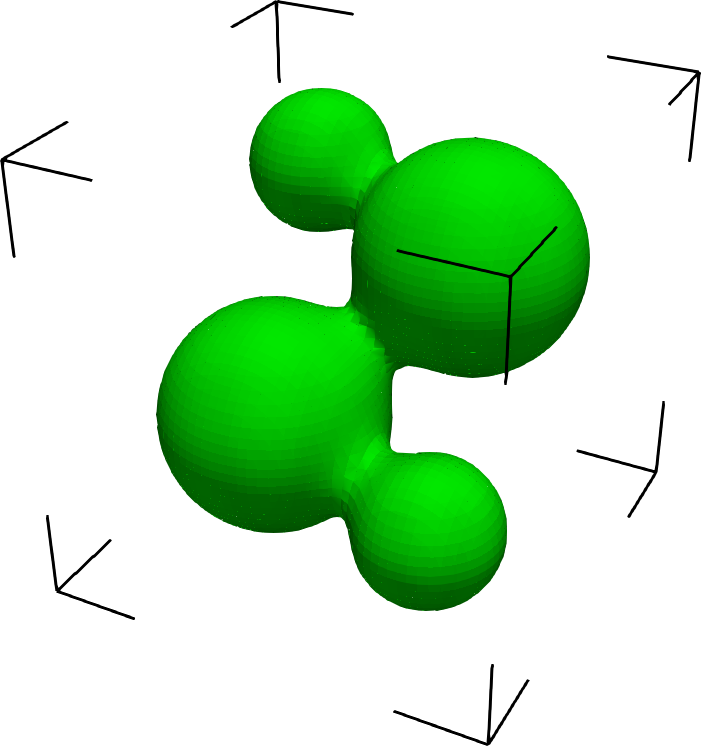} & 
\includegraphics[width=0.17\textwidth]{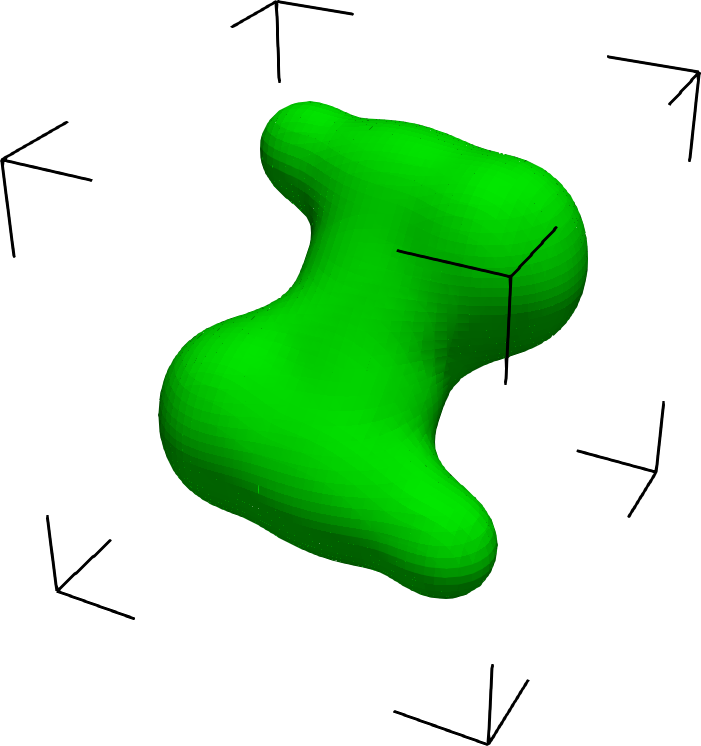} &
\includegraphics[width=0.17\textwidth]{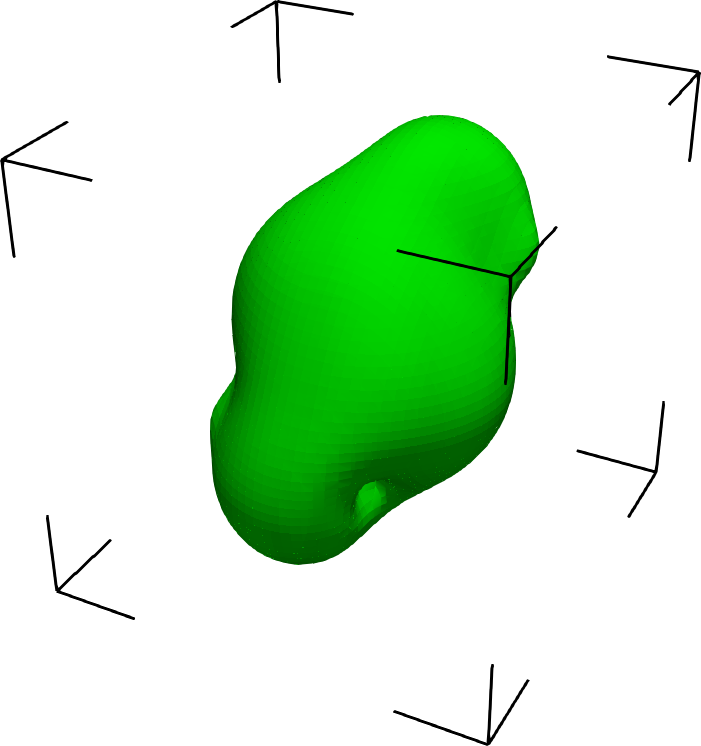} & 
\includegraphics[width=0.17\textwidth]{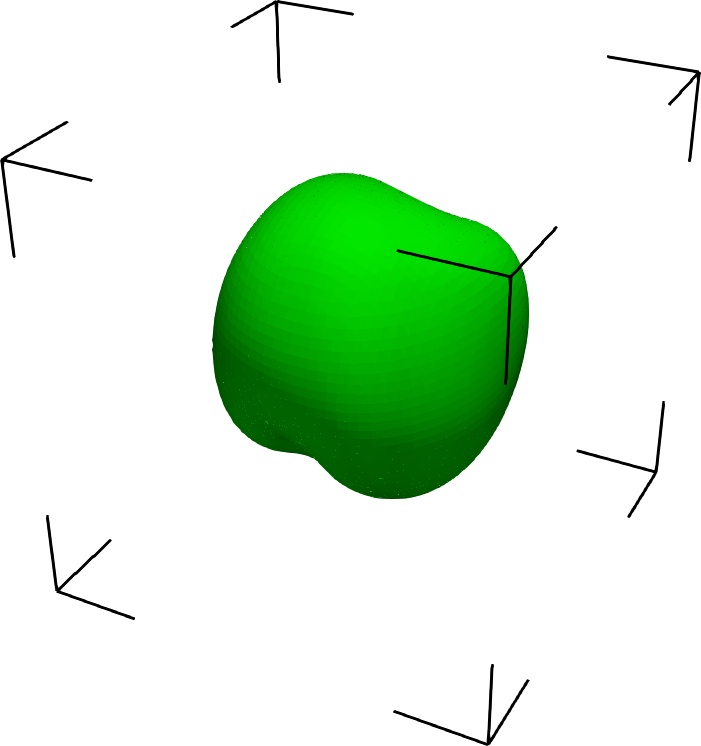} \\ &
\includegraphics[width=0.17\textwidth]{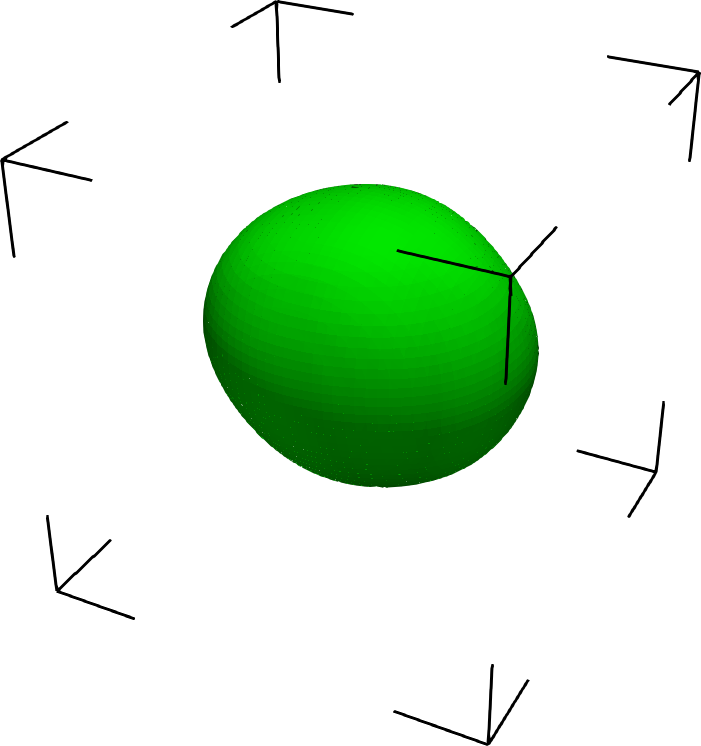} &
\includegraphics[width=0.17\textwidth]{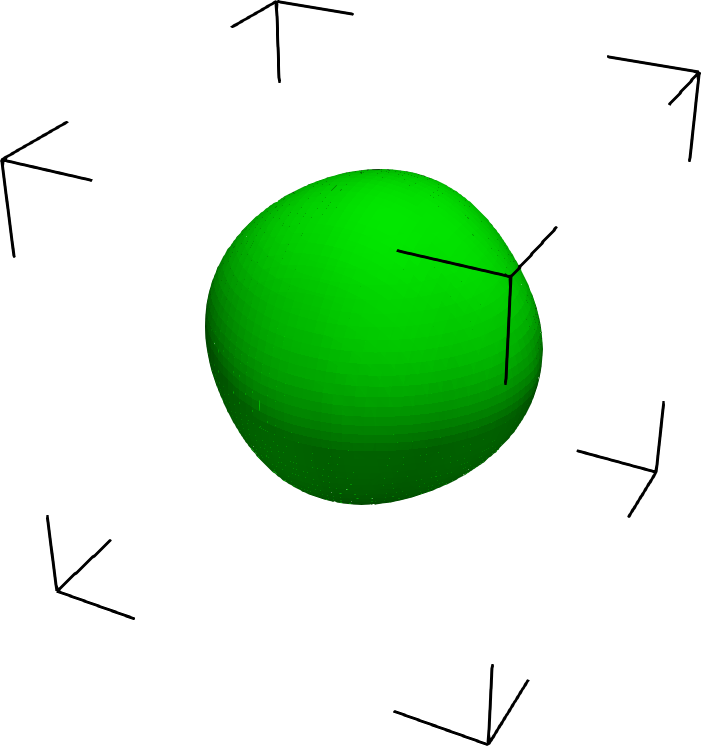} &
\includegraphics[width=0.17\textwidth]{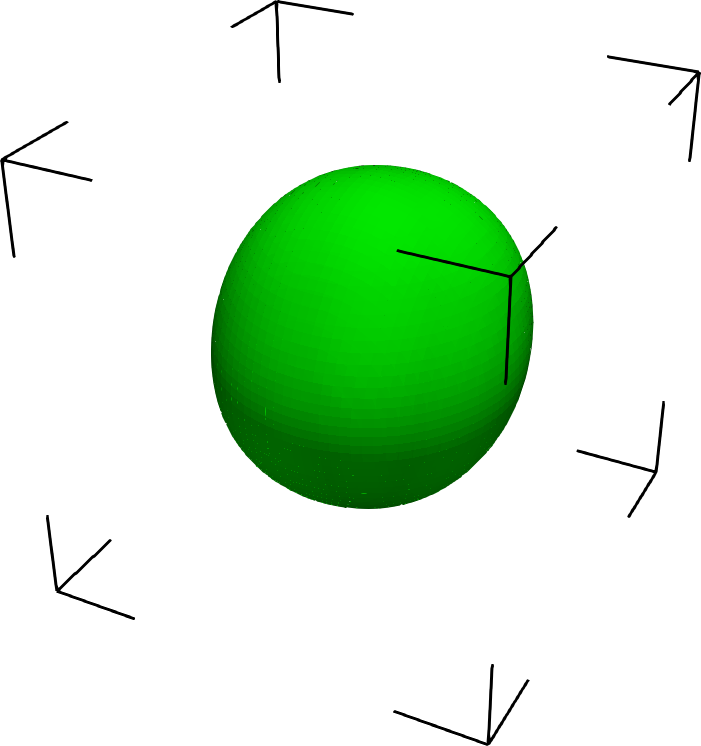} &
\includegraphics[width=0.17\textwidth]{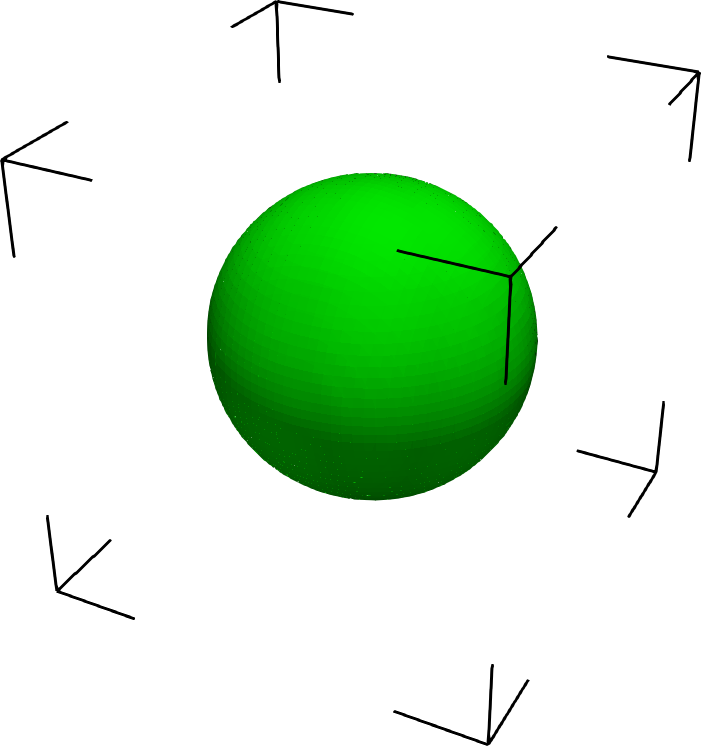} &
\includegraphics[width=0.17\textwidth]{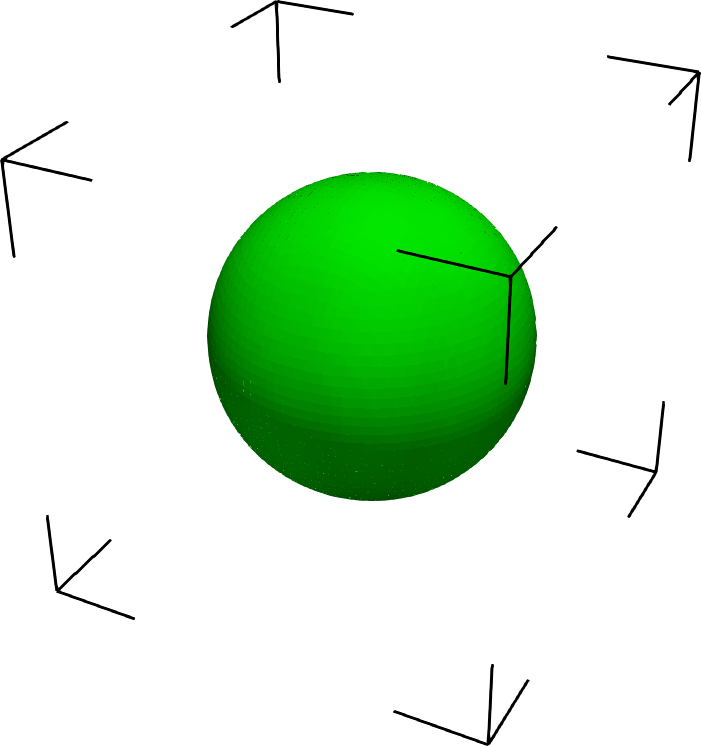} \\
\midrule
\multirow{2}{*}{\begin{sideways}{\footnotesize degenerate mobility $+$ GL potential $\hspace{-1.85cm}$}\end{sideways}} &
\includegraphics[width=0.17\textwidth]{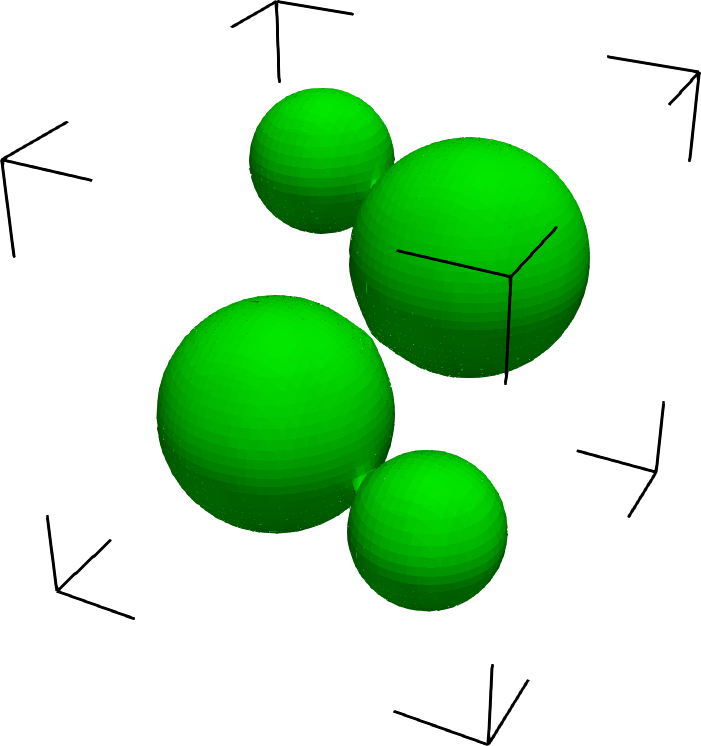} &
\includegraphics[width=0.17\textwidth]{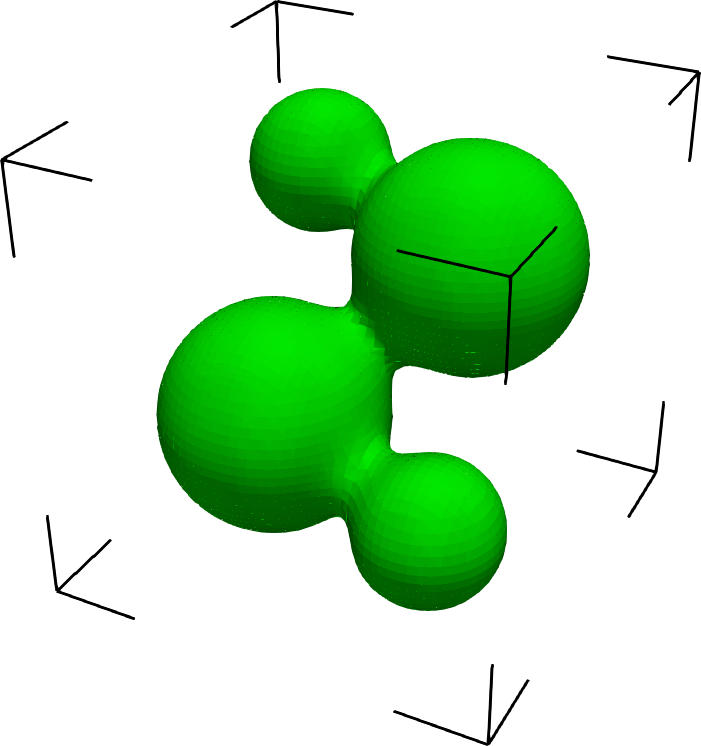} & 
\includegraphics[width=0.17\textwidth]{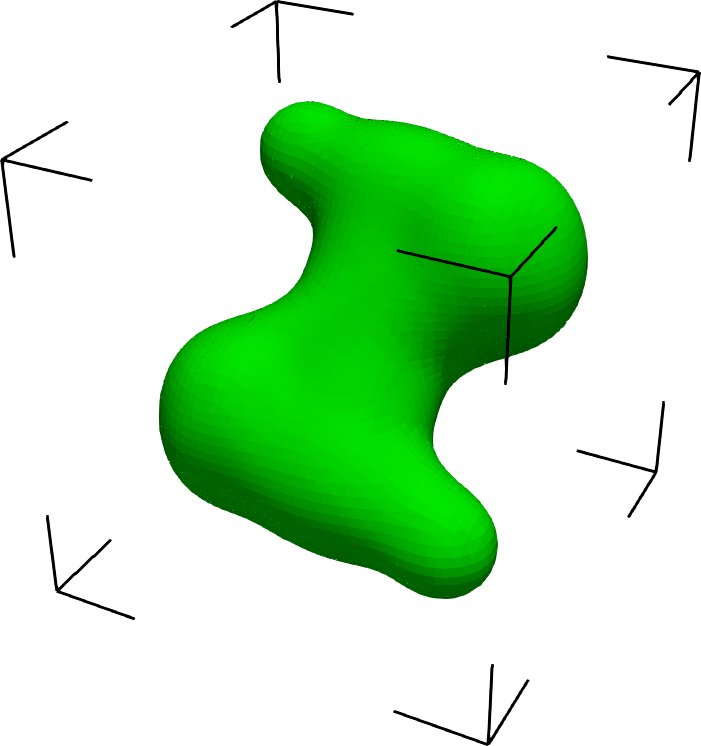} &
\includegraphics[width=0.17\textwidth]{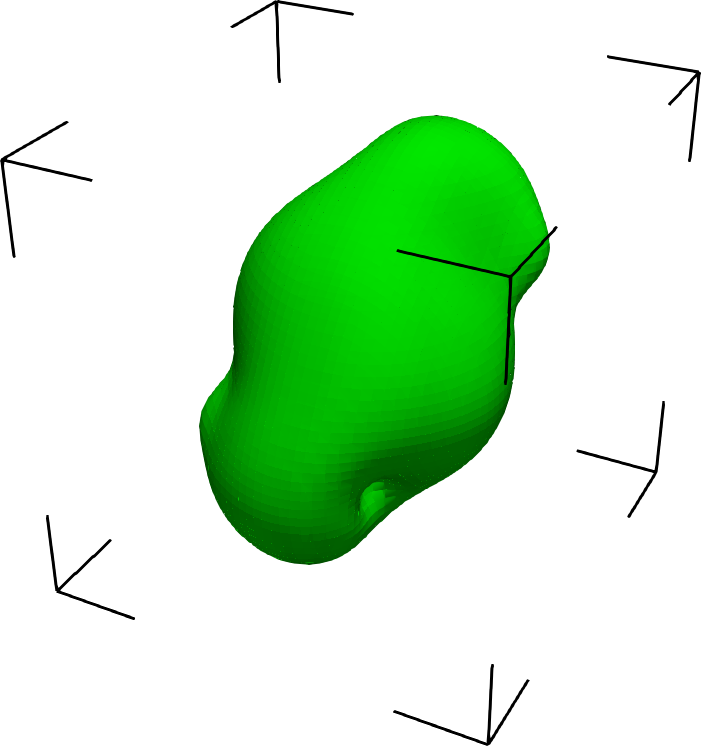} & 
\includegraphics[width=0.17\textwidth]{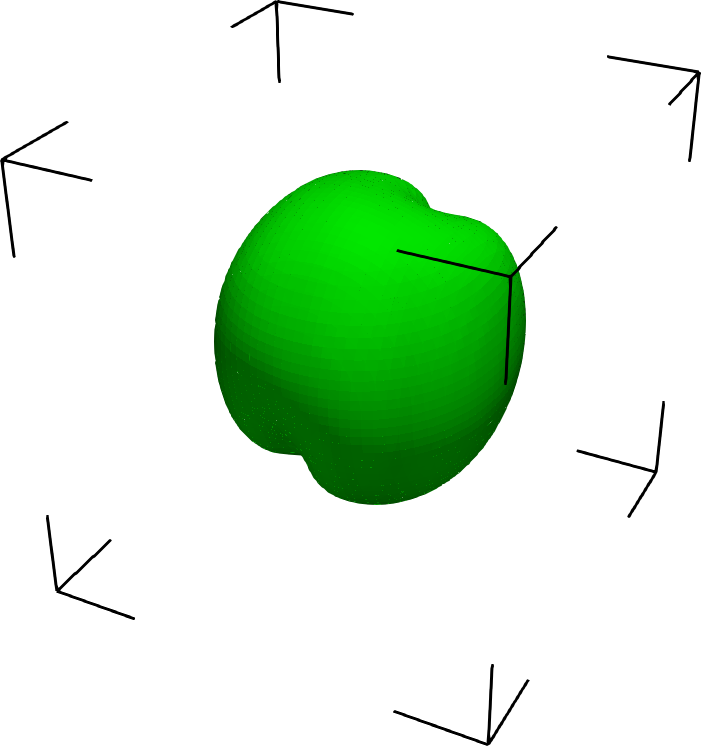} \\ &
\includegraphics[width=0.17\textwidth]{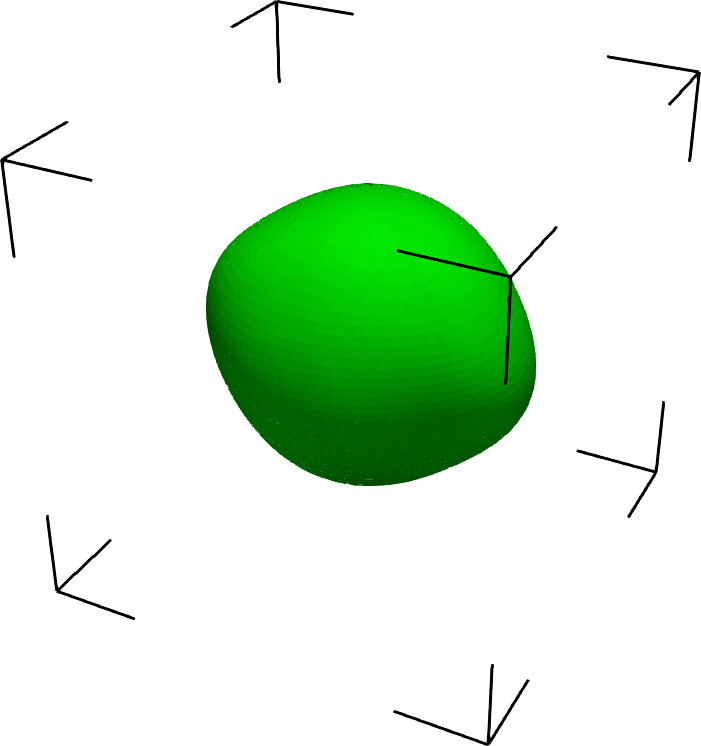} &
\includegraphics[width=0.17\textwidth]{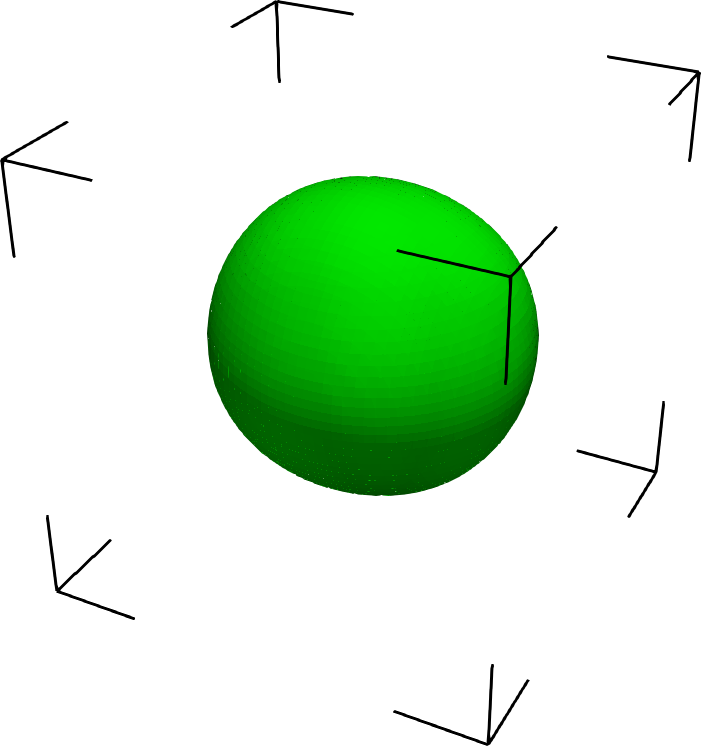} &
\includegraphics[width=0.17\textwidth]{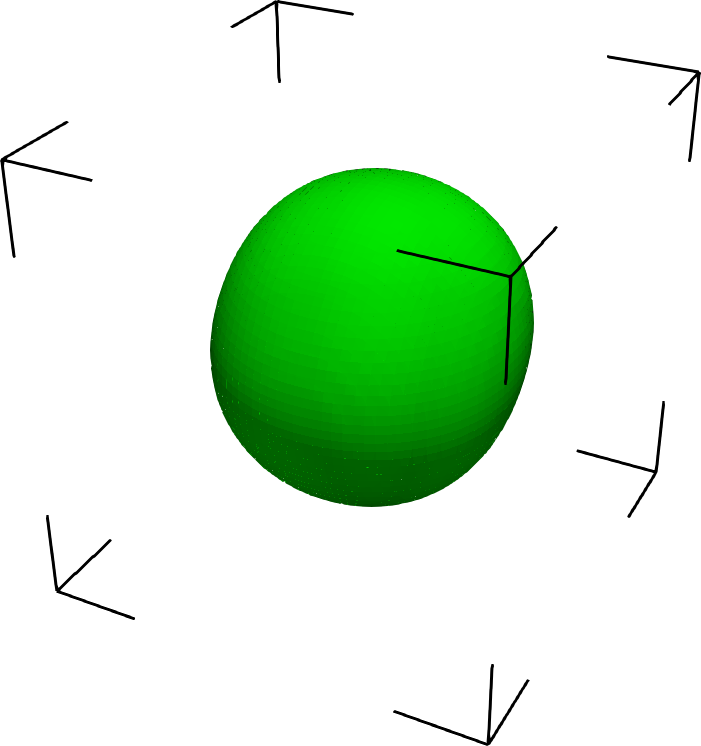} &
\includegraphics[width=0.17\textwidth]{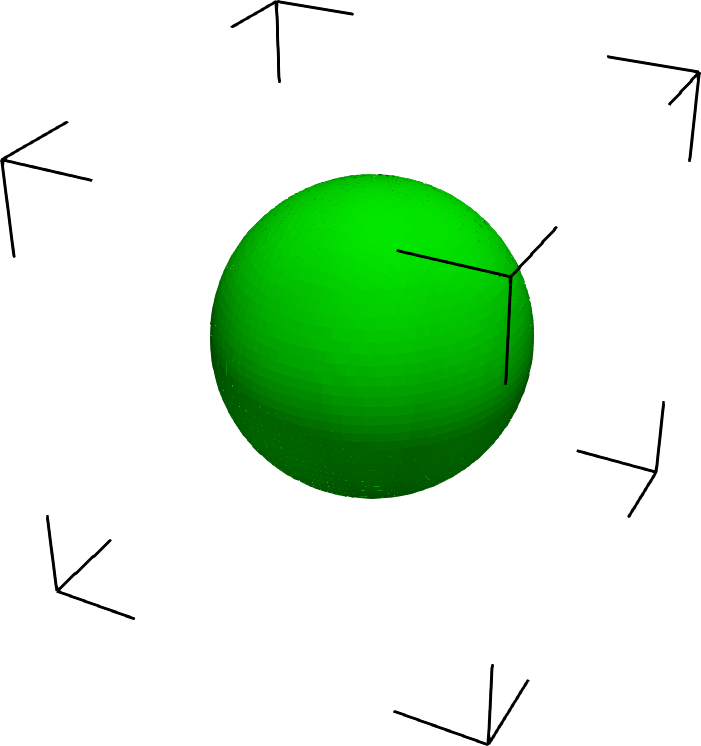} &
\includegraphics[width=0.17\textwidth]{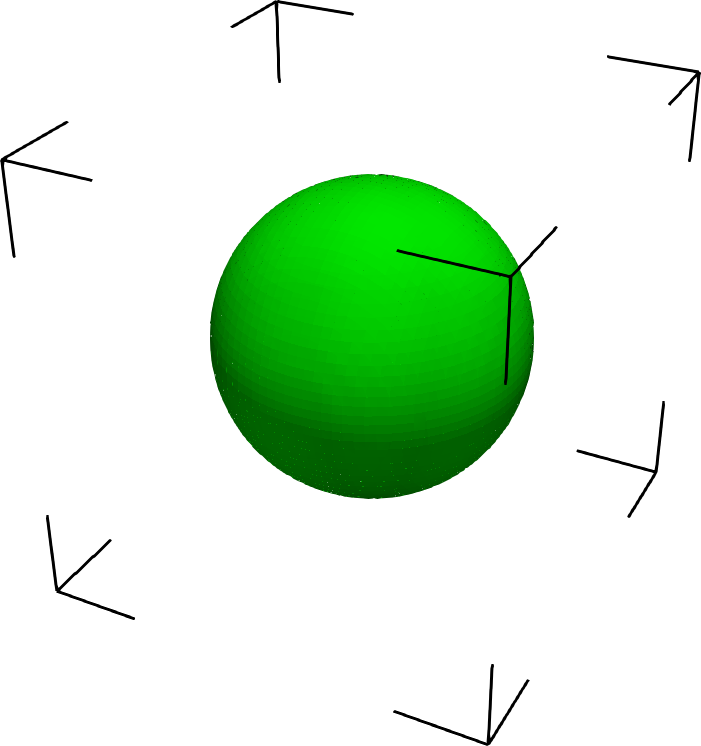} \\
\end{tabularx}
\caption{3D views of the evolution of the order parameter field. Selected snapshots at time steps $1$, $3$, $11$, $23$, $39$, $56$, $72$, $90$, $256$, and $512$. The dynamics are visually similar in these scenarios. However, there are visible differences in certain 2D profiles, see Figure~\ref{fig:numerical_experiments:drop_c2}.}
\label{fig:numerical_experiments:drop_c1}
\end{figure}
\begin{figure}[htbp]
\centering
\begin{tabularx}{0.975\linewidth}{@{~~}c@{~}c@{~}c@{~}c@{}}
\includegraphics[width=0.235\textwidth]{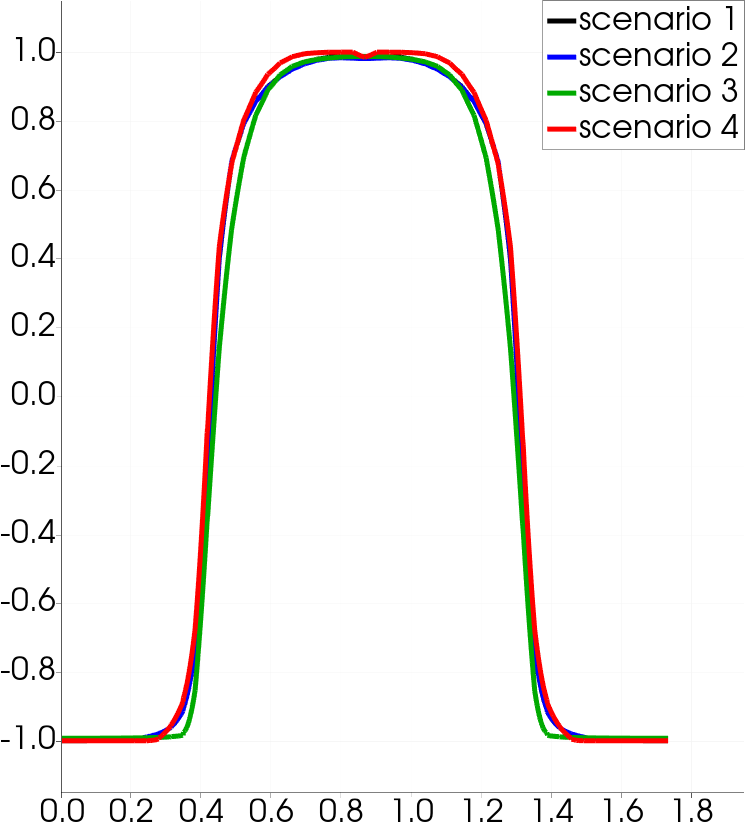} & 
\includegraphics[width=0.235\textwidth]{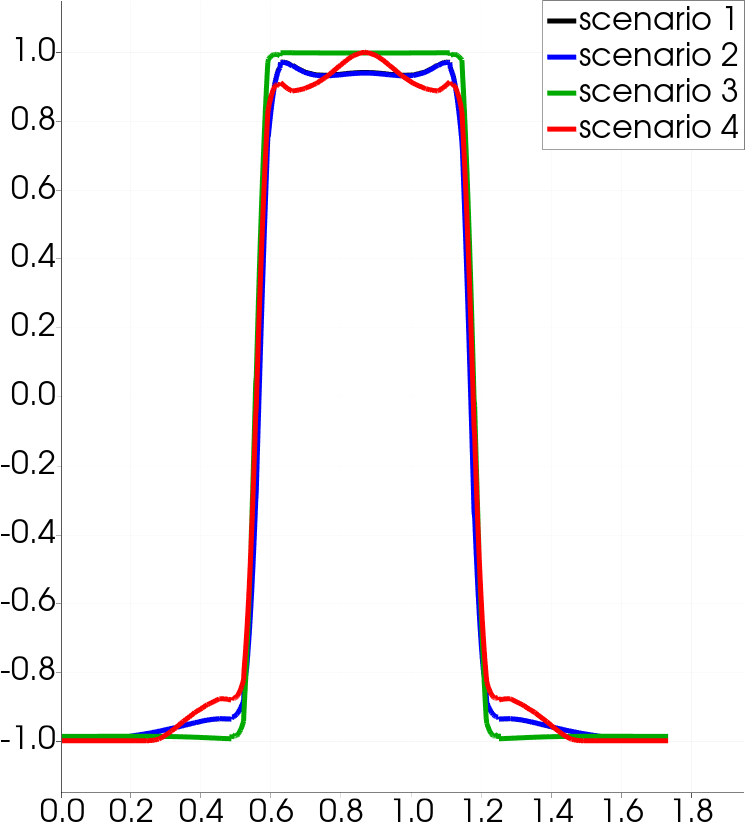} &
\includegraphics[width=0.235\textwidth]{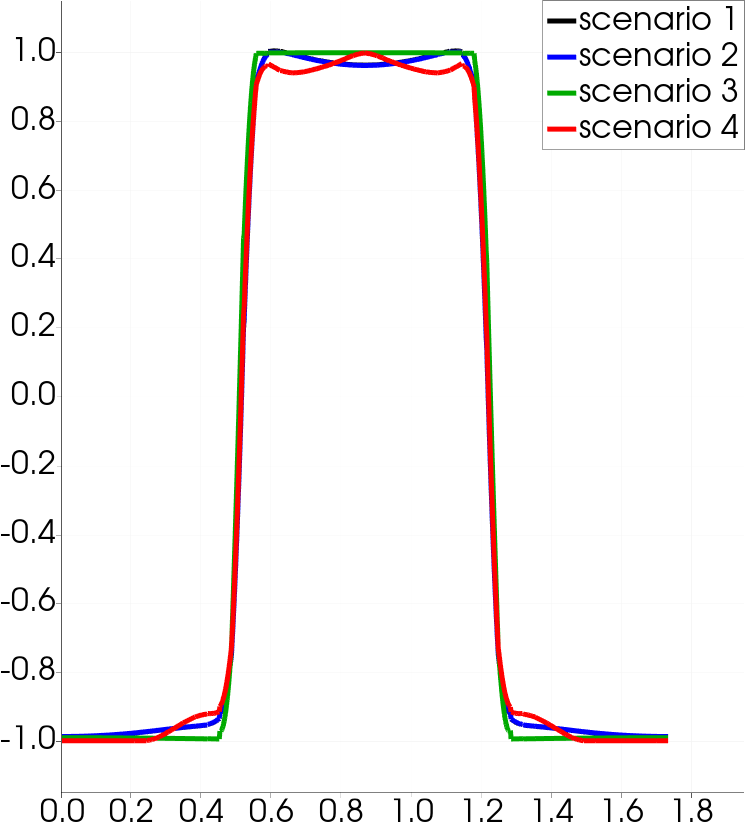} &
\includegraphics[width=0.235\textwidth]{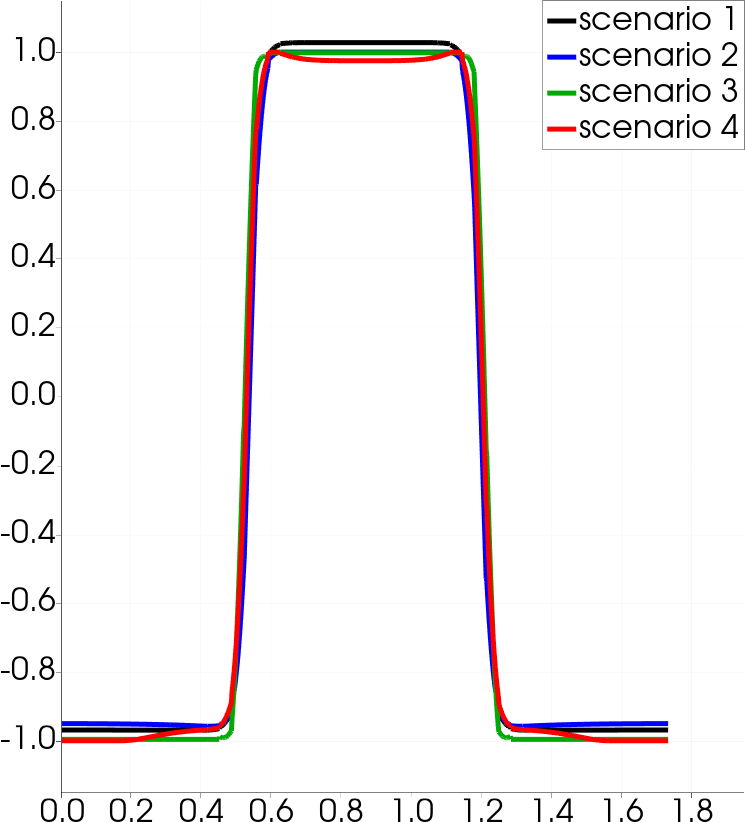} \\
\end{tabularx}
\caption{Plots of order parameter extracted along the line $\{(x,y,z)\in \Omega: x = y = z\}$. Selected snapshots at time steps $23$, $56$, $90$, and $512$. Scenario 1: constant mobility with GL polynomial potential and do not apply any limiter. The rest scenarios apply limiters. Scenario 2: constant mobility with GL polynomial potential. Scenario 3: constant mobility with FH logarithmic potential. Scenario 4: degenerate mobility with GL polynomial potential.}
\label{fig:numerical_experiments:drop_c2}
\end{figure}
\begin{figure}[htbp]
\centering
\begin{tabularx}{\linewidth}{@{}c@{~}c@{~}c@{}}
\includegraphics[width=0.35\textwidth]{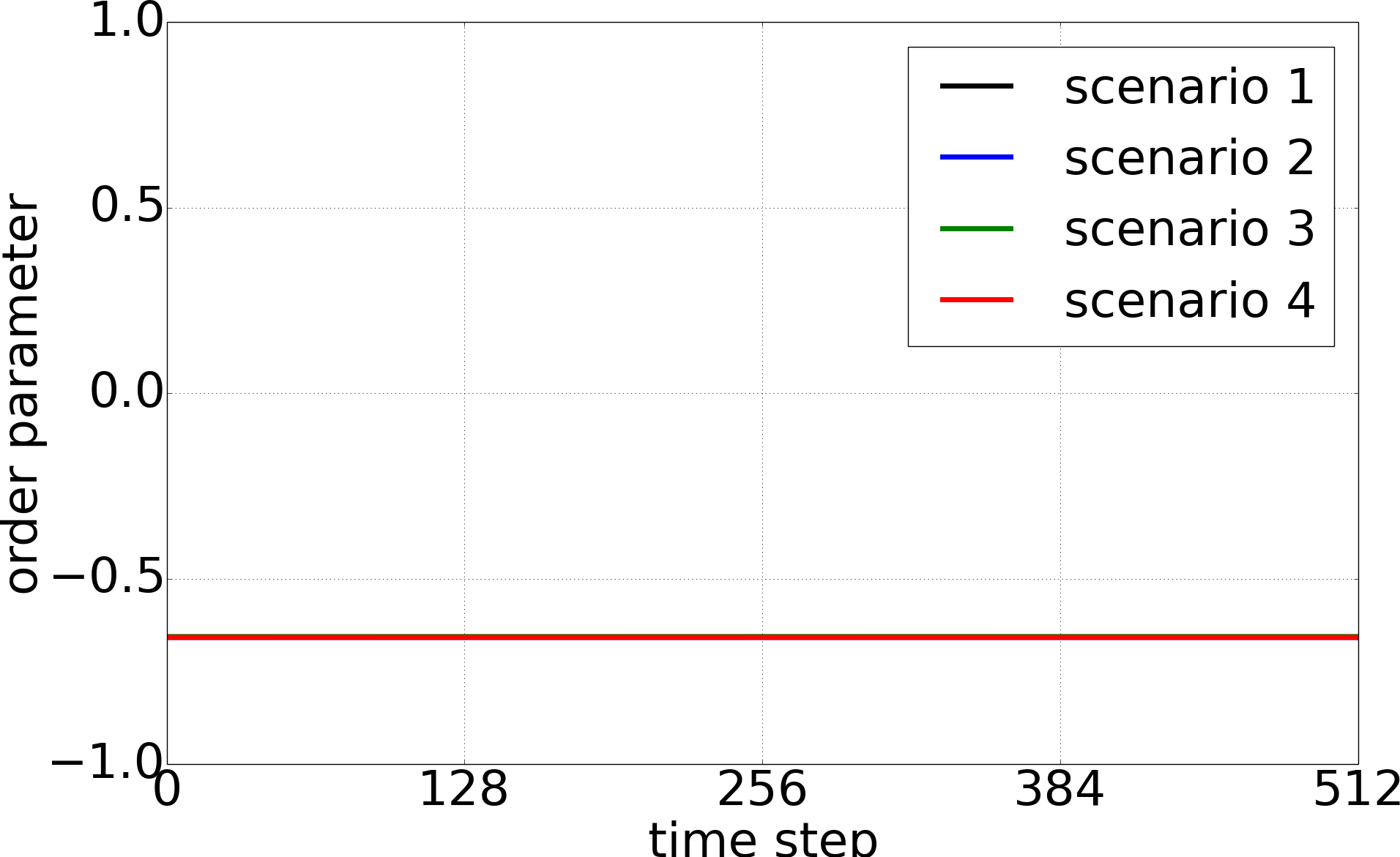} &
\includegraphics[width=0.33\textwidth]{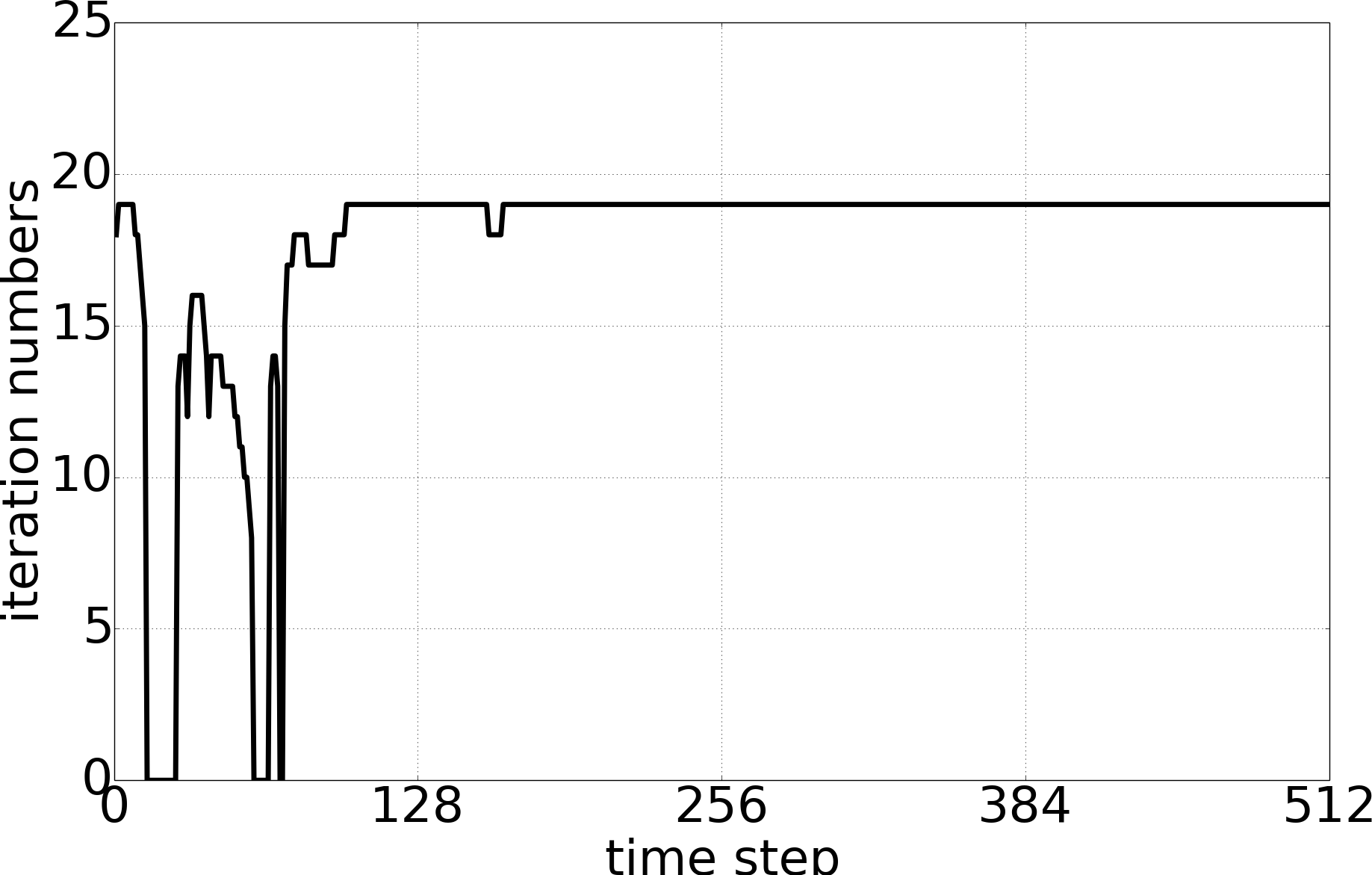} &
\includegraphics[width=0.33\textwidth]{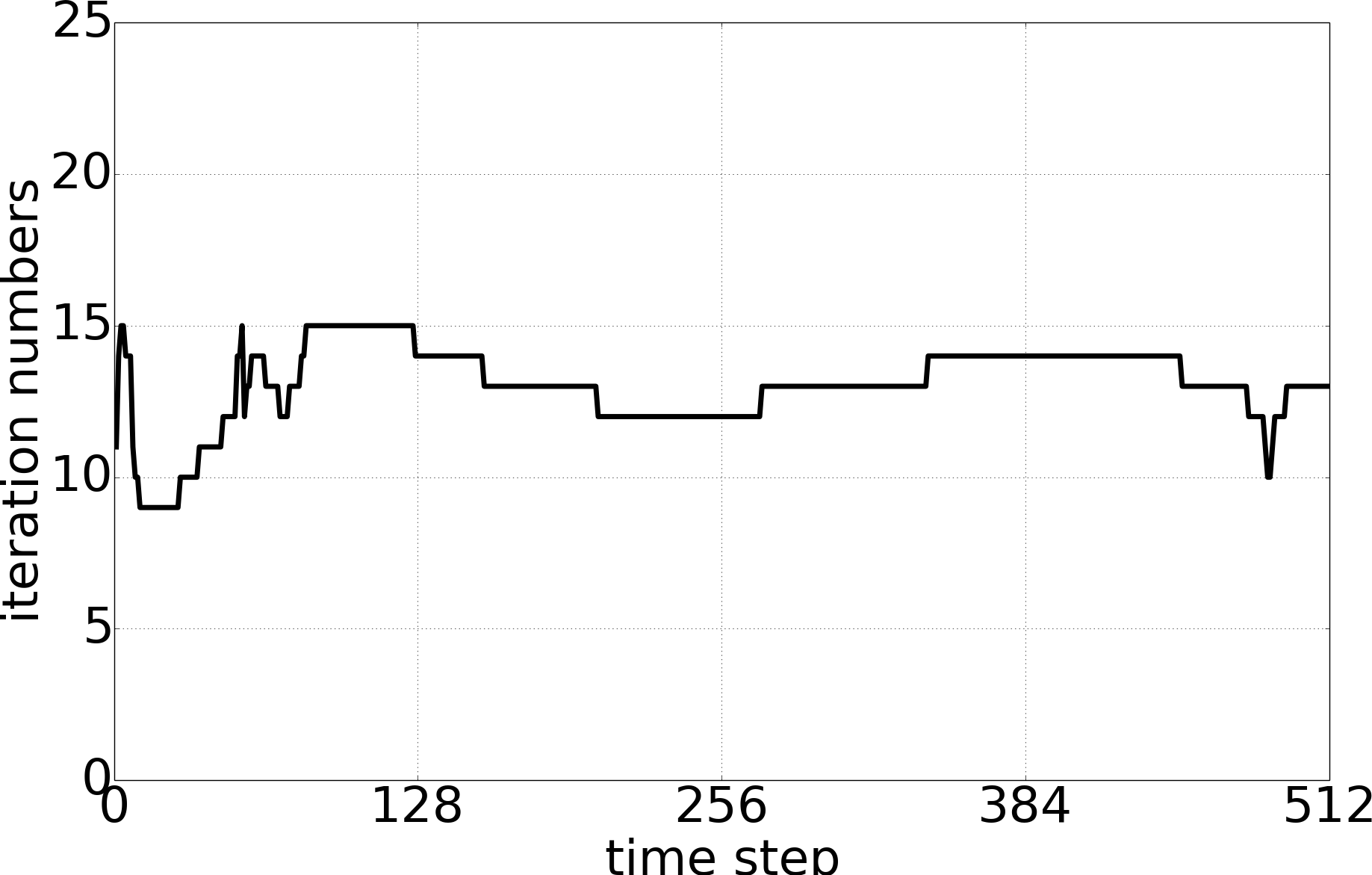} \\
\end{tabularx}
\caption{Left: the average of order parameter at each time step. Middle and right: the number of Douglas--Rachford iterations for scenario 2 and 4 at each time step. Scenario 1: constant mobility with GL polynomial potential and do not apply any limiter. The rest scenarios apply limiters. Scenario 2: constant mobility with GL polynomial potential. Scenario 3: constant mobility with FH logarithmic potential. Scenario 4: degenerate mobility with GL polynomial potential.}
\label{fig:numerical_experiments:drop_iter_and_mass}
\end{figure}

\section{Conclusion}
\label{sec:remark}
In this paper, we have analyzed the asymptotic linear convergence rate for using Douglas--Rachford splitting methods of a simple nonsmooth convex minimization, which forms a high order accurate cell average limiter.
We obtain an explicit dependence of the convergence rate on the parameters, which gives a principle of parameter selection for accelerating the asymptotic convergence rate.
Our optimization scheme is efficient and our two-stage limiting strategy is well-suited for high order accurate DG schemes for large-scale simulations.


\bibliographystyle{siamplain}
\bibliography{references}

\begin{thebibliography}{10}

\bibitem{bailo2023unconditional}
{\sc R.~Bailo, J.~Carrillo, S.~Kalliadasis, and S.~Perez}, {\em {Unconditional
  bound-preserving and energy-dissipating finite-volume schemes for the
  Cahn-Hilliard equation}}, Communications in Computational Physics,  (2023).

\bibitem{barrett1999finite}
{\sc J.~W. Barrett, J.~F. Blowey, and H.~Garcke}, {\em Finite element
  approximation of the {C}ahn--{H}illiard equation with degenerate mobility},
  SIAM Journal on Numerical Analysis, 37 (1999), pp.~286--318.

\bibitem{beck2017first}
{\sc A.~Beck}, {\em First-Order Methods in Optimization}, SIAM, 2017.

\bibitem{carlson2011dissipation}
{\sc A.~Carlson, M.~Do-Quang, and G.~Amberg}, {\em Dissipation in rapid dynamic
  wetting}, Journal of Fluid Mechanics, 682 (2011), pp.~213--240.

\bibitem{chambolle2016introduction}
{\sc A.~Chambolle and T.~Pock}, {\em An introduction to continuous optimization
  for imaging}, Acta Numerica, 25 (2016), pp.~161--319.

\bibitem{CHEN2019100031}
{\sc W.~Chen, C.~Wang, X.~Wang, and S.~M. Wise}, {\em Positivity-preserving,
  energy stable numerical schemes for the {C}ahn--{H}illiard equation with
  logarithmic potential}, Journal of Computational Physics: X, 3 (2019),
  p.~100031.

\bibitem{cheng2022new}
{\sc Q.~Cheng and J.~Shen}, {\em {A new Lagrange multiplier approach for
  constructing structure preserving schemes, I. Positivity preserving}},
  Computer Methods in Applied Mechanics and Engineering, 391 (2022), p.~114585.

\bibitem{cheng2022new-2}
{\sc Q.~Cheng and J.~Shen}, {\em {A new Lagrange multiplier approach for
  constructing structure preserving schemes, II. Bound preserving}}, SIAM
  Journal on Numerical Analysis, 60 (2022), pp.~970--998.

\bibitem{demanet2016eventual}
{\sc L.~Demanet and X.~Zhang}, {\em Eventual linear convergence of the
  {D}ouglas--{R}achford iteration for basis pursuit}, Mathematics of
  Computation, 85 (2016), pp.~209--238.

\bibitem{du2021maximum}
{\sc Q.~Du, L.~Ju, X.~Li, and Z.~Qiao}, {\em Maximum bound principles for a
  class of semilinear parabolic equations and exponential time-differencing
  schemes}, SIAM Review, 63 (2021), pp.~317--359.

\bibitem{ern2022invariant}
{\sc A.~Ern and J.-L. Guermond}, {\em {Invariant-Domain-Preserving High-Order
  Time Stepping: I. Explicit Runge--Kutta Schemes}}, SIAM Journal on Scientific
  Computing, 44 (2022), pp.~A3366--A3392.

\bibitem{fan2022positivity}
{\sc C.~Fan, X.~Zhang, and J.~Qiu}, {\em {Positivity-preserving high order
  finite difference WENO schemes for compressible Navier-Stokes equations}},
  Journal of Computational Physics, 467 (2022), p.~111446.

\bibitem{fortin2000augmented}
{\sc M.~Fortin and R.~Glowinski}, {\em Augmented Lagrangian Methods:
  Applications to the Numerical Solution of Boundary-value Problems}, Elsevier,
  2000.

\bibitem{frank2018finite}
{\sc F.~Frank, C.~Liu, F.~Alpak, and B.~Riviere}, {\em A finite
  volume/discontinuous {G}alerkin method for the advective {C}ahn--{H}illiard
  equation with degenerate mobility on porous domains stemming from micro-{CT}
  imaging}, Computational Geosciences, 22 (2018), pp.~543--563.

\bibitem{frank2018direct}
{\sc F.~Frank, C.~Liu, F.~O. Alpak, S.~Berg, and B.~Riviere}, {\em Direct
  numerical simulation of flow on pore-scale images using the phase-field
  method}, SPE Journal, 23 (2018), pp.~1833--1850.

\bibitem{frank2018energy}
{\sc F.~Frank, C.~Liu, A.~Scanziani, F.~O. Alpak, and B.~Riviere}, {\em An
  energy-based equilibrium contact angle boundary condition on jagged surfaces
  for phase-field methods}, Journal of Colloid and Interface Science, 523
  (2018), pp.~282--291.

\bibitem{frank2020bound}
{\sc F.~Frank, A.~Rupp, and D.~Kuzmin}, {\em {Bound-preserving flux limiting
  schemes for DG discretizations of conservation laws with applications to the
  Cahn--Hilliard equation}}, Computer Methods in Applied Mechanics and
  Engineering, 359 (2020), p.~112665.

\bibitem{girault2005discontinuous}
{\sc V.~Girault, B.~Riviere, and M.~Wheeler}, {\em A discontinuous {G}alerkin
  method with nonoverlapping domain decomposition for the {S}tokes and
  {N}avier-{S}tokes problems}, Mathematics of Computation, 74 (2005),
  pp.~53--84.

\bibitem{glowinski2003finite}
{\sc R.~Glowinski}, {\em Finite element methods for incompressible viscous
  flow}, Handbook of Numerical Analysis, 9 (2003), pp.~3--1176.

\bibitem{goldstein2009split}
{\sc T.~Goldstein and S.~Osher}, {\em {The split Bregman method for
  L1-regularized problems}}, SIAM Journal on Imaging Sciences, 2 (2009),
  pp.~323--343.

\bibitem{guermond2006overview}
{\sc J.-L. Guermond, P.~Minev, and J.~Shen}, {\em An overview of projection
  methods for incompressible flows}, Computer Methods in Applied Mechanics and
  Engineering, 195 (2006), pp.~6011--6045.

\bibitem{guermond2019invariant}
{\sc J.-L. Guermond, B.~Popov, and I.~Tomas}, {\em Invariant domain preserving
  discretization-independent schemes and convex limiting for hyperbolic
  systems}, Computer Methods in Applied Mechanics and Engineering, 347 (2019),
  pp.~143--175.

\bibitem{huang2022bound}
{\sc F.~Huang, J.~Shen, and K.~Wu}, {\em Bound/positivity preserving and
  unconditionally stable schemes for a class of fourth order nonlinear
  equations}, Journal of Computational Physics, 460 (2022), p.~111177.

\bibitem{knyazev2007majorization}
{\sc A.~V. Knyazev and M.~E. Argentati}, {\em Majorization for changes in
  angles between subspaces, {R}itz values, and graph {L}aplacian spectra}, SIAM
  Journal on Matrix Analysis and Applications, 29 (2007), pp.~15--32.

\bibitem{kuzmin2009explicit}
{\sc D.~Kuzmin}, {\em {Explicit and implicit FEM-FCT algorithms with flux
  linearization}}, Journal of Computational Physics, 228 (2009),
  pp.~2517--2534.

\bibitem{lions1979splitting}
{\sc P.-L. Lions and B.~Mercier}, {\em Splitting algorithms for the sum of two
  nonlinear operators}, SIAM Journal on Numerical Analysis, 16 (1979),
  pp.~964--979.

\bibitem{liu2019interior}
{\sc C.~Liu, F.~Frank, F.~O. Alpak, and B.~Riviere}, {\em An interior penalty
  discontinuous {G}alerkin approach for 3{D} incompressible {N}avier--{S}tokes
  equation for permeability estimation of porous media}, Journal of
  Computational Physics, 396 (2019), pp.~669--686.

\bibitem{liu2020efficient}
{\sc C.~Liu, F.~Frank, C.~Thiele, F.~O. Alpak, S.~Berg, W.~Chapman, and
  B.~Riviere}, {\em An efficient numerical algorithm for solving viscosity
  contrast {C}ahn--{H}illiard--{N}avier--{S}tokes system in porous media},
  Journal of Computational Physics, 400 (2020), p.~108948.

\bibitem{liu2022convergence}
{\sc C.~Liu, R.~Masri, and B.~Riviere}, {\em Convergence of a decoupled
  splitting scheme for the {C}ahn--{H}illiard--{N}avier--{S}tokes system}, SIAM
  Journal on Numerical Analysis (to appear),  (2023).
\newblock arXiv:2210.05625.

\bibitem{liu2022pressure}
{\sc C.~Liu, D.~Ray, C.~Thiele, L.~Lin, and B.~Riviere}, {\em A
  pressure-correction and bound-preserving discretization of the phase-field
  method for variable density two-phase flows}, Journal of Computational
  Physics, 449 (2022), p.~110769.

\bibitem{liu2023positivity}
{\sc C.~Liu and X.~Zhang}, {\em A positivity-preserving implicit-explicit
  scheme with high order polynomial basis for compressible {N}avier--{S}tokes
  equations}, arXiv:2305.05769,  (2023).

\bibitem{masri2022discontinuous}
{\sc R.~Masri, C.~Liu, and B.~Riviere}, {\em A discontinuous {G}alerkin
  pressure correction scheme for the incompressible {N}avier--{S}tokes
  equations: {S}tability and convergence}, Mathematics of Computation, 91
  (2022), pp.~1625--1654.

\bibitem{masri2023improved}
{\sc R.~Masri, C.~Liu, and B.~Riviere}, {\em Improved a priori error estimates
  for a discontinuous {G}alerkin pressure correction scheme for the
  {N}avier--{S}tokes equations}, Numerical Methods for Partial Differential
  Equations, 39 (2023), pp.~3108--3144.

\bibitem{nesterov2013gradient}
{\sc Y.~Nesterov}, {\em Gradient methods for minimizing composite functions},
  Mathematical Programming, 140 (2013), pp.~125--161.

\bibitem{qin2018implicit}
{\sc T.~Qin and C.-W. Shu}, {\em {Implicit positivity-preserving high-order
  discontinuous Galerkin methods for conservation laws}}, SIAM Journal on
  Scientific Computing, 40 (2018), pp.~A81--A107.

\bibitem{riviere2008discontinuous}
{\sc B.~Riviere}, {\em Discontinuous {G}alerkin methods for solving elliptic
  and parabolic equations: theory and implementation}, SIAM, 2008.

\bibitem{shen2012modeling}
{\sc J.~Shen}, {\em Modeling and numerical approximation of two-phase
  incompressible flows by a phase-field approach}, in Multiscale Modeling and
  Analysis for Materials Aimulation, World Scientific, 2012, pp.~147--195.

\bibitem{shen2015decoupled}
{\sc J.~Shen and X.~Yang}, {\em Decoupled, energy stable schemes for
  phase-field models of two-phase incompressible flows}, SIAM Journal on
  Numerical Analysis, 53 (2015), pp.~279--296.

\bibitem{srinivasan2018positivity}
{\sc S.~Srinivasan, J.~Poggie, and X.~Zhang}, {\em {A positivity-preserving
  high order discontinuous Galerkin scheme for convection--diffusion
  equations}}, Journal of Computational Physics, 366 (2018), pp.~120--143.

\bibitem{sun2018discontinuous}
{\sc Z.~Sun, J.~A. Carrillo, and C.-W. Shu}, {\em {A discontinuous Galerkin
  method for nonlinear parabolic equations and gradient flow problems with
  interaction potentials}}, Journal of Computational Physics, 352 (2018),
  pp.~76--104.

\bibitem{tierra2015numerical}
{\sc G.~Tierra and F.~Guill{\'e}n-Gonz{\'a}lez}, {\em Numerical methods for
  solving the {C}ahn--{H}illiard equation and its applicability to related
  energy-based models}, Archives of Computational Methods in Engineering, 22
  (2015), pp.~269--289.

\bibitem{xu2014parametrized}
{\sc Z.~Xu}, {\em Parametrized maximum principle preserving flux limiters for
  high order schemes solving hyperbolic conservation laws: one-dimensional
  scalar problem}, Mathematics of Computation, 83 (2014), pp.~2213--2238.

\bibitem{zhang2017positivity}
{\sc X.~Zhang}, {\em On positivity-preserving high order discontinuous
  {G}alerkin schemes for compressible {N}avier--{S}tokes equations}, Journal of
  Computational Physics, 328 (2017), pp.~301--343.

\bibitem{zhang2010maximum}
{\sc X.~Zhang and C.-W. Shu}, {\em On maximum-principle-satisfying high order
  schemes for scalar conservation laws}, Journal of Computational Physics, 229
  (2010), pp.~3091--3120.

\bibitem{zhang2010positivity}
{\sc X.~Zhang and C.-W. Shu}, {\em On positivity-preserving high order
  discontinuous {G}alerkin schemes for compressible {E}uler equations on
  rectangular meshes}, Journal of Computational Physics, 229 (2010),
  pp.~8918--8934.

\end{thebibliography}
\end{document}